\documentclass[a4paper,12pt]{preprint}
\usepackage[margin=1.3in]{geometry}  
\usepackage[full]{textcomp}

\usepackage{mathalfa}
\usepackage{ulem}
\usepackage{microtype}
\usepackage{url}
\usepackage[shortlabels]{enumitem}

\usepackage{booktabs}

\usepackage{enumitem}
\usepackage{amsmath}
\usepackage{amsfonts} 
\usepackage{amssymb}             

\usepackage{comment}

\usepackage{mathrsfs}
\usepackage{booktabs}
\usepackage{mathtools}

\usepackage{tikz}
\usepackage{amsthm}                
\usepackage{mhequ}
\usepackage{hyperref}

\hypersetup{pdfstartview=}

\addtocontents{toc}{\protect\hypersetup{hidelinks}}

\DeclareSymbolFont{sfoperators}{OT1}{ptm}{m}{n}
\DeclareSymbolFontAlphabet{\mathsf}{sfoperators}

\makeatletter
\def\operator@font{\mathgroup\symsfoperators}
\makeatother

\numberwithin{equation}{section}

\newtheorem{thm}{Theorem}[section]

\newtheorem{lem}[thm]{Lemma}
\newtheorem{prop}[thm]{Proposition}

\theoremstyle{remark}

\newtheorem{rmk}[thm]{Remark}

\makeatletter
\def\th@newremark{\th@remark\thm@headfont{\bfseries}}

\def\bdiamond{\mathop{\mathpalette\bdi@mond\relax}}
\newcommand\bdi@mond[2]{%
	\vcenter{\hbox{\m@th
			\scalebox{\ifx#1\displaystyle 2.6\else1.8\fi}{$#1\diamond$}%
	}}%
}

\def\bDiamond{\mathop{\mathpalette\bDi@mond\relax}}
\newcommand\bDi@mond[2]{%
	\vcenter{\hbox{\m@th
			\scalebox{\ifx#1\displaystyle 2.6\else1.2\fi}{$#1\Diamond$}%
	}}%
}
\makeatletter

\definecolor{darkgreen}{rgb}{0.1,0.7,0.1}
\definecolor{darkred}{rgb}{0.7,0.1,0.1}
\definecolor{darkblue}{rgb}{0,0,0.7}
\newcommand{\EE}{\mathbb{E}}     

\newcommand{\NN}{\mathbb{N}}
\newcommand{\PP}{\mathbb{P}}     

\newcommand{\RR}{\mathbb{R}}      

\newcommand{\ZZ}{\mathbb{Z}}      

\newcommand{\aA}{\mathcal{A}}

\newcommand{\cC}{\mathcal{C}}
\newcommand{\dD}{\mathcal{D}}

\newcommand{\lL}{\mathcal{L}}
\newcommand{\mM}{\mathcal{M}}

\newcommand{\oO}{\mathcal{O}}
\newcommand{\pP}{\mathcal{P}}
\newcommand{\qQ}{\mathcal{Q}}

\newcommand{\sS}{\mathcal{S}}

\newcommand{\vV}{\mathcal{V}}
\newcommand{\wW}{\mathcal{W}}
\newcommand{\xX}{\mathcal{X}}
\newcommand{\yY}{\mathcal{Y}}

\newcommand{\fC}{\mathfrak{C}}
\newcommand{\fD}{\mathfrak{D}}
\newcommand{\ff}{\mathfrak{f}}

\makeatletter 
\newcommand{\cov}{{\operator@font cov}}
\newcommand{\var}{{\operator@font var}}
\newcommand{\corr}{{\operator@font corr}}
\newcommand{\diam}{{\operator@font diam}}
\newcommand{\Av}{{\operator@font Av}}
\newcommand{\trig}{{\operator@font trig}}
\newcommand{\Enh}{{\operator@font Enh}}
\newcommand{\EEnh}{\overline {\operator@font Enh}}
\newcommand{\EC}{{\operator@font EC}}
\makeatother

\newcommand{\E}{\mathbf{E}}

\renewcommand{\r}{\mathbf{r}}
\renewcommand{\t}{\mathbf{t}}

\newcommand{\eps}{\varepsilon}
\newcommand{\md}{\mathrm{d}}
\renewcommand{\d}{\partial}

\newcommand{\dist}{\operatorname{dist}}

\newcommand{\bracket}[1]{\langle #1 \rangle}

\newcommand{\err}{\mathrm{err }}

\usepackage{amsmath}
\allowdisplaybreaks[4]

\title{Interface fluctuations for $1$D stochastic Allen-Cahn equation -- singular regime}
\author{Weijun Xu$^1$ \quad Shuhan Zhou$^2$}
\institute{Beijing International Center for Mathematical Research, Peking University, China. \email{weijunxu@bicmr.pku.edu.cn}
\and Peking University, China. \email{zhoushuhan@stu.pku.edu.cn
}
}

\begin{document}
\maketitle

\begin{abstract}
    We study interface fluctuations for the $1$D stochastic Allen-Cahn equation perturbed by half a spatial derivative of the spacetime white noise. This half derivative makes the solution distribution-valued, so that proper renormalization is needed to make sense of the solution. 

    We show that if the noise is sufficiently small, then an analogue of the classical results by \cite{Fun95,BBDMP98} holds in this singular regime. More precisely, for initial data close to the traveling wave solution of the deterministic equation, under proper long time scaling, the solution still stays close to the family of traveling waves, and the interface location moves according to an approximate diffusion process. There is one interesting difference between our singular regime and the classical situation: even if the solution and its approximate phase separation point are both well defined, the intended diffusion describing the movement of the canonical candidate of the phase point is not (even for fixed $\eps$). Two infinite quantities arise from the derivation of such an SDE, one due to singularity of the noise, and the other from renormalization. Magically, it turns out that they cancel out each other, thus making the derivation of the interface SDE valid in the $\eps \rightarrow 0$ limit. 
\end{abstract}

\setcounter{tocdepth}{2}
\tableofcontents

\section{Introduction}

In this article, we study the sharp interface limit problem (as $\eps \rightarrow 0$) for the following formal 1D Wick ordered stochastic Allen-Cahn equation:
\begin{equation}\label{e:main_eqn}
    \partial_t u_\eps =  \Delta u_\eps + u_\eps -  u_\eps^{\diamond (2n+1)}  + \eps^{\gamma} a_\eps \sqrt{\fD}\dot{W}\;.
\end{equation}
Here, $\gamma > \frac{1}{2}$ is a fixed parameter, $a_\eps = a(\sqrt{\eps} \cdot)$ for some smooth function $a$ with compact support in $(-1,1)$, $\dot{W}$ is the one-dimensional spacetime white noise (time derivative of the $\lL^2$ cylindrical Wiener process $W$), $\fD = \frac{|\d_x|}{2\pi}$ is the differential operator normalized in Fourier space (specified in \eqref{e:half_derivative} below), and $(\cdot)^{\diamond (2n+1)}$ is the $(2n+1)$-th Wick product induced by the Gaussian structure of the linearized equation with an additional mass $\mu>0$. Hence, the solution $u_\eps$ itself depends on the choice of $\mu>0$. The necessity of the Wick ordering comes from the fact that the additional $\sqrt{\fD}$ operation on the spacetime white noise makes the solution distribution valued, and hence proper renormalization is needed to make sense of the solution. The rigorous meaning of \eqref{e:main_eqn} will be given in Section~\ref{sec:renormalization} below. 

The choice of the cutoff on the interval of length $\oO(1/\sqrt{\eps})$ is to be consistent with the set up in \cite{Fun95}. Indeed, by a change of parameter, one can make the cutoff to be $a \big( \eps^\theta \cdot \big)$ for every $\theta > 0$. 

The Allen-Cahn equation is a popular model to study dynamical phase separation. Let $m:\RR\rightarrow[-1,1]$ be the unique increasing solution of the steady-state problem
\begin{equation}\label{e:m}
    \Delta m + m - m^{2n+1} = 0\;, \qquad m(\pm \infty) = \pm 1\;, \quad m(0)=0\;.
\end{equation}
The solution space for \eqref{e:m} without the restriction $m(0) = 0$ is the one dimensional manifold
\begin{equation*} 
    \mM:=\{m(\cdot-\theta)\,: \; \; \theta \in \RR\}\;.
\end{equation*}
If there is no noise in the equation \eqref{e:main_eqn} (that is, $\gamma=+\infty$), solutions starting close enough to $\mM$ will converge to $\mM$ as $t \rightarrow +\infty$. In particular, solutions starting on $\mM$ do not evolve with time and stay there forever. 

The stochastic perturbation with a small additive spacetime white noise (that is, \eqref{e:main_eqn} with $\sqrt{\fD}$ removed) has also been well understood (\cite{Fun95, BDMP95, BBDMP98, XZZ24}). The general result is that under proper long time (polynomial in $\eps$), the solution is $u_\eps$ is close to $m_{\xi_t^\eps / \sqrt{\eps}}$, where $\xi_t^\eps$ moves approximately according to a diffusion process. The main result of this article is that such a phenomenon still holds in the singular regime \eqref{e:main_eqn} if the solution is properly renormalized. We state our main theorem below.

\begin{thm}
\label{thm:main}
    Fix $\gamma>\frac{1}{2}$ and $\kappa \in (0,\gamma)$ arbitrary. Let $\xi_0\in \RR$, and $u_{\eps}$ be the solution to \eqref{e:main_eqn} with initial data $u_\eps[0] = m_{\xi_0 / \sqrt{\eps}} \in \mM$. Then there exists a process $\xi^{\eps}_t\in \cC(\RR^+,\RR)$ satisfying both of the following:
    \begin{enumerate}
        \item For every $T>0$ and $N>0$, we have
        \begin{equation} \label{e:closeness_process}
        \PP\left(\sup_{t \in [0,T]} \|u_{\eps}[\eps^{-2\gamma-1}t] - m_{\xi^{\eps}_t/\sqrt{\eps}}\|_{\cC^{-\kappa}} > \eps^{\gamma -\kappa}\right) \lesssim \eps^N
        \end{equation}
        for all $\eps\in[0,1]$. 

        \item For every $T>0$, $\xi^{\eps}_t$ converges weakly on $\cC([0,T],\RR)$ to a limiting process $\xi$, which satisfies the It\^o SDE
        \begin{equation} \label{e:limit_sde}
            \md \xi_t = \alpha_1 a(\xi_t)\,\md B_t + \alpha_2 a(\xi_t)a'(\xi_t) \,\md t
        \end{equation}
        with initial data $\xi_0$. Here $B_t$ is the standard Brownian motion, and $\alpha_1$, $\alpha_2$ are explicit constants specified in \eqref{e:alpha1} and \eqref{e:alpha2} below. 
    \end{enumerate}
\end{thm}

An intermediate step towards proving Theorem~\ref{thm:main} is to show $u_\eps[t]$ being close to the stable manifold $\mM$ for arbitrary polynomially long times. We single it out as the following statement.

\begin{thm} \label{thm:Ckappacloseness}
    Fix $\gamma>0$, $\kappa'>0$ and $\nu\in (0,\gamma)$ arbitrary. There exists $p^*\geq 1$ such that for every $p\in [p^*,+\infty]$ and $N ,N'> 0$, the solution $u_\eps$ to \eqref{e:main_eqn} with initial data $u_\eps[0] \in \mM$ satisfies
    \begin{equation*}
        \PP \; \Big( \sup \limits_{t \in [0, \eps^{-N}]} \dist_{\wW^{-\kappa',p}}(u_\eps[t], \mM) > \eps^{\gamma- \nu} \Big) \lesssim \eps^{N'}\;
    \end{equation*}
    for all $\eps \in[0,1]$. Here, $\dist_{\wW^{-\kappa',p}}(v,\mM)$ denotes the $\wW^{-\kappa',p}$ distance of $v$ from $\mM$ as defined in \eqref{e:distance_Ckappa}. 
\end{thm}

\begin{rmk}
The reason we use $\wW^{-\kappa',p}$ for $p<+\infty$ in Theorem~\ref{thm:Ckappacloseness} instead of $\cC^{-\kappa}$ as in Theorem~\ref{thm:main} is that it is a separable space that allows us to apply It\^o's formula. Once It\^o's formula has been used to write down the (approximate) SDE \eqref{e:Ito}, we switch back to $\lL^\infty$-based spaces which is convenient for most of the analysis regarding the deterministic Allen-Cahn flow. We refer to Section~\ref{sec:SDE} for more precise discussions. 
\end{rmk}

Assuming the validity of Theorem~\ref{thm:Ckappacloseness}, the proof for Theorem~\ref{thm:main} follows closely the functional analytic framework developed in \cite{Fun95}. One first identifies a natural candidate for the ``separation point process" $\xi_t^\eps$, derives an SDE for it, and then shows that the SDE converges to a limit (in law) as $\eps \rightarrow 0$. 

The description of the SDE (for fixed $\eps$) involves Fr\'echet derivatives of the functional $\zeta$ from the deterministic Allen-Cahn flow (defined in Proposition~\ref{pr:LinftyEC} below). In the classical situation where the stochastic perturbation in \eqref{e:main_eqn} is given by the spacetime white noise (that is, without the operator $\sqrt{\fD}$), the SDE is well defined for every fixed $\eps$. It was shown in \cite{Fun95} that if $\gamma$ is sufficiently large, then the drift and martingale terms in the SDE converge to well-defined limits and hence the interface process $\xi_t^\eps$ also converges to a limiting diffusion. 

However, in our situation (the stochastic perturbation is half a derivative worse than the spacetime white noise, making renormalization necessary), the intended SDE turns out to be a priori ill-defined for any fixed $\eps>0$! When one tries to formally write down the SDE for fixed $\eps$, there are two sources for the drift term -- one from the drift term in the PDE \eqref{e:main_eqn} for $u_\eps$, and the other from the quadratic variation process of the martingale part in \eqref{e:main_eqn}. One can show that both of these quantities are infinity. 

The key observation of the current work is that these two infinite quantities magically cancel out each other, and that their sum converges to a well-defined limit as $\eps \rightarrow 0$. To rigorously justify such cancellation and convergence, we regularize the solution at scale $\delta>0$ with a specific dependence on the original parameter $\eps$, and carefully analyze the asymptotic behavior of various terms as $(\eps, \delta) \rightarrow (0,0)$ in this specific way. This is carried out in Section~\ref{sec:convergence}.

\begin{rmk}
    Theorem~\ref{thm:main} is expected to hold for all $\gamma>0$. In the classical situation when the stochastic perturbation in \eqref{e:main_eqn} is given by the spacetime white noise (without $\sqrt{\fD}$ operation), \cite{Fun95} developed a functional framework and proved such a result under the restriction $\gamma > 5$. The main reason for this restriction is to provide a sufficiently large power of $\eps$ to kill some potentially divergent terms from the derivation of the SDE. Later, \cite{BBDMP98} proved the same result for all $\gamma > 0$ with a completely different method based on probabilistic couplings. Very recently, \cite{XZZ24} showed that the full $\gamma>0$ regime can also be covered with the functional analytic approach in \cite{Fun95} by introducing additional functional correctors to cancel out the potentially divergent terms for the SDE. 

    In our situation, the reason we assume $\gamma>\frac{1}{2}$ is the same for the restriction $\gamma>5$ as in \cite{Fun95} -- there are terms from the SDEs which vanishes for $\gamma>\frac{1}{2}$ (but can be potentially divergent if $\gamma<\frac{1}{2}$). It is possible to remove this restriction by introducing additional correctors as in \cite{XZZ24}, but such a procedure will be technically involved. Hence, we leave the statement of the main theorem with the restriction $\gamma>\frac{1}{2}$. 
\end{rmk}

\begin{rmk}
    One may also consider the more general situation where $\sqrt{\fD}$ is replaced by $\fD^{\sigma}$ for $\sigma > 0$. When $0<\sigma <\frac{1}{2}$, the equation is classically well-posed and one does not need to renormalize it. In this case, a similar result as Theorem~\ref{thm:main} can be proven by the methods in \cite{Fun95, XZZ24}. The coefficients $\alpha_1$ and $\alpha_2$ in the limiting SDE will depend on $\sigma$, and one expects $\alpha_2\to\infty$ as $\sigma\to\frac{1}{2}$. On the other hand, the arguments in the current article also apply to the case where $\sigma$ is slightly larger than $\frac{1}{2}$ (with logarithmic divergence worsen to a small power divergence). But we restrict ourselves to $\sigma=\frac{1}{2}$ to avoid notations on various exponents and keep technicalities clean. 
\end{rmk}

\subsection*{Organization of the article}

The rest of the article is organized as follows. In Section~\ref{sec:deterministic_preliminaries}, we provide some well-known facts on the deterministic Allen-Cahn flow and the solution theory for \eqref{e:main_eqn}. In Section~\ref{sec:closeness}, we prove Theorem~\ref{thm:Ckappacloseness} -- showing that the solution $u_\eps$ to \eqref{e:main_eqn} starting from the stable manifold $\mM$ remains close to $\mM$ in arbitrarily polynomially long time. Section~\ref{sec:convergence} is devoted to the proof of Theorem~\ref{thm:main}. As mentioned above, we carry out a regularization procedure and carefully study the asymptotic behavior of various quantities, including the precise cancellation effects and convergence properties. The analysis relies on various properties and bounds of the Fr\'echet derivatives of the functional $\zeta$ from the deterministic Allen-Cahn flow. These ingredients are listed and proved in Section~\ref{sec:deterministic}.

\subsection*{Notations}

We now introduce frequently used notations in this article.

\begin{flushleft}
\textbf{Deterministic Allen-Cahn flow}
\end{flushleft}

Throughout, we fix the integer $n \geq 1$, and write
\begin{equation*}
    f_n(v) := v - v^{2n+1}\;.
\end{equation*}
Let $F^t$ denote the deterministic Allen-Cahn flow in the sense that $F^t(v)$ satisfies the equation
\begin{equation} \label{e:deterministic_flow}
    \d_t F^t(v) = \Delta F^t(v) + f_n \big( F^t(v) \big)\;, \qquad F^0(v) = v\;.
\end{equation}
We use $m$ to denote the unique increasing solution to the stationary equation \eqref{e:m} subject to $\pm 1$ boundary conditions and centering condition $m(0)=0$. Let $\mM$ be the manifold of translations of $m$, which consists of all the solutions to \eqref{e:m} subject to $\pm 1$ boundary conditions. We also write $m_\theta = m(\cdot - \theta)$. Let
\begin{equation}\label{e:distance_Ckappa}
    \dist_{\wW^{\nu,p}}(v, \mM) := \inf_{\theta} \|v - m_\theta\|_{\wW^{\nu,p}}
\end{equation}
and 
\begin{equation*}
    \dist (v, \mM) =\dist_{\lL^\infty}(v,\mM):= \inf_{\theta} \|v - m_\theta\|_{\lL^\infty}
\end{equation*}
be the  $\wW^{\nu,p}$ and $\lL^\infty$-distance of $v$ from $\mM$. For $r > 0$, we write
\begin{equation*}
    \vV_{r} = \{ v: \dist(v, \mM)< r \}\;.
\end{equation*}
Let $\beta>0$ be such that Proposition~\ref{pr:LinftyEC} holds. It in particular implies the existence of a functional $\zeta: \vV_\beta \rightarrow \RR$ such that $F^t(v) \rightarrow m_{\zeta(v)}$ exponentially fast in $\lL^\infty$ uniformly over $v \in \vV_\beta$. This functional $\zeta$ will play an essential role throughout the article. We write
\begin{equation*}
    \vV_{\beta,0} = \big\{ v \in \vV_\beta: \zeta(v)=0 \big\}\;.
\end{equation*}
Another special role is played by the generator of the linearized operator of $F^t$ at $v=m$, which we denote by (with a flipped sign)
\begin{equation} \label{e:op_A}
    \aA = -\Delta - f_n'(m)=-\Delta -1 + (2n+1)m^{2n}\;.
\end{equation}
Let $\pP$ denote the $\lL^2$ projection onto the one dimensional space spanned by $m'$ in the sense that 
\begin{equation} \label{e:projection}
    \pP v := \frac{\bracket{v, m'}}{\|m'\|_{\lL^2}} \, m'\;.
\end{equation}
Since $m'\in\sS$, $\pP$ extends to all functions in $\sS'$.

\begin{flushleft}
\textbf{Hermite polynomial and Wick product}
\end{flushleft}

Let $H_k(\cdot)$ be the $k$-th standard Hermite polynomial with leading coefficient $1$, that is, $H_0 = 1$, $H_1(x) = x$, $H_2(x) = x^2 - 1$, etc. For $\sigma>0$, define the $k$-th Hermite polynomial with variance $\sigma^2$ by
\begin{equation} \label{e:Hermite}
    H_k (x; \sigma^2) := \sigma^k H_k (x / \sigma)\;.
\end{equation}
Throughout, we fix $\mu > 0$ and let $X_\eps$ be the stationary solution to the linearized equation with the additional mass $\mu>0$ (defined in \eqref{e:linear_solution}). We use $\diamond$ to denote the Wick product with respect to the Gaussian structure induced by $X_\eps$. More precisely, $X_\eps^{\diamond k}$ is the $k$-th Wick product of $X_\eps$ (defined in Lemma~\ref{le:convergence_Wick_power}), and
\begin{equation} \label{e:Wick}
    F^{\diamond N} := \sum_{k=0}^{N} \begin{pmatrix} N \\ k \end{pmatrix} X_\eps^{\diamond (N-k)} (F - X_\eps)^{k}
\end{equation}
for general process $F$.

\begin{flushleft}
\textbf{Others}
\end{flushleft}

Throughout, we fix $a \in \cC^\infty(\RR)$ with compact support in $(-1,1)$, and write $a_\eps := a(\sqrt{\eps} \cdot)$ for the cutoff in \eqref{e:main_eqn}. For $x \in \RR$, we write
\begin{equation*}
    \bracket{x} := \sqrt{1 +x^2}\;.
\end{equation*}
For a Schwartz function/distribution $\varphi$ on $\RR$, let $\widehat{\varphi}$ denote its Fourier transform in the sense that
\begin{equation*}
    \widehat{\varphi}(\theta) := \int_{\RR} \varphi(x) e^{-2\pi i \theta x} \,\md x\;.
\end{equation*}
The operation $\sqrt{\fD}$ is the Fourier multiplier
\begin{equation} \label{e:half_derivative}
    \widehat{\sqrt{\fD}\varphi}(\theta) := |\theta|^\frac{1}{2} \widehat{\varphi}(\theta)\;.
\end{equation}
Let
\begin{equation} \label{e:heat_kernel}
    q_t (x) := \frac{1}{\sqrt{4 \pi t}} e^{-\frac{|x|^2}{4t}}
\end{equation}
be the standard heat kernel on $\RR$. Let $\qQ_\delta$ be the convolution by $q_{\delta^2}$ in the sense that
\begin{equation*}
    \qQ_\delta \varphi := q_{\delta^2} * \varphi
\end{equation*}
for functions $\varphi$ on $\RR$. We use $\sqrt{\fD}^{\otimes 2}$ and $\qQ_\delta^{\otimes 2}$ be the operations on functions with two variables such that
\begin{equation} \label{e:cross_fractional_derivative}
    (\sqrt{\fD}^{\otimes 2} \varphi)(x_1, x_2) := \sqrt{\fD_{x_1}} \sqrt{\fD_{x_2}} \varphi\;,
\end{equation}
and
\begin{equation} \label{e:cross_heat_mollification}
    (\qQ_{\delta}^{\otimes 2} \varphi)(x_1, x_2) := \iint\limits_{\RR^2} q_{\delta^2}(x_1 - y_1) q_{\delta^2} (x_2 - y_2) \varphi(y_1, y_2) \,\md y_1 \,\md y_2\;.
\end{equation}
For $\alpha \in (0,1)$, $\cC^\alpha$ denotes the usual H\"older space with norm
\begin{equation*}
    \|f\|_{\cC^\alpha} := \|f\|_{\lL^\infty} + \sup_{x \neq y} \frac{|f(x) - f(y)|}{|x-y|^\alpha}\;.
\end{equation*}
The norm $\cC^{-\alpha}$ (for $\alpha \in (0,1)$) is defined as
\begin{equation*}
    \|\varphi\|_{\cC^{-\alpha}} := \sup_{x \in \RR} \sup_{\lambda \in (0,1)} \sup_{\psi: \|\psi\|_{\cC^\alpha} \leq 1} \lambda^\alpha|\bracket{\varphi, \psi_x^\lambda}|\;,
\end{equation*}
where $\psi_x^\lambda := \lambda^{-1} \psi \big( \frac{\cdot - x}{\lambda} \big)$. Also, for $\alpha \in [0,1)$ and $p \in [1, +\infty)$, $\wW^{\alpha,p}$ denotes the standard Sobolev space (with $\wW^{0,p} = \lL^p$). For $\alpha \in (0,1)$ and $p \in (1,+\infty)$, $\wW^{-\alpha,p}$ is the dual of $\wW^{\alpha, p'}$ with $p'$ being the conjugate of $p$. We also write $\wW^{\alpha,\infty} = \cC^\alpha$
for $\alpha \in (-1,1) \setminus \{0\}$ for simplicity. We have the standard embedding
\begin{equation*}
    \wW^{\alpha, p}(\RR) \hookrightarrow \cC^{\alpha'}(\RR)
\end{equation*}
for $\alpha, \alpha' \in \RR$ with $\alpha' < \alpha - \frac{1}{p}$. 

For $\Phi: \xX \rightarrow \yY$ defined on a Banach space $\xX$, we use $D^n \Phi$ to denote its $n$-th Fr\'echet derivative. We use $\bracket{F, G}$ to denote the usual pairing between $G \in \xX$ and $F \in \xX^*$. In most situations it is clear what the relevant spaces $\xX$ and $\xX^*$ are, so we omit them in the notation. 

Throughout, $\gamma>\frac{1}{2}$ is a fixed parameter, and $\beta>0$ is also fixed (such that Proposition~\ref{pr:LinftyEC} holds). We also fix small exponents $0 < \kappa' \ll \kappa \ll \eta \ll (\gamma \wedge 1)$ such that the following relations hold:
\begin{equation} \label{e:exponents_dependence}
    \gamma - \kappa > \frac{1}{2}\;, \quad 4 \kappa < \eta := \frac{1}{100n}\;, \quad \kappa' + 2000 n (\gamma+1) \kappa' < \kappa\;.
\end{equation}
Recall that $\eps>0$ is a parameter in the original equation that will be sent to $0$. Let $\delta >0$ be another parameter (used as regularization in Section~\ref{sec:convergence} below) such that
\begin{equation} \label{e:parameters_dependence}
    \delta = \delta (\eps) = \eps^{1000n(\gamma+1)}\;.
\end{equation}
Note that the deterministic bounds in Section~\ref{sec:deterministic} are uniform in all $(\eps, \delta) \in (0,\frac{1}{2})^2$, but the convergence in Section~\ref{sec:convergence} is based on the joint limit $(\eps,\delta) \rightarrow (0,0)$ along the above relation. 

Finally, we use $c$, $C$ to denote generic constants whose values may change from line to line. We also use the notation $A \lesssim B$ to denote $A \leq C \cdot B$ for some constant $C$ {independent of the parameter $\eps$ (and other parameters that are normally clear from the context).

\subsection*{Acknowledgements}

We thank Tadahisa Funaki and Wenhao Zhao for helpful discussions. W. Xu is supported by the Ministry of Science and Technology via the National Key R\&D Program of China (no.2023YFA1010102) and National Science Foundation China via the standard project grant (no.8200906145).

\section{Preliminaries}
\label{sec:deterministic_preliminaries}

\subsection{Deterministic Allen-Cahn flow}

We first recall a well-known fact about the function $m$ defined in \eqref{e:m}. 

\begin{lem} \label{lem:statationary_exponential_decay}
    The function $m$ is smooth, and its derivative is a Schwartz function with
    \begin{equation*}
        |m'(x)|\lesssim e^{-c |x|}
    \end{equation*}
    for some $c>0$. 
\end{lem}
\begin{proof}
See for example \cite[Section~2]{BCS20}. 
\end{proof}

Recall from \eqref{e:op_A} that $-\aA$ is the linearized deterministic flow with respect to the initial data, and from \eqref{e:projection} that $\pP$ is the $\lL^2$ projection onto the one dimensional subspace spanned by $m'$, and extended to all Sobolev functions. The next statement says that $e^{-t \aA}$ behaves like $e^{t\Delta}$ for small $t$, and has exponential decay outside the image of $\pP$ as $t \rightarrow +\infty$.

\begin{prop} \label{prop:LinftySG}
There exists $c>0$ such that for every $\nu_2\geq \nu_1$ and $p\in[1,+\infty]$, we have
    \begin{equation*}
    \|e^{-t \aA}-\pP\|_{\wW^{\nu_1,p} \rightarrow \wW^{\nu_2,p}} \lesssim t^{-\frac{\nu_2 - \nu_1}{2}} e^{-ct}
\end{equation*} 
for all $t \geq 0$.
\end{prop}
\begin{proof}
We first consider the case $\nu_1 =\nu_2 = 0$. The case for $p=2$ is well-known; see for example \cite[Lemma~3.1]{Fun95}. For case for $p=+\infty$ is proven in \cite[Section~2, Eq.(23)]{BCS20}. The bound is then extended to all $p \geq 2$ by Riesz-Thorin theorem, and then to all $p \in [1,2]$ by duality. See also \cite[Proposition~3.2]{XZZ24} for a sketch argument. 

The statement for general $\nu_1 \leq \nu_2$ follows from the smoothing effect of $e^{t\Delta}$ and boundedness of $f_n'(m) = 1 - (2n+1) m^{2n}$.
\end{proof}

Recall from \eqref{e:deterministic_flow} that $F^t$ denotes the deterministic Allen-Cahn flow. An important feature of $F^t$ is that if the initial data $v$ is sufficiently close to the stable manifold $\mM$, then there is a unique $\zeta(v) \in \RR$ such that $F^t(v)$ converges to $m_{\zeta(v)} \in \mM$ as $t \rightarrow +\infty$. A precise statement is the following. 

\begin{prop} \label{pr:LinftyEC}
    There exist $\beta,c,C> 0$ such that for all $v\in\vV_{\beta}$, there exists a $\zeta(v)\in\RR$ such that
    \begin{equation*}
        \Vert F^t(v) - m_{\zeta(v)}\Vert_{\lL^\infty} \leq Ce^{-ct}\dist(v,\mM)\;.
    \end{equation*}
\end{prop}
\begin{proof}
    A version of $\lL^2$ convergence for $v$ in an $\lL^2$ neighborhood of $\mM$ is proven in \cite[Theorem~7.1]{Fun95}. For the above $\lL^\infty$ version, the new ingredients are the exponential decay of $\|e^{-t\aA}-\pP\|_{\lL^\infty \rightarrow \lL^\infty}$ in Proposition~\ref{prop:LinftySG} and the properties of the so-called Fermi coordinate (also known as the linear center) functional in \cite[Proposition~3.3]{XZZ24}. Based on these two ingredients, the proof follows in the same way as \cite[Theorem~5]{BBDMP98}, so we omit the details. 
\end{proof}

\begin{lem} \label{lem:fermi}
    For every $\nu>0$ and $p\in(1,+\infty]$, there exists $\beta>0$ such that for every $v$ with $\dist_{\wW^{-\nu,p}}(v ,\mM)<\beta$, there exists a unique $\ell = \ell(v) \in \RR$ such that
    \begin{equation*} 
        \bracket{v,m'_\ell} = 0\;.
    \end{equation*}
    Furthermore, there exists $C>0$ such that if $\dist_{\wW^{-\nu,p}}(v ,\mM)< \beta$ and $\theta \in \RR$ are such that
    \begin{equation*}
        \|v - m_\theta\|_{\wW^{-\nu,p}} < \beta\;,
    \end{equation*}
    then we have the bound
    \begin{equation*}
        |\ell(v) - \theta| \leq C \|v - m_\theta\|_{\wW^{-\nu,p}}\;.
    \end{equation*}
\end{lem}
\begin{proof}
    A version in the $\lL^2$ neighborhood of $\mM$ is considered in \cite[Section~3]{Fun95}. A version in the $\lL^\infty$ neighborhood of $\mM$ is proved in \cite[Proposition~3.3]{XZZ24}. The proof for $v$ in the $\wW^{-\nu,p}$ neighborhood of $\mM$ follows from the same way as in \cite[Proposition~3.3]{XZZ24}, so we omit the details.
\end{proof}

\subsection{Renormalization and well-posedness} 
\label{sec:renormalization}

We now give a rigorous meaning of the solution $u_\eps$ to \eqref{e:main_eqn}. Fix $\mu> 0$. Let $X_{\eps}$ be the stationary-in-time solution to
\begin{equation} \label{e:linear_solution}
    \partial_t X_\eps = (\Delta - \mu) X_\eps + \eps^{\gamma} a_\eps \sqrt{\fD}\dot{W}\;.
\end{equation}
Recall $\qQ_\delta$ is the mollification by the heat kernel $q_{\delta^2}$. Let
\begin{equation} \label{e:variance_variable}
    \fC_\eps^{(\delta)}(x) := \EE \, |(\qQ_\delta X_\eps)(t,x)|^2\;,
\end{equation}
which does not depend on $t$ by stationarity of $X_\eps$. For every $k \in \NN$, the $k$-th Wick power of $X_\eps$, denoted by $X_\eps^{\diamond k}$, is defined as the $\delta \rightarrow 0$ limit as $H_{2n+1}\big( \qQ_\delta X_\eps; \fC_\eps^{(\delta)} \big)$. The following lemma guarantees that it is well defined. 

\begin{lem} \label{le:convergence_Wick_power}
Fix $\eps>0$ and $k \in \NN$. There is a stochastic object $X_\eps^{\diamond k}$ such that for every $T>0$, every $p \in [1, +\infty)$ and every $\nu > 0$, we have
\begin{equation*}
    \EE \sup_{t \in [0,T]} \left\| H_{k}\big( \qQ_\delta X_\eps[t]; \fC_\eps^{(\delta)} \big) -X_\eps^{\diamond k}[t]\right\|_{\cC^{-\nu}}^p \lesssim_{\eps,T} \delta^{\frac{\nu p}{2}}\;,
\end{equation*}
where we recall from \eqref{e:Hermite} that $H_k (\cdot\, ;  \fC_\eps^{(\delta)})$ is the $k$-th Hermite polynomial with variance $\fC_\eps^{(\delta)}$. The proportionality constant is independent of $\delta$. 
\end{lem}

The proportionality constant here depends on $\eps$ and $T$, but we do not quantify their dependence at this stage. The above lemma is to rigorously define the objects $X_\eps^{\diamond k}$ on $[0,T]$ for every fixed $\eps$ and $T$. Quantitative properties of $X_\eps^{\diamond k}$ will be investigated in Appendix~\ref{app:linear_smallness}. 

Given the objects $X_\eps^{\diamond k}$ defined in Lemma~\ref{le:convergence_Wick_power}, let $w_\eps$ be the solution to the equation
\begin{equation*}
    \partial_t w_\eps =\Delta w_\eps+ w_\eps + (1+\mu)X_\eps - \sum_{k=0}^{2n+1} \binom{2n+1}{k}X_\eps^{\diamond (2n+1-k)} w_\eps^{k}
\end{equation*}
with initial data $w_\eps[0] = u_\eps[0] - X_\eps[0]$. Then $u_\eps := X_\eps +w_\eps$ solves the SPDE
\begin{equation*}
    \partial_t u_\eps = \Delta u_\eps + u_\eps - u_\eps^{\diamond (2n+1)} + \eps^{\gamma} a_\eps \sqrt{\fD}\dot{W}\;,
\end{equation*}
where $u_\eps^{\diamond (2n+1)}$ is the Wick product with respect to the Gaussian structure induced by $X_\eps$ in the sense that (see \eqref{e:Wick} for the notation)
\begin{equation} \label{e:u_Wick_expansion}
    u_\eps^{\diamond (2n+1)} := \sum_{k=0}^{2n+1} \binom{2n+1}{k}  X_\eps^{\diamond (2n+1-k)} w_\eps^k\;.
\end{equation}
It is important to note that the definition of $u_\eps$ depends on $\mu$ since the renormalization $X_\eps^{\diamond k}$ does. Throughout most of this article, $\mu$ is treated as a fixed constant, except when calculating $\alpha_2$ (a diffusion constant in the limiting SDE \eqref{e:limit_sde}) where we specify its dependence on $\mu$. Moreover, we will see that $\alpha_2 \rightarrow \infty$ as $\mu \rightarrow 0$ (while $\alpha_1$ does not depend on $\mu$). This suggests that the analysis may fail to yield similar results if $\mu = 0$ in the definition of $X_\eps$.

We now give asymptotic behavior of the renormalization function $\fC_\eps^{(\delta)}$ as $(\eps,\delta) \rightarrow (0,0)$. Define the constant $C^{(\delta)}$ by
\begin{equation} \label{e:C_delta}
    C^{(\delta)} := \int_{\RR} \frac{|\eta| e^{-8 \pi^2 \delta^2 \eta^2}}{2(4 \pi^2 \eta^2 + \mu)}\,\md \eta\;.
\end{equation}
We have the following lemma regarding $\fC_\eps^{(\delta)}$. 

\begin{lem} \label{le:renormalisation_approx}
We have the bound
\begin{equation*}
    \left\| \fC_\eps^{(\delta)} - \eps^{2\gamma} C^{(\delta)} a_\eps^2 \right\|_{\lL^\infty} \lesssim \eps^{2\gamma+1}
\end{equation*}
for all $\eps\in[0,1]$ and $\delta\in(0,\frac{1}{2}]$. 
\end{lem}
\begin{proof}
Postponed to Appendix~\ref{app:renormalization}. 
\end{proof}

Now we show that the constant $C^{(\delta)}$ blows up logarithmically in $\delta$. 

\begin{lem} \label{lem:Cdelta_divergence}
There exists a universal constant $\alpha$ (independent of $\delta$ and $\mu$) such that
    \begin{equ}
        C^{(\delta)} - \frac{|\log \delta|}{4 \pi^2} = \alpha - \frac{\log \mu}{8 \pi^2} + \oO \big( \mu \delta^2 \big)
    \end{equ}
as $\delta \rightarrow 0$. 
\end{lem}
\begin{proof}
By a direct change of variable, we have
\begin{equation*}
    C^{(\delta)} = \frac{e^{2 \mu \delta^2}}{8 \pi^2} \int_{2 \mu \delta^2}^{+\infty} \frac{e^{-y}}{y} \, {\rm d}y\;.
\end{equation*}
The conclusion then follows. 
\end{proof}

\section{Bounding distance from $\mM$ -- proof of Theorem~\ref{thm:Ckappacloseness}}
\label{sec:closeness}

In this section, we prove Theorem~\ref{thm:Ckappacloseness}, showing that the solution always stays close to the manifold. The argument is essentially the same as \cite[Section~5]{XZZ24} but more technically involved, since the solution $u_\eps$ here is distribution valued. 

Recall from \eqref{e:linear_solution} that $X_{\eps}$ is the stationary solution the linearized equation with an additional mass $\mu>0$. The following lemma gives the smallness of $X_{\eps}$ at any polynomial time scale.

\begin{lem} \label{lem:smallnessOfLinearSolution}
    Fix $N$ and $N'$ arbitrary. For every $p \geq 1$, let $A_\eps = A_\eps^p$ be the event
\begin{equation*}
    A_\eps:=\Big\{\omega:\sup_{\substack{1\leq k\leq 2n+1\\t \in [0, \eps^{-N}]}} \Vert X_\eps^{\diamond k}[t] \Vert_{\wW^{-\kappa',p}}\leq  \eps^{\gamma - \frac{\nu}{3}}\Big\}\;.
\end{equation*}
    Then, there exists $p_0 > 1$ such that for all $p \in [p_0, +\infty]$, we have
    \begin{equation*}
        \PP \big( A_\eps \big) \geq 1 - C \eps^{N'}
    \end{equation*}
    for all $\eps \in [0,1]$. The constant $C$ here is independent of $\eps$ and $p \in [p_0, +\infty]$ (but could depende on $p_0$). 
\end{lem}

\begin{lem} \label{lem:PerturbationVersion}
Fix $n \geq 1$ and $\eta =\frac{1}{100n}$. Let $\kappa'\in(0,\frac{\eta}{2})$, $\nu\in(0,\gamma)$ and $p\in(\frac{2}{\eta},+\infty]$. Let $\Gamma_\eps^{(0)}, \dots, \Gamma_\eps^{(2n)} \in \cC(\RR^+, \sS')$ be continuous evolution in the space of distributions such that 
\begin{equation*}
    \sum_{k=0}^{2n} \sup_{t \in [0,T_\eps]} \left\|\Gamma_\eps^{(k)}[t] \right\|_{\wW^{-\kappa',p}} \leq \eps^{\gamma - \frac{\nu}{2}}
\end{equation*}
for some $T_\eps \rightarrow +\infty$ as $\eps \rightarrow 0$. Let $w_\eps$ be the solution to
\begin{equation}\label{e:perturbed_allen_cahn}
    \partial_t w_\eps = \Delta w_\eps  + w_\eps-w_\eps^{2n+1} + \sum_{k=0}^{2n} \Gamma^{(k)}_\eps  w_\eps^k
\end{equation}
with initial data $w_\eps[0]$ satisfying
\begin{equation} \label{e:w_initial_close}
    \dist_{\wW^{-\kappa',p}}(w_\eps[0], \mM) \leq\eps^{\gamma - \nu}\;.
\end{equation}
Then, there exist $\eps_0, T>0$ depending on $\gamma$, $\nu$ and $\kappa$ only such that 
\begin{equation} \label{e:closeness_longtime_1}
    \sup \limits_{t \in [T, T_\eps]} \dist_{\wW^{\eta,p}}(w_\eps[t], \mM) \leq \eps^{\gamma - \nu}
    \end{equation}
for all $\eps < \eps_0$. Moreover, if the initial data satisfies the stronger assumption
\begin{equation} \label{e:w_initial_strong}
    \dist_{\wW^{-\kappa',p}}(w_\eps[0], \mM) \leq \eps^{\gamma - \frac{\nu}{2}}\;,
\end{equation}
then we have 
\begin{equation} \label{e:closeness_longtime_2}
    \sup \limits_{t \in [0, T_\eps]} \dist_{\wW^{-\kappa',p}}(w_\eps[t], \mM) \leq \eps^{\gamma - \nu}
\end{equation}
for all $\eps<\eps_0$.
\end{lem}
\begin{proof}
We start with the bound \eqref{e:closeness_longtime_1}. We first show there exist $\eps_0, T>0$ depending on $\gamma$, $\nu$ and $\kappa'$ only such that 
\begin{equation}\label{e:closeness_onestep}
    \sup \limits_{t \in [T, 2T]} \dist_{\wW^{\eta,p}} \big(w_\eps[t], \mM \big) \leq \eps^{\gamma - \nu}
\end{equation}
for all $\eps < \eps_0$. Recall the functional $\ell$ given in Lemma~\ref{lem:fermi}. Let $D_\eps = w_\eps - m_{\ell(w_\eps[0])}$. We assume without loss of generality that $\ell(w_\eps[0]) = 0$, so that
\begin{equation*}
    \bracket{D_\eps[0], m'} = \bracket{w_\eps[0], m'} - \bracket{m, m'} = 0\;.
\end{equation*}
Then, $D_\eps$ satisfies the equation
\begin{equation*}
    \d_t D_\eps = -\aA D_\eps + Q_\eps + R_\eps\;, \quad D_\eps[0] = w_\eps[0] - m\;.
\end{equation*}
where $\aA = -\Delta -1 + (2n+1)m^{2n}$ is defined in \eqref{e:op_A}, and $Q_\eps$ and $R_\eps$ are given by
\begin{equation*}
    Q_\eps = - \sum_{k=2}^{2n+1} \begin{pmatrix} 2n+1 \\ k \end{pmatrix} m^{2n+1-k} D_\eps^k\;, \quad R_\eps = \sum_{k=0}^{2n} \Gamma_\eps^{(k)} w_\eps^k
\end{equation*}
respectively. Furthermore, by Lemma~\ref{lem:fermi} and the assumption \eqref{e:w_initial_close}, the initial data $D_\eps[0]$ satisfies the bound 
\begin{equation} \label{e:D_initial_small}
    \begin{split}
    \|D_\eps[0]\|_{\wW^{-\kappa',p}} &\leq \dist_{\wW^{-\kappa',p}} (w_\eps[0], \mM) + \|m_{\theta} - m\|_{\wW^{-\kappa',p}}\\
    &\lesssim \dist_{\wW^{-\kappa',p}} (w_\eps[0], \mM) \lesssim \eps^{\gamma-\nu}\;,
    \end{split}
\end{equation}
where $\theta \in \RR$ is such that $\|w_\eps[0] - m_\theta\|_{\wW^{-\kappa',p}} = \dist_{\wW^{-\kappa',p}}(w_\eps[0], \mM)$. By Duhamel's formula, we have
\begin{equation}\label{e:D_Duhamel}
    D_\eps[t]  =  e^{-t\aA} D_\eps[0] + \int_0^t e^{-(t-s)\aA} \big( Q_\eps[s] + R_\eps[s] \big) \,\md s\;.
\end{equation}
Taking $\wW^{\eta,p}$-norm on both sides and applying Proposition~\ref{prop:LinftySG}, we get
\begin{equation} \label{e:duhamel_D}
\begin{aligned}
    \Vert &D_\eps[t] \Vert_{\wW^{\eta,p}} \lesssim \big( 1 + t^{-\frac{\kappa'+\eta}{2}} \big) e^{-ct} \|D_\eps[0]\|_{\wW^{-\kappa',p}}\\
    &+ \int_0^t \Big( \big(1+(t-s)^{-\frac{\eta}{2}}\big) \Vert Q_\eps[s] \Vert_{\lL^p} + \big(1+(t-s)^{-\frac{\kappa'+\eta}{2}}\big) \Vert R_\eps[s] \Vert_{\wW^{-\kappa',p}} \Big)\,\md s
\end{aligned}
\end{equation}
Here we used that $\bracket{D_\eps[0], m'} = 0$ to get the exponential factor $e^{-c t}$. For $Q_\eps$ and $R_\eps$ in the integrands, we have the bounds
\begin{equation} \label{e:Q_bound}
    \Vert Q_\eps[s]\Vert_{\lL^p} \lesssim \Vert D_\eps[s]\Vert_{\lL^p}\Vert D_\eps[s]\Vert_{\lL^\infty} \big(1+\Vert D_\eps[s] \Vert_{\lL^\infty}^{2n-1}\big)\;,
\end{equation}
and
\begin{equation}\label{e:R_bound}
    \Vert R_\eps[s]\Vert_{\wW^{-\kappa',p}} \lesssim \sum_{k=0}^{2n} \|\Gamma_\eps^{(k)}[s]\|_{\wW^{-\kappa',p}} \|w_\eps[s]\|_{\cC^{\eta/2}}^{k} \lesssim \eps^{\gamma-\frac{\nu}{2}} \big( 1 + \|D_\eps[s]\|_{\wW^{\eta,p}}^{2n} \big)\;,
\end{equation}
where we have used $w_\eps = m + D_\eps$, the assumptions on $\Gamma_\eps^{(k)}$ and the embedding $\wW^{\eta,p} \hookrightarrow \cC^{\eta/2}$ in dimension one (since $p>\frac{2}{\eta}$). Plugging these bounds back into \eqref{e:duhamel_D} and using the Sobolev embedding $\wW^{\eta,p} \hookrightarrow \lL^\infty$ as well as \eqref{e:D_initial_small}, we deduce there exists $C^*>1$ such that 
\begin{equation} \label{e:duhamel_D_Ceta}
\begin{split}
    \|D_\eps[t]\|_{\wW^{\eta,p}} \leq &\phantom{1} C^* \bigg( \eps^{\gamma-\nu} (1 + t^{-\frac{\kappa'+\eta}{2}}) \, e^{-ct} + \eps^{\gamma-\frac{\nu}{2}} (1+t)\\
    &+ \int_{0}^{t} \big( 1 + (t-s)^{-\frac{\kappa'+\eta}{2}} \big) \|D_\eps[s]\|_{\wW^{\eta,p}}^2 \big(1 + \|D_\eps[s]\|_{\wW^{\eta,p}}^{2n-1} \big) {\rm d}s \bigg)\;.
\end{split}    
\end{equation}
Now, let
\begin{equation*}
    \tau_\eps^* := \inf \left\{ t > 0: \|D_\eps[t]\|_{\wW^{\eta,p}} \geq 2C^*(1+t^{-\frac{\kappa'+\eta}{2}}) \eps^{\gamma-\nu} \right\}\;.
\end{equation*}
Then from \eqref{e:duhamel_D_Ceta}, we deduce there exists $C_0>0$ such that
\begin{equation}\label{e:D_Ceta_bound}
    \|D_\eps[t]\|_{\wW^{\eta,p}} 
    \leq C_0 \big( \eps^{2(\gamma-\nu)} + \eps^{\gamma-\frac{\nu}{2}} \big) \, (1+t)+C^*(1+t^{-\frac{\kappa'+\eta}{2}})e^{-ct}\eps^{\gamma-\nu}\;,
\end{equation}
for all $t\in[0,\tau_\eps^*]$. Furthermore, by continuity, we also have
\begin{equation*}
    \|D_{\eps}[\tau_\eps^*]\|_{\wW^{\eta,p}} \geq 2C^*(1+ (\tau_\eps^*)^{-\frac{\kappa'+\eta}{2}}) \eps^{\gamma-\nu}\;.
\end{equation*}
Combining the above two bounds, we get
\begin{align*}
    &2C^*(1+(\tau_\eps^*)^{-\frac{\kappa'+\eta}{2}}) \eps^{\gamma-\nu} \leq \|D_\eps[\tau_\eps^*]\|_{\wW^{\eta,p}} \\
    \leq &C_0 \big( \eps^{2(\gamma-\nu)} + \eps^{\gamma-\frac{\nu}{2}} \big) \, (1+\tau_\eps^*)+C^*(1+(\tau_\eps^*)^{-\frac{\kappa'+\eta}{2}})e^{-c\tau_\eps^*}\eps^{\gamma-\nu}\;. 
\end{align*}
Since $\nu \in (0, \gamma)$, this further implies there exists $c', \theta>0$ such that $\tau_\eps^{*} \geq c' \eps^{-\theta} \wedge T_\eps$ for all sufficiently small $\eps$. 

We now choose $T>0$ such that
\begin{equation*}
    C^* (1+T^{-\frac{\kappa'+\eta}{2}}) \, e^{-cT} < \frac{1}{2}\;. 
\end{equation*}
For this choice of $T$, we have $\tau_\eps^* > 2T$ for sufficiently small $\eps$. Since \eqref{e:D_Ceta_bound} holds for all $t \in [0, \tau_\eps^*]$ and hence in particular up to $2T$, we deduce the bound \eqref{e:closeness_onestep} holds for sufficiently small $\eps$. 

To extend it to all $t \in [T,T_\eps]$, we note that the same argument applies to the solution starting at time $T$ and the corresponding bound in the time interval $[2T, 3T]$, and so on. The bound \eqref{e:closeness_longtime_1} then follows by iteration. 

Now we turn to showing \eqref{e:closeness_longtime_2} under the stronger assumption \eqref{e:w_initial_strong}. In view of \eqref{e:closeness_longtime_1}, it suffices to prove \eqref{e:closeness_longtime_2} with supremum taken over $t \in [0,T]$. Note that for $t < \tau_\eps^{*}$ (and hence also $t \in [0,T]$ since $\tau_\eps^* > 2T$), taking $\wW^{-\kappa',p}$-norm on both sides of \eqref{e:D_Duhamel} and using \eqref{e:Q_bound} and \eqref{e:R_bound}, we have
\begin{equation*}
    \begin{split}
    \|D_\eps[t]\|_{\wW^{-\kappa',p}} &\lesssim \|D_\eps[0]\|_{\wW^{-\kappa',p}} + \int_0^t \big( \|Q_\eps[s]\|_{\lL^p} + \|R_\eps[s]\|_{\wW^{-\kappa',p}} \big)\,\md s\\
    &\lesssim \big( \eps^{2(\gamma-\nu)} + \eps^{\gamma-\frac{\nu}{2}} \big) \, (1+t) + \eps^{\gamma-\frac{\nu}{2}}\;.
    \end{split}
\end{equation*}
The claim \eqref{e:closeness_longtime_2} then follows.
\end{proof}

We are now ready to prove Theorem~\ref{thm:Ckappacloseness}. 

\begin{proof} [Proof of Theorem~\ref{thm:Ckappacloseness}]
Recall $X_\eps$ is the stationary solution to the linearized equation \eqref{e:linear_solution} and $X_\eps^{\diamond k}$ is its $k$-th Wick power with respect to its Gaussian structure. 

Taking $\Gamma_\eps^{(k)} = - \binom{2n+1}{k} X_\eps^{\diamond (2n+1-k)}$ for $1 \leq k \leq 2n$ and $\Gamma_\eps^{(0)} = (1+\mu) X_\eps - X_\eps^{\diamond (2n+1)}$ in \eqref{e:perturbed_allen_cahn}, then $u_\eps := X_\eps + w_\eps$ solves \eqref{e:main_eqn}. The claim of the theorem then follows from Lemmas~\ref{lem:smallnessOfLinearSolution} and~\ref{lem:PerturbationVersion}. 
\end{proof}


\section{Convergence}\label{sec:convergence}

Recall $u_\eps$ is the solution to \eqref{e:main_eqn} with initial data $u_\eps[0] \in \mM$, as rigorously defined in Section~\ref{sec:renormalization}. Let
\begin{equation*}
    v_\eps[t] := u_\eps [\eps^{-2\gamma-1}t]\;.
\end{equation*}
Then $v_\eps$ satisfies the equation (in law)
\begin{equ} \label{e:v_eps}
    \d_t v_\eps = \eps^{-2\gamma-1} \big( \Delta v_\eps + v_\eps -v_\eps^{\diamond (2n+1)} \big) +\eps^{-\frac{1}{2}} a_\eps \sqrt{\fD}\dot{W}\;, \quad v_\eps[0] \in \mM\;.
\end{equ}
Here, the Wick product is taken with respect to the Gaussian structure of the stationary solution to \eqref{e:linear_solution}, and is the same with the Wick product taken in \eqref{e:main_eqn}. More precisely, we have
\begin{equation*}
    \big(v_\eps[t]\big)^{\diamond (2n+1)} = \big(u_\eps[t_\eps]\big)^{\diamond (2n+1)} = \sum_{k=0}^{2n+1} \begin{pmatrix} 2n+1 \\ k \end{pmatrix} \big( X_{\eps}[t_\eps] \big)^{\diamond (2n+1-k)} \big( w_\eps [t_\eps] \big)^k
\end{equation*}
for $t_\eps = \eps^{-2\gamma-1}t$. We now start to identify a candidate of the phase separation point and analyze its evolution. Ideally, one would like to set
\begin{equation*}
    \xi_\eps(t) := \sqrt{\eps} \zeta \big( v_\eps[t] \big)\;,
\end{equation*}
derive an SDE for $\xi_\eps(t)$ and study the asymptotics as $\eps \rightarrow 0$. Indeed, applying It\^o's formula to $\xi_\eps$, we obtain
\begin{equation*}
    \md \xi_\eps (t) = \big( b_{1,\eps}(t) + b_{2,\eps}(t) \big) \,\md t + \md \, \text{Mart}_\eps(t)\;,
\end{equation*}
where the drift terms
\begin{equation*}
    \begin{split}
    b_{1,\eps}(t) &= \eps^{-2\gamma-\frac{1}{2}} \langle D \zeta (v_\eps), \, \Delta v_\eps + v_\eps - v_\eps^{\diamond (2n+1)} \rangle\;,\\
    b_{2,\eps}(t) &= \frac{1}{2 \sqrt{\eps}} \int_\RR \sqrt{\fD}^{\otimes 2} \big( a_\eps^{\otimes 2} \cdot D^2 \zeta (v_\eps) \big)(y,y) \,\md y
    \end{split}
\end{equation*}
come respectively from the drift and quadratic variations parts of \eqref{e:v_eps}, and the martingale term $\text{Mart}_\eps$ from the stochastic integral in \eqref{e:v_eps} has the expression
\begin{equation*}
    \md \, \text{Mart}_\eps(t) = \langle D \zeta \big( v_\eps[t] \big), \, a_\eps \sqrt{\fD} \md W[t] \rangle\;.
\end{equation*}
Here, we recall from \eqref{e:cross_fractional_derivative} that the integrand in $b_{2,\eps}(t)$ above is the cross fractional derivative $\sqrt{\fD}^{\otimes 2}$ on the two-variable function $a_\eps^{\otimes 2} \cdot D^2 \zeta (v_\eps)$ and then (formally) evaluated at the diagonal. 

While one can show the martingale term $\text{Mart}_\eps$ converges to a limit as $\eps \rightarrow 0$, both $b_{1,\eps}$ and $b_{2,\eps}$ above are just formal expressions. In fact, both of them are infinite for every fixed $\eps>0$. 

To see this, we first note that by orthogonality between $D \zeta$ and the gradient flow (see \eqref{e:zeta_magical_cancellation} below), we formally have
\begin{equation*}
    \begin{split}
    b_{1,\eps}(t) &= \eps^{-2\gamma-\frac{1}{2}} \langle D \zeta (v_\eps), \, \Delta v_\eps + v_\eps - v_\eps^{2n+1} + v_\eps^{2n+1} - v_\eps^{\diamond (2n+1)} \rangle\\
    &= \eps^{-2\gamma-\frac{1}{2}} \langle D \zeta (v_\eps), \, v_\eps^{2n+1} - v_\eps^{\diamond (2n+1)} \rangle\;.
    \end{split}
\end{equation*}
The difference $v_\eps^{2n+1} - v_\eps^{\diamond (2n+1)}$ contains an infinite renormalization, and hence it is natural to expect that $b_{1,\eps}$ above is ill defined. As for $b_{2,\eps}$, it turns out that even the integrand on its right hand side is not defined for fixed $y$, and hence is also infinite. 

Magically, it turns out that the infinite parts of $b_{1,\eps}$ and $b_{2,\eps}$ cancel out, and their sum converges to a finite limit as $\eps \rightarrow 0$. This gives the limiting equation \eqref{e:limit_sde}. 

Before we give rigorous justification of this cancellation and convergence, we first define the constant $C_0^*$ given by
\begin{equation} \label{e:C0}
    C_0^* := \int_{\RR} D \zeta (y; m_{\theta}) \cdot (y-\theta) \, m_\theta^{2n-1}(y) \, {\rm d}y = - \frac{1}{\|m'\|_{\lL^2}^2} \int_{\RR} y \, m'(y) \, m^{2n-1}(y) \, {\rm d}y\;.
\end{equation}
The integral in the middle is independent of $\theta \in \RR$ by translation properties of $D \zeta$ (\eqref{e:Dzeta_translation} below), and one can take $\theta = 0$ (and hence $m_\theta = m$) for simplicity. The second equality comes from the identity $D \zeta (m) = - m' / \|m'\|_{\lL^2}^2$ (see \cite[Theorem~7.2]{Fun95} and \cite[Lemma~3.11, Proposition~4.2]{XZZ24}). 

The constant $C_0^*$ will appear in the cancellation of the formal sum $b_{1,\eps} + b_{2,\eps}$. In fact, we will show later that the divergent parts of $b_{1,\eps}$ and $b_{2,\eps}$ are from the same divergent quantity multiplied by the above two expressions of $C_0^*$ but with opposite signs, and hence they cancel out. In the special case $n=1$, we have $C_0^* = -\frac{3}{2}$.

\subsection{Derivation of the SDE -- proof of Theorem~\ref{thm:main}}
\label{sec:SDE}

We now start to justify this cancellation and convergence. This involves careful analysis of a regularized version of $\xi_\eps$ given above. Recall from \eqref{e:parameters_dependence} that we always set the regularization parameter $\delta$ depending on $\eps$ such that $\delta = \delta (\eps) = \eps^{1000 n(\gamma+1)}$. Also recall from \eqref{e:exponents_dependence} the small fixed parameters $0 < \kappa' \ll \kappa \ll \eta = \frac{1}{100n}$. Let
\begin{equation} \label{e:v_mollify}
    v_\eps^{(\delta)} := \qQ_\delta v_\eps\;,
\end{equation}
where we recall $\qQ_\delta$ is mollification by the heat kernel $q_{\delta^2}$. We also fix $p>\frac{1}{\kappa'}$ sufficiently large (such that Lemma~\ref{lem:smallnessOfLinearSolution} holds). Define the stopping times $\widetilde{\tau}_\eps^{(\delta)}$ and $\tau_\eps^{(\delta)}$ by
\begin{equation} \label{e:stoppingtime_defn}
    \begin{split}
    \widetilde{\tau}_\eps^{(\delta)} &:= \inf \Big\{t>0: \, \exists \, 1 \leq k \leq 2n+1 \text{ such that } \|X_\eps^{\diamond k}[t]\|_{\wW^{-\kappa',p}}> \eps^{\gamma-\frac{\kappa'}{3}}\\
    &\phantom{1111111}\text{ or } \| H_{k}\big( \qQ_\delta X_\eps[t]; \, \fC_\eps^{(\delta)} \big) - X_\eps^{\diamond k}[t]\|_{\cC^{-\frac{\eta}{2}}}> \delta^{\frac{\eta}{5}} \Big\}\;,\\
    \tau_\eps^{(\delta)} &:= \eps^{2\gamma+1} \, \widetilde{\tau}_\eps^{(\delta)}\;.
    \end{split}
\end{equation}
Then by Lemmas~\ref{lem:smallnessOfLinearSolution} and~\ref{lem:linear_final_approx}, for every $T >0$ and $N>0$, we have
\begin{equation} \label{e:stoppingtime_generic}
\PP(\tau_\eps^{(\delta)} \leq T) \lesssim \eps^N\;.
\end{equation}
The solution $u_\eps$ to \eqref{e:main_eqn} with $u_\eps[0] \in \mM$ has the following properties before the stopping time $\widetilde{\tau}_\eps^{(\delta)}$ (also $v_\eps^{(\delta)}$ before the stopping time $\tau_\eps^{(\delta)}$).

\begin{lem}\label{lem:solution_bound}
    We have the bounds
    \begin{equation*}
        \begin{split}
        &\dist_{\wW^{-\kappa',p}}(u_\eps[t], \mM) \leq \eps^{\gamma-\kappa'}\;,\\
        &\|w_\eps[t]\|_{\cC^\eta} \lesssim 1+ t^{-\eta}\;, \quad \|u_\eps^{\diamond (2n+1)}[t]\|_{\cC^{-\kappa}} \lesssim 1+ t^{-\frac{1}{4}}
        \end{split}
    \end{equation*}
    for all $t\in[0,\widetilde{\tau}_\eps^{(\delta)}]$ and all sufficiently small $\eps$. As a consequence (of the first bound for $u_\eps[t]$), we have
    \begin{equ}
        \dist_{\wW^{-\kappa',p}}(v_\eps[t], \mM) \leq \eps^{\gamma-\kappa'}
    \end{equ}
    for all $t\in[0,\tau_\eps^{(\delta)}]$ and all sufficiently small $\eps$.
\end{lem}
\begin{proof}
    By definition of the stopping time $\widetilde{\tau}_\eps^{(\delta)}$ and that $u_\eps[0] \in \mM$, $w_\eps[0] = u_\eps[0] - X_\eps[0]$ satisfies
    \begin{equation*}
        \dist_{\wW^{-\kappa', p}}(w_\eps[0], \mM) \leq \eps^{\gamma-\frac{\kappa'}{3}}
    \end{equation*}
    if $\widetilde{\tau}_\eps^{(\delta)}>0$. Hence, by Lemma~\ref{lem:PerturbationVersion} and the definition \eqref{e:stoppingtime_defn} again, we have
    \begin{equation*}
        \dist_{\wW^{-\kappa', p}}(u_\eps[t], \mM) \leq \|X_\eps[t]\|_{\wW^{-\kappa', p}} + \dist_{\wW^{-\kappa', p}}(w_\eps[t], \mM) \leq \eps^{\gamma - \kappa'}
    \end{equation*}
    for all $t \in [0, \widetilde{\tau}_\eps^{(\delta)}]$ and all sufficiently small $\eps$.

    We now turn to the bounds for $w_\eps[t]$ and $u_\eps^{\diamond (2n+1)}$. By the embedding $\wW^{-\kappa',p} \hookrightarrow \cC^{-\kappa}$, we have
    \begin{equation*}
        \max_{1 \leq k \leq 2n+1} \sup_{t \in [0, \widetilde{\tau}_\eps^{(\delta)}]} \|X_\eps^{\diamond k}[t]\|_{\cC^{-\kappa}} \lesssim \eps^{\gamma- \frac{\kappa'}{3}}\;.
    \end{equation*}
    Combining \eqref{e:closeness_longtime_1}, \eqref{e:D_Ceta_bound} and that $\tau_\eps^* > 2T$ in the proof of Lemma~\ref{lem:PerturbationVersion}, we have
    \begin{equation*}
    \dist_{\cC^\eta} \big( w_\eps[t], \mM \big) \lesssim 1+ t^{-\eta}
    \end{equation*}
    for all $t\in[0,\widetilde{\tau}_\eps^{(\delta)}]$ and all sufficiently small $\eps$. Together with $\|m\|_{\cC^\eta} \lesssim 1$, this implies the desired bound for $\|w_\eps[t]\|_{\cC^\eta}$ before $\widetilde{\tau}_\eps^{(\delta)}$. 
    
    As for $u_\eps^{\diamond (2n+1)}$, by the expression \eqref{e:u_Wick_expansion}, the desired bound then follows from that for $\|X_\eps^{\diamond (2n+1-k)}\|_{\cC^{-\kappa}}$ in and that for $\|w_\eps\|_{\cC^\eta}$. 
    
    Finally, the bound for $v_\eps$ up to time $\tau_\eps^{(\delta)}$ follows directly from the one for $u_\eps$ and a rescaling of time. 
\end{proof}

Recall from \eqref{e:v_mollify} that $v_\eps^{(\delta)} := \qQ_\delta v_\eps$. We have the following immediate bounds. 

\begin{lem} \label{le:v_eps_regularization}
    We have the bounds 
    \begin{equation*}
        \dist_{\wW^{\kappa',p}} \big( v_\eps^{(\delta)}, \mM \big) \lesssim \delta + \eps^{\gamma-\kappa'} \delta^{-2\kappa'}\;, \qquad \|v_\eps-v_\eps^{(\delta)}\|_{\cC^{-\kappa}}\lesssim \delta+\eps^{\gamma-\kappa'}
    \end{equation*}
    for all $t \in [0, \tau_\eps^{(\delta)}]$ and all sufficiently small $\eps$, and the proportionality constants are deterministic and independent of $\eps$. As a consequence, for the choice $\delta = \eps^{1000n(\gamma + 1)}$ above, we have
    \begin{equation*}
        \dist \big(v_\eps^{(\delta)}, \mM \big) \lesssim \delta + \eps^{\gamma-\kappa'} \delta^{-2\kappa'}\;, \qquad  \|v_\eps^{(\delta)}\|_{\lL^\infty} \lesssim 1 + \eps^{\gamma-\kappa'} \delta^{-2\kappa'} \lesssim 1
    \end{equation*}
    for all $t \in [0, \tau_\eps^{(\delta)}]$ and all sufficiently small $\eps$. 
\end{lem}
\begin{proof}
    For every $t$, there exists $\theta = \theta (v_\eps[t])$ such that
    \begin{equation*}
         \|v_\eps - m_{\theta}\|_{\wW^{-\kappa',p}} = \dist_{\wW^{-\kappa',p}} \big( v_\eps, \mM \big) \lesssim \eps^{\gamma-\kappa'}\;,
    \end{equation*}
    where we omitted the time $t$ for notational simplicity, and the second inequality above follows from Lemma~\ref{lem:solution_bound}. We have
    \begin{equation*}
        \begin{split}
        \dist_{\wW^{\kappa',p}} \big( v_\eps^{(\delta)}, \mM \big) \leq \|v_\eps^{(\delta)} - m_\theta\|_{\wW^{\kappa',p}} &\leq \left\| \qQ_{\delta} m_{\theta} - m_{\theta} \right\|_{\wW^{\kappa',p}}+\left\| \qQ_\delta \big( v_\eps - m_{\theta} \big) \right\|_{\wW^{\kappa',p}}\\
        &\lesssim \delta + \eps^{\gamma-\kappa'} \delta^{-2\kappa'}\;.
        \end{split}
    \end{equation*}
    Here, the last inequality follows from exponential decay of $m_\theta'$ for the first term (even if $m_\theta$ itself is not in $\lL^p$ for finite $p$), and smallness of $\|v_\eps - m_\theta\|_{\wW^{-\kappa',p}}$ for the second term. Similarly, we have
    \begin{equation*}
        \|v_\eps - v_\eps^{(\delta)}\|_{\cC^{-\kappa}} \leq \|\qQ_\delta m_\theta - m_\theta\|_{\cC^{-\kappa}} + \|v_\eps - m_{\theta}\|_{\cC^{-\kappa}} + \|\qQ_\delta \big( v_\eps - m_\theta \big)\|_{\cC^{-\kappa}} \lesssim \delta + \eps^{\gamma-\kappa'}\;.
    \end{equation*}
    Finally, the uniform boundedness of $\|v_\eps^{(\delta)}\|_{\lL^\infty}$ follows from the closeness of $v_\eps^{(\delta)}$ to $\mM$ in $\wW^{\kappa',p}$ norm and the embedding $\wW^{\kappa',p}\hookrightarrow \lL^\infty$. 
\end{proof}

We now start to derive the equation of motion for the separation point. We first need to provide a candidate for it. Define
\begin{equation} \label{e:xi_defn}
    \xi_\eps^{(\delta)}(t) := \sqrt{\eps} \, \zeta \big( v_\eps^{(\delta)}[t\wedge\tau_\eps^{(\delta)}] \big)\;, \quad \delta = \delta(\eps) = \eps^{1000n(\gamma+1)}\;.
\end{equation}
All the bounds below are uniform over $\eps\in[0,1]$ and $\delta \in(0,\frac{1}{2}]$ with the relationship $\delta = \eps^{1000n(\gamma+1)}$. It follows from \eqref{e:stoppingtime_generic}, Proposition~\ref{pr:LinftyEC} (applied to $v_\eps^{(\delta)}[t] = F^{0}\big(v_\eps^{(\delta)}[t]\big)$) and Lemma~\ref{le:v_eps_regularization} that there exists $C>0$ such that for every $N>0$, we have
\begin{equation*}
\PP \Big( \sup_{t \in [0,T]}
    \|v_\eps^{(\delta)}[t] - m_{\xi_\eps^{(\delta)}(t) / \sqrt{\eps}}\|_{\lL^\infty} > C\big(\delta+\eps^{\gamma-\kappa'}\delta^{-2\kappa'}\big)\Big) \lesssim \eps^N\;.
\end{equation*}
The proportionality constant depends on $N$ but is independent of $(\eps, \delta)$. 
Combining this with Lemma~\ref{le:v_eps_regularization} again, we deduce that the process $\xi_\eps^{(\delta)}$ defined in \eqref{e:xi_defn} satisfies
\begin{equation}\label{e:closeness_process_proof}
    \PP \Big( \sup_{t \in [0,T]}\|v_\eps[t] - m_{\xi_\eps^{(\delta)}(t) / \sqrt{\eps}}\|_{\cC^{-\kappa}} > \eps^{\gamma-\kappa}\Big) \lesssim \eps^N\;.
\end{equation}
It remains to derive the SDE for $\xi_\eps^{(\delta)}$ as $(\eps, \delta) \rightarrow (0,0)$ along the relation \eqref{e:xi_defn}. For this, we first note $v_\eps^{(\delta)}$ satisfies the equation
\begin{equation} \label{e:v_eps_delta}
    \d_t v_\eps^{(\delta)} = \eps^{-2\gamma-1} \big( \Delta v_\eps^{(\delta)} + f_n (v_\eps^{(\delta)}) + R_\eps^{(\delta)} \big) + \eps^{-\frac{1}{2}} \qQ_\delta \big( a_\eps \sqrt{\fD} \dot{W} \big)\;,
\end{equation}
where we write $f_n(v) = v - v^{2n+1}$, and
\begin{equation} \label{e:commutator_error}
    R_\eps^{(\delta)} = \big( v_\eps^{(\delta)} \big)^{2n+1} - \qQ_\delta(v_\eps^{\diamond(2n+1)})\;.
\end{equation}
Using It\^o's formula (with $v_\eps^{(\delta)}$ as a process in a small $\wW^{\kappa',p}$-neighborhood of $\mM$), we see $\xi_\eps^{(\delta)}(t) = \sqrt{\eps} \zeta \big( v_\eps^{(\delta)}[t\wedge \tau_\eps^{(\delta)}] \big)$ satisfies the SDE
\begin{equation} \label{e:Ito}
    {\rm d} \xi_\eps^{(\delta)}(t) = \big( \, b_{1,\eps}^{(\delta)} (t) + b_{2,\eps}^{(\delta)}(t) \, \big) \, {\rm d}t + {\rm d} \sigma_\eps^{(\delta)}(t)\;,\quad t\leq \tau_\eps^{(\delta)}\;,
\end{equation}
where the drift terms $b_{1,\eps}^{(\delta)}$ and $b_{2,\eps}^{(\delta)}$ are given by
\begin{equation} \label{e:drift_terms}
    \begin{split}
    b_{1,\eps}^{(\delta)}(t) &= \eps^{-2\gamma-\frac{1}{2}} \left\langle D \zeta \big( v_\eps^{(\delta)}\big), \, \Delta v_\eps^{(\delta)} + f_n(v_\eps^{(\delta)}) + R_\eps^{(\delta)} \, \right\rangle\;,\\
    b_{2,\eps}^{(\delta)}(t) &= \frac{1}{2 \sqrt{\eps}} \int_{\RR} \sqrt{\fD}^{\otimes 2} \Big( a_\eps^{\otimes 2} \cdot \qQ_\delta^{\otimes 2} \big( D^2 \zeta (v_\eps^{(\delta)}) \big) \Big) (y,y) \, {\rm d}y\;.
    \end{split}
\end{equation}
They come respectively from the drift and quadratic variation parts of \eqref{e:v_eps_delta}. The martingale part $\md \sigma_\eps^{(\delta)}(t)$ is given by
\begin{equation}\label{e:martingale}
    {\rm d} \sigma_\eps^{(\delta)}(t) = \left\langle \sqrt{\fD} \Big( a_\eps \cdot \qQ_\delta \big( D \zeta (v_\eps^{(\delta)}[t]) \big)  \Big), {\rm d} W(t) \right\rangle\;,\quad t\leq \tau_\eps^{(\delta)}\;. 
\end{equation}
Up to this stage, $D \zeta$ and $D^2 \zeta$ are Fr\'echet derivatives of $\zeta$ in a $\wW^{\kappa',p}$-neighborhood of $\mM$. They are well defined due to Lemma~\ref{lem:Dzeta} and the embedding $\wW^{\kappa',p} \hookrightarrow \lL^\infty$. Note that we use $\wW^{\kappa',p}$ instead of $\lL^\infty$ here since the latter is not separable and hence not suitable for It\^o's formula. 

On the other hand, by the same argument as in \cite[Section~2.3, Lemma A.4]{XZZ24}, the expressions \eqref{e:drift_terms} and \eqref{e:martingale} remain unchanged with $D \zeta$ and $D^2 \zeta$ changed to Fr\'echet derivatives with respect to the $\lL^\infty$ structure (instead of that of $\wW^{\kappa',p}$). In other words, the $\wW^{\kappa',p}$-space is used only to write down the SDE \eqref{e:Ito}. Once we have derived the SDE, only the $\lL^\infty$-bounds in Section~\ref{sec:deterministic} are needed to analyze the quantities in \eqref{e:drift_terms} and \eqref{e:martingale}. 

The convergence of $\xi_\eps^{(\delta)}$ to the desired SDE follows from the following two propositions. 

\begin{prop} \label{pr:martingale_convergence}
    There exists $\alpha_1 \in \RR$ and $\nu>0$ such that the quadratic variation process $[\sigma_\eps^{(\delta)}]$ of the martingale $\sigma_\eps^{(\delta)}$ satisfies
    \begin{equation*}
        \left| \frac{\rm d}{{\rm d} t} [\sigma_\eps^{(\delta)}](t) - \alpha_1^2 a^2 \big( \xi_\eps^{(\delta)}(t) \big) \right| \lesssim \eps^\nu+\delta^\nu
    \end{equation*}
    for all $t\in[0,\tau_\eps^{(\delta)}]$.
\end{prop}

\begin{prop} \label{pr:drift_convergence_final}
    There exists $\alpha_2 \in \RR$ and $\nu>0$ such that
    \begin{align*}
        \phantom{111}\Big| b_{1,\eps}^{(\delta)}(t) + b_{2,\eps}^{(\delta)}(t) - &\alpha_2 a \big( \xi_\eps^{(\delta)}(t) \big) a' \big( \xi_\eps^{(\delta)}(t) \big)  \Big|\lesssim \big( \delta \eps^{-\frac{1}{2}} + \eps^{\gamma-\kappa'-\frac{1}{2}} \delta^{-2\kappa'} + \eps^\nu \big)|\log\delta|\\
        &+\eps^{-2\gamma-\frac{1}{2}}\left( \delta^{\frac{\eta}{5}}(1+t^{-\frac{1}{4}}) + \eps^{4 \gamma} |\log \delta|^2 \right)+\delta^\nu
    \end{align*}
    for all $t\in[0,\tau_\eps^{(\delta)}]$.
\end{prop}

We postpone the proof of Propositions~\ref{pr:martingale_convergence} and~\ref{pr:drift_convergence_final} to later in this section. Assuming these two propositions, we are ready to prove the main theorem of this article. 

\begin{proof}[Proof of Theorem~\ref{thm:main}]
    The estimate \eqref{e:closeness_process_proof} directly implies \eqref{e:closeness_process}. Based on Propositions~\ref{pr:martingale_convergence}, ~\ref{pr:drift_convergence_final} and the choice of $\delta$ that $\delta= \eps^{1000n(\gamma+1)}$, the convergence in law of $\xi_\eps^{(\delta)}$ follows from a standard martingale method, see for example \cite[Lemma 8.3]{Fun95}. This concludes the proof of Theorem~\ref{thm:main}.
\end{proof}

\subsection{Convergence of the martingale term -- proof of Proposition~\ref{pr:martingale_convergence}}

We finally turn to the martingale term $\sigma_\eps^{(\delta)}$. By the expression \eqref{e:martingale}, the time derivative of its quadratic variation process is given by
\begin{equation} \label{e:qv_expression}
    \frac{\md}{\md t}[\sigma_\eps^{(\delta)}] = \Big\| \sqrt{\fD} \Big( a_\eps \cdot \qQ_\delta \big( D \zeta (v_\eps^{(\delta)}) \big) \Big) \Big\|_{\lL^2}^2\;,
\end{equation}
where we again omit the notation $t$ for simplicity. We decompose the term inside the $\sqrt{\fD}$ operation above as
\begin{align*}
    a_\eps \cdot \qQ_\delta  D \zeta \big(v_\eps^{(\delta)}\big) 
    =&a_\eps \cdot \qQ_\delta \Big( D \zeta \big(v_\eps^{(\delta)}\big)   - D \zeta \big(m_{\zeta(v_\eps^{(\delta)})}\big) \Big)\\
    &+ \qQ_\delta D \zeta \big(m_{\zeta(v_\eps^{(\delta)})}\big)\cdot\Big( a(\xi_\eps^{(\delta)}) + \err_{1,\eps}^{(\delta)}(y)\Big)\;,
\end{align*}
where 
\begin{equ}
    \err_{1,\eps}^{(\delta)}(y) =  a_\eps(y) - a\big( \xi_\eps^{(\delta)} \big)\;.
\end{equ}
By Lemmas~\ref{lem:Dzeta_continuity_aeps}, ~\ref{lem:Dzeta_taylor}, the translation property of $D \zeta$ (Assertion 2 in Lemma~\ref{lem:Dzeta}) and Remark~\ref{rmk:Dzeta_mollification}, we have
\begin{equation*}
    \left\| \sqrt{\fD} \Big( a_\eps \cdot \qQ_\delta \big( D \zeta (v_\eps^{(\delta)}) \big) \Big) - a (\xi_\eps^{(\delta)}) \cdot \sqrt{\fD} \qQ_\delta \Big( D \zeta \big( m_{\zeta(v_\eps^{(\delta)})} \big) \Big)  \right\|_{\lL^2}^2 \lesssim \eps^\nu
\end{equation*}
for some $\nu>0$. Using again the translation property for $D \zeta$, we see
\begin{equation*}
    \left\| \sqrt{\fD} \qQ_\delta \big( D \zeta ( m_{\zeta(v)} ) \big) \right\|_{\lL^2}^2 = \left\| \sqrt{\fD} \qQ_\delta \big( D \zeta (m) \big) \right\|_{\lL^2}^2
\end{equation*}
for every $v$ sufficiently close to $\mM$ in $\lL^\infty$-distance. Plugging the above bound and identity back into \eqref{e:qv_expression}, we get
\begin{equation*}
     \left| \frac{\md}{\md t} [\sigma_\eps^{(\delta)}](t) - a^2(\xi_\eps^{(\delta)}) \cdot \| \sqrt{\fD} \qQ_\delta \big(D\zeta(m) \big)\|_{\lL^2}^2 \right| \lesssim \eps^\nu
\end{equation*}
for all $t\in[0,\tau_\eps^{(\delta)}]$. Furthermore, since
\begin{equation*}
    \left\| \qQ_\delta \sqrt{\fD} \big( D \zeta(m) \big) - \sqrt{\fD} \big( D \zeta(m) \big) \right\|_{\lL^2} \lesssim \delta^\nu
\end{equation*}
for some $\nu>0$, we then deduce Proposition~\ref{pr:martingale_convergence} with
\begin{equation}\label{e:alpha1}
    \alpha_1 = \|\sqrt{\fD}D\zeta(m)\|_{\lL^2}\;.
\end{equation}

\subsection{Analysis of the drift $b_{1,\eps}^{(\delta)}$}

By orthogonality between $D \zeta$ and the deterministic gradient flow (see \eqref{e:zeta_magical_cancellation} below), we have
\begin{equation*}
    \left\langle D \zeta (v_\eps^{(\delta)}), \, \Delta v_\eps^{(\delta)} + f_n(v_\eps^{(\delta)}) \right\rangle = 0\;,
\end{equation*}
where we recall that $f_n(v) = v - v^{2n+1}$. Hence, the expression for $b_{1,\eps}^{(\delta)}$ is reduced to
\begin{equation} \label{e:b1_reduced}
    b_{1,\eps}^{(\delta)}(t) = \eps^{-2\gamma-\frac{1}{2}} \big\langle D \zeta \big( v_\eps^{(\delta)}[t]\big), \, R_\eps^{(\delta)}[t] \, \big\rangle\;,
\end{equation}
where we recall the expression of $R_\eps^{(\delta)}$ from \eqref{e:commutator_error}. It then remains to understand the action of $D \zeta (v_\eps^{(\delta)})$ on $R_\eps^{(\delta)}$. We have the following lemma regarding $R_\eps^{(\delta)}$.

\begin{lem} \label{le:drift_1_remainder}
We have the bound
    \begin{align*}
    \|R_\eps^{(\delta)}[t] - n(2n+1) \fC_\eps^{(\delta)} \big(v_\eps^{(\delta)}[t]\big)^{2n-1} \|_{\cC^{-\frac{\eta}{2}}}
    \lesssim \delta^{\frac{\eta}{5}}(1+t^{-\frac{1}{4}}) + \eps^{4 \gamma} |\log \delta|^2
\end{align*}
for all $t\in[0,\tau_\eps^{(\delta)}]$.
\end{lem}
\begin{proof}
It suffices to prove the bound for a time rescaled version of $R_\eps^{(\delta)}$ and then rescale time back. Write
\begin{equation*}
    \widetilde{R}_\eps^{(\delta)}[t] := R_\eps^{(\delta)}[\eps^{2\gamma+1} t] = \big( u_\eps^{(\delta)}[t]\big)^{2n+1} - \qQ_\delta \big( u_\eps^{\diamond (2n+1)}[t] \big)\;,
\end{equation*}
where we have written $u_\eps^{(\delta)} = \qQ_\delta u_\eps$ for simplicity. 

Recall from \eqref{e:variance_variable} that $\fC_\eps^{(\delta)}$ denotes the variance of $X_\eps$. By definition of Hermite polynomials, there exist constants $\lambda_{n,k}$ for $1 \leq k \leq n$ such that
\begin{equation*}
    \begin{split}
    H_{2n+1} \big( u_\eps^{(\delta)}\,; \, \fC_\eps^{(\delta)} \big) = \, &\big( u_\eps^{(\delta)} \big)^{2n+1} - n (2n+1) \fC_\eps^{(\delta)} \big( u_\eps^{(\delta)} \big)^{2n-1}\\
    &+ \sum_{k=2}^{n} \lambda_{n,k} \big( \fC_\eps^{(\delta)} \big)^{k} \, \big( u_\eps^{(\delta)} \big)^{2(n-k)+1}\;.
    \end{split}
\end{equation*}
Hence, we have
\begin{equation} \label{e:drift_1_remainder_expression}
    \begin{split}
    \widetilde{R}_\eps^{(\delta)} - &n(2n+1) \fC_\eps^{(\delta)} \, \big( u_\eps^{(\delta)} \big)^{2n-1} = \Big( H_{2n+1} \big( u_\eps^{(\delta)}; \, \fC_\eps^{(\delta)} \big) - u_\eps^{\diamond (2n+1)} \Big)\\
    &+ \Big( u_\eps^{\diamond (2n+1)} - \qQ_\delta \big( u_\eps^{\diamond (2n+1)} \big) \Big) - \sum_{k=2}^{n} \lambda_{n,k} \big( \fC_\eps^{(\delta)} \big)^k \big( u_\eps^{(\delta)} \big)^{2(n-k)+1}\;,
    \end{split}
\end{equation}
where we omitted the time $t$ in notation for simplicity. We control each of the three terms on the right hand side above. For the first one, since $u_\eps^{(\delta)} = \qQ_\delta X_\eps + \qQ_\delta w_\eps$ and that $\fC_\eps^{(\delta)}$ is the variance of $\qQ_\delta X_\eps$, we have
\begin{equation*}
    H_{2n+1} (u_\eps^{(\delta)}; \fC_\eps^{(\delta)}) = \sum_{k=0}^{2n+1} \begin{pmatrix} 2n+1 \\ k \end{pmatrix} H_{k}\big( \qQ_\delta X_\eps; \, \fC_\eps^{(\delta)} \big) (\qQ_\delta w_\eps)^{(2n+1-k)}\;.
\end{equation*}
Comparing it with the expression of $u_\eps^{\diamond (2n+1)}$ from \eqref{e:u_Wick_expansion} and using the definition of the stopping time $\widetilde{\tau}_\eps^{(\delta)}$ in \eqref{e:stoppingtime_defn} and Lemma~\ref{lem:solution_bound}, we deduce
\begin{equation*}
    \begin{split}
    \big\| H_{2n+1}(u_\eps^{(\delta)}; \fC_\eps^{(\delta)}) &- u_\eps^{\diamond (2n+1)} \big\|_{\cC^{-\frac{\eta}{2}}} \lesssim \sum_{k=0}^{2n+1} \Big( \| H_{k}\big( \qQ_\delta X_\eps; \, \fC_\eps^{(\delta)} \big) - X_\eps^{\diamond k} \|_{\cC^{-\frac{\eta}{2}}} \|\qQ_\delta w_\eps\|_{\cC^{\eta}}^{2n+1-k}\\
    &+ \|X_\eps^{\diamond k}\|_{\cC^{-\kappa}} \|\qQ_\delta w_\eps - w_\eps\|_{\cC^{\frac{\eta}{2}}} \big( \|\qQ_\delta w_\eps\|_{\cC^{\eta}}^{2n-k} + \|w_\eps\|_{\cC^{\eta}}^{2n-k} \big) \Big)\\
    &\phantom{1111111111111}\lesssim \delta^{\frac{\eta}{5}} \big(1 + t^{-\eta} \big) ^{2n+1} + \delta^{\frac{\eta}{2}} \big(1 + t^{-\eta} \big) ^{2n+1}\;,
    \end{split}
\end{equation*}
where the second line above is understood to be $0$ for $k=2n+1$. For the second term on the right hand side of \eqref{e:drift_1_remainder_expression}, by Lemma~\ref{lem:solution_bound}, we can bound it directly by
\begin{equation*}
    \left\| u_\eps^{\diamond (2n+1)} - \qQ_\delta \big( u_\eps^{\diamond (2n+1)} \big) \right\|_{\cC^{-\frac{\eta}{2}}} \lesssim \delta^{\frac{\eta}{2}-\kappa} \|u_\eps^{\diamond (2n+1)}\|_{\cC^{-\kappa}} \lesssim \delta^{\frac{\eta}{4}} \big( 1 + t^{-\frac{1}{4}} \big)\;.
\end{equation*}
For the third term, for each $k \geq 2$, we have
\begin{equation*}
    \left\|\big(\fC_\eps^{(\delta)}\big)^{k} \cdot \big(v_\eps^{(\delta)}\big)^{2(n-k)+1} \right\|_{\cC^{-\frac{\eta}{2}}} \lesssim \|\fC_\eps^{(\delta)}\|_{\lL^\infty}^k \|v_\eps^{(\delta)}\|_{\lL^\infty}^{2(n-k)+1} \lesssim \big( \eps^{2\gamma} |\log \delta| \big)^k\;,
\end{equation*}
where we have used Lemmas~\ref{le:renormalisation_approx} and~\ref{lem:Cdelta_divergence} to control $\|\fC_\eps^{(\delta)}\|_{\lL^\infty}$ by $\eps^{2\gamma} |\log \delta|$ and Lemma~\ref{le:v_eps_regularization} to control $\|v_\eps^{(\delta)}\|_{\lL^\infty}$ just by a constant. 

Plugging the above three bounds back into the right hand side of \eqref{e:drift_1_remainder_expression} and rescaling the time back, we have thus completed the proof of the lemma. 
\end{proof}

\begin{rmk} \label{rmk:Dzeta_reg_extra}
    Note that the term $R_\eps^{(\delta)}$ in the pairing \eqref{e:b1_reduced} is originally viewed as an element in $\lL^\infty$. But in order to use Lemma~\ref{le:drift_1_remainder} to replace $R_\eps^{(\delta)}$ by the quantity $n(2n+1) \fC_\eps^{(\delta)} (v_\eps^{(\delta)})^{2n-1}$, one needs extra regularity of the kernel of $D \zeta (v_\eps^{(\delta)})$ beyond being a linear functional on $\lL^\infty$. Fortunately, Lemma~\ref{le:Dzeta_reg_int} provides this extra regularity needed. 
\end{rmk}

According to Lemma~\ref{le:renormalisation_approx}, we decompose $\fC_\eps^{(\delta)}$ as
\begin{equation*}
    \fC_\eps^{(\delta)} = \eps^{2\gamma} C^{(\delta)} a_\eps^2 + \big( \fC_\eps^{(\delta)} - \eps^{2\gamma} C^{(\delta)} a_\eps^2 \big)\;,
\end{equation*}
where the constant $C^{(\delta)}$ is given in \eqref{e:C_delta}. Combined with Lemma~\ref{le:drift_1_remainder}, this suggests a decomposition of $R_\eps^{(\delta)}$ as
\begin{equation*}
    \begin{split}
    R_\eps^{(\delta)} = &n (2n+1) \, \eps^{2\gamma}C^{(\delta)} a_\eps^2 \big( v_\eps^{(\delta)} \big)^{2n-1} + n (2n+1) \, \big( \fC_\eps^{(\delta)} - \eps^{2\gamma}C^{(\delta)} a_\eps^2 \big) \cdot \big( v_\eps^{(\delta)} \big)^{2n-1}\\
    &+ \Big( R_\eps^{(\delta)} - n(2n+1) \fC_\eps^{(\delta)} \big(v_\eps^{(\delta)}\big)^{2n-1} \Big)\;.
    \end{split}
\end{equation*}
Hence, it is natural to define
\begin{equation} \label{e:b1_tilde}
    \widetilde{b}_{1,\eps}^{(\delta)}(t) := n(2n+1) \cdot \frac{C^{(\delta)}}{\sqrt{\eps}} \cdot \langle D \zeta \big( v_\eps^{(\delta)} \big), \, a_\eps^2 \cdot \big( v_\eps^{(\delta)} \big)^{2n-1} \rangle\;.
\end{equation}
By the bound \eqref{e:Dzeta_reg} (see Remark~\ref{rmk:Dzeta_reg_extra} above), Lemmas~\ref{le:drift_1_remainder} and~\ref{le:renormalisation_approx} and that $\|v_\eps^{(\delta)}\|_{\lL^\infty} \lesssim 1$, we have
\begin{equation} \label{e:b1_tilde_difference}
    \left| b_{1,\eps}^{(\delta)}(t) - \widetilde{b}_{1,\eps}^{(\delta)}(t) \right| \lesssim \sqrt{\eps} +\eps^{-2\gamma-\frac{1}{2}} \left( \delta^{\frac{\eta}{5}}(1+t^{-\frac{1}{4}}) + \eps^{4 \gamma} |\log \delta|^2 \right)
\end{equation}
for all $t\in[0,\tau_\eps^{(\delta)}]$. Hence, it remains to analyze $\widetilde{b}_{1,\eps}^{(\delta)}$. 

Taylor expanding $a_\eps^2$ at $\zeta \big( v_\eps^{(\delta)} \big)$ and recalling $\xi_\eps^{(\delta)} = \sqrt{\eps} \zeta \big( v_\eps^{(\delta)} \big)$, we have
\begin{equation} \label{e:b1_Taylor_a}
    a_\eps^2(y) = a^2(\sqrt{\eps} y) = a^2 \big( \xi_\eps^{(\delta)} \big) + 2 \sqrt{\eps} a \big( \xi_\eps^{(\delta)} \big) a' \big( \xi_\eps^{(\delta)} \big) \, \big( y - \zeta(v_\eps^{(\delta)}) \big) + \err_{2,\eps}^{(\delta)}\;,
\end{equation}
where the error term satisfies the pointwise bound
\begin{equation} \label{e:b1_Taylor_error}
    \big| \err_{2,\eps}^{(\delta)}(y) \big| \lesssim \eps \, \big| y - \zeta \big( v_\eps^{(\delta)} \big) \big|^2\;.
\end{equation}
According to the definition of $\widetilde{b}_{1,\eps}$ in \eqref{e:b1_tilde}, we need to multiply each term on the right hand side of \eqref{e:b1_Taylor_a} by $\big( v_\eps^{(\delta)} \big)^{2n-1}$ and then test with $D \zeta \big( v_\eps^{(\delta)} \big)$. For the ``constant term" $a^2 \big( \xi_\eps^{(\delta)} \big)$, using \eqref{e:Dzeta_test_v_odd_1} and Lemma~\ref{le:v_eps_regularization}, we have
\begin{equation} \label{e:b1_test_1}
    \begin{split}
    \big| \langle D \zeta \big(v_\eps^{(\delta)} \big), \, a^2\big(\xi_\eps^{(\delta)}\big) \cdot \big( v_\eps^{(\delta)} \big)^{2n-1} \rangle \big| &= a^2 (\xi_\eps^{(\delta)}) \cdot \big| \langle D \zeta \big(v_\eps^{(\delta)} \big), \, \big( v_\eps^{(\delta)} \big)^{2n-1} \rangle \big|\\
    &\lesssim \dist \big( v_\eps^{(\delta)}, \mM \big)\lesssim \delta + \eps^{\gamma-\kappa'} \delta^{-2\kappa'}\;.
    \end{split}
\end{equation}
For the second one, $a\big( \xi_\eps^{(\delta)} \big) a' \big( \xi_\eps^{(\delta)} \big)$ is now the ``constant term". Neglecting this ``constant term" for the moment, by the bound \eqref{e:Dzeta_test_v_odd_2}, the first expression of $C_0^*$ in \eqref{e:C0} and Lemma~\ref{le:v_eps_regularization}, we have
\begin{equation} \label{e:b1_test_2}
    \begin{split}
    &\phantom{111}\left| \int_{\RR} D \zeta \big( y; v_\eps^{(\delta)} \big) \, \big( y - \zeta (v_\eps^{(\delta)}) \big) \, \big( v_\eps^{(\delta)}(y) \big)^{2n-1} {\rm d}y - C_0^*  \right|\\
    &\lesssim \dist \big( v_\eps^{(\delta)}, \mM \big) \lesssim \delta + \eps^{\gamma-\kappa'} \delta^{-2\kappa'}\;.
    \end{split}
\end{equation}
For the error term, by Assertion 3 in Lemma~\ref{lem:Dzeta}, the bound \eqref{e:b1_Taylor_error} and that $\|v_\eps^{(\delta)}\|_{\lL^\infty} \lesssim 1$, we have
\begin{equation} \label{e:b1_test_3}
    \left| \langle D \zeta \big( v_\eps^{(\delta)} \big), \, \err_{2,\eps}^{(\delta)} \cdot \big( v_\eps^{(\delta)} \big)^{2n-1} \rangle \right| \lesssim \eps\;.
\end{equation}
Combining the bounds \eqref{e:b1_test_1}, \eqref{e:b1_test_2} and \eqref{e:b1_test_3} and noting the form of the Taylor expansion in \eqref{e:b1_Taylor_a}, we get
\begin{equation*}
    \left| \langle D \zeta \big( v_\eps^{(\delta)} \big), \, a_\eps^2 \cdot \big( v_\eps^{(\delta)} \big)^{2n-1}  \rangle - 2 C_0^* \sqrt{\eps} \, a(\xi_\eps^{(\delta)}) \, a'(\xi_\eps^{(\delta)}) \right| \lesssim \eps + \delta + \eps^{\gamma-\kappa'} \delta^{-2\kappa'}\;.
\end{equation*}
By the expression of $\widetilde{b}_{1,\eps}^{(\delta)}$ in \eqref{e:b1_tilde}, multiplying both sides above by $n(2n+1) \cdot \frac{C^{(\delta)}}{\sqrt{\eps}}$ and employing the bound of $C^{(\delta)}$ in Lemma~\ref{lem:Cdelta_divergence}, we obtain
\begin{equation*}
    \left| \tilde{b}_{1,\eps}^{(\delta)} - 2n (2n+1) C^{(\delta)} C_{0}^{*} a \big( \xi_\eps^{(\delta)} \big) a' \big( \xi_\eps^{(\delta)} \big) \right| \lesssim \Big( \frac{\eps + \delta}{\sqrt{\eps}} + \eps^{\gamma-\kappa'-\frac{1}{2}} \delta^{-2\kappa'} \Big) \cdot |\log\delta|
\end{equation*}
for all $t\in[0,\tau_\eps^{(\delta)}]$. Combing it with \eqref{e:b1_tilde_difference}, we get
\begin{equation} \label{e:b1_approx_final}
    \begin{split}
    \Big| b_{1,\eps}^{(\delta)} - 2n (2n+1) C^{(\delta)} &C_{0}^{*} a \big( \xi_\eps^{(\delta)} \big) a' \big( \xi_\eps^{(\delta)} \big) \Big| \lesssim \Big( \frac{\eps + \delta}{\sqrt{\eps}} + \eps^{\gamma-\kappa'-\frac{1}{2}} \delta^{-2\kappa'} \Big) \cdot |\log\delta|\\
    &+\eps^{-2\gamma-\frac{1}{2}}\left( \delta^{\frac{\eta}{5}} (1+t^{-\frac{1}{4}}) + \eps^{4 \gamma} |\log \delta|^2 \right)
    \end{split}
\end{equation}
for all $t\in[0,\tau_\eps^{(\delta)}]$.

\subsection{Analysis of the drift $b_{2,\eps}^{(\delta)}$}

We now turn to the drift term $b_{2,\eps}^{(\delta)}$ in \eqref{e:drift_terms}. We have
\begin{equation*}
    \begin{split}
    \qQ_{\delta}^{\otimes 2} &\big( D^2 \zeta \big( v_\eps^{(\delta)} \big) \big) \cdot a_\eps^{\otimes 2} = \qQ_{\delta}^{\otimes 2} \Big( D^2 \zeta \big( v_\eps^{(\delta)} \big) - D^2 \zeta \big( m_{\zeta(v_\eps^{(\delta)})} \big) \Big) \cdot a_\eps^{\otimes 2}\\
    &+ \Big[ \Big(  \qQ_{\delta}^{\otimes 2} D^2 \zeta \big( m_{\zeta(v_\eps^{(\delta)})} \big) \Big) \times \Big( a^2 \big( \xi_\eps^{(\delta)} \big)\\
    &+ \sqrt{\eps} a \big( \xi_\eps^{(\delta)} \big) a' \big( \xi_\eps^{(\delta)} \big) \big( y_1 - \zeta (v_\eps^{(\delta)}) + y_2 - \zeta (v_\eps^{(\delta)}) \big) + \err_{3,\eps}^{(\delta)} \Big) \Big]\;,
    \end{split}
\end{equation*}
where
\begin{equation*}
    \begin{split}
    \err_{3,\eps}^{(\delta)}(y_1, y_2) := &a_\eps(y_1) a_\eps(y_2) - a^2 \big( \xi_\eps^{(\delta)} \big)\\
    &- \sqrt{\eps} a \big( \xi_\eps^{(\delta)} \big) a' \big( \xi_\eps^{(\delta)} \big) \Big( y_1 - \zeta (v_\eps^{(\delta)}) + y_2 - \zeta (v_\eps^{(\delta)}) \Big)\;.
    \end{split}
\end{equation*}
Let
\begin{equation*}
    \begin{split}
    \widetilde{b}_{2,\eps}^{(\delta)}(t) &:= \frac{1}{2} \, a \big( \xi_\eps^{(\delta)}(t) \big) \, a' \big( \xi_\eps^{(\delta)}(t) \big)\\
    &\cdot \int_{\RR} \Big[ \sqrt{\fD}^{\otimes 2} \Big( \qQ_\delta^{\otimes 2} D^2 \zeta \big( m_{\zeta (v_\eps^{(\delta)})} \big) \cdot \big( y_1 + y_2 - 2 \zeta(v_\eps^{(\delta)}) \big) \Big) \Big](y,y) {\rm d}y\;,
    \end{split}
\end{equation*}
where the actions $\qQ_\delta^{\otimes 2}$ and $\sqrt{\fD}^{\otimes 2}$ are on the variables $y_1$ and $y_2$, and after that, one further restricts to $y_1 = y_2$ for integration. 

Note that by translation property of $D^2 \zeta$ (\eqref{e:Dzeta_translation} in Lemma~\ref{lem:Dzeta}) and anti-symmetry of $D^2 \zeta (m)$ (\eqref{e:D2zeta_odd} in Lemma~\ref{lem:Dzeta}), we have
\begin{equation*}
\int_{\RR} \sqrt{\fD}^{\otimes 2} \Big( \qQ_\delta^{\otimes 2} D^2 \zeta \big( m_{\zeta(v)} \big) \Big) \, (y,y) \,\md y = 0\;.
\end{equation*}
Hence, comparing the expression of $b_{2,\eps}^{(\delta)}$ in \eqref{e:drift_terms} and that of $\widetilde{b}_{2,\eps}^{(\delta)}$, we deduce from Lemmas~\ref{lem:D2zeta_continuity_aeps}, ~\ref{lem:D2zeta_taylor}, Assertion 2 in Lemma~\ref{lem:Dzeta} and Lemma~\ref{le:v_eps_regularization} that there exists $\nu>0$ such that
\begin{equation} \label{e:difference_b2_tilde}
    \left| b_{2,\eps}^{(\delta)} - \widetilde{b}_{2,\eps}^{(\delta)}  \right| \lesssim\Big( \frac{\delta}{\sqrt{\eps}} + \eps^{\gamma-\kappa'-\frac{1}{2}} \delta^{-2\kappa'} + \eps^\nu \Big)|\log \delta| \;.
\end{equation}
By the translation property (\eqref{e:Dzeta_translation} in Lemma~\ref{lem:Dzeta}), we see $\widetilde{b}_{2,\eps}^{(\delta)}$ has the expression
\begin{equation*}
    \widetilde{b}_{2,\eps}^{(\delta)}(t) = \frac{1}{2} \, a \big( \xi_\eps^{(\delta)}(t) \big) \, a' \big( \xi_\eps^{(\delta)}(t) \big)  \int_{\RR}  \Big( \sqrt{\fD}^{\otimes 2} \big( \qQ_\delta^{\otimes 2} D^2 \zeta(m) \cdot (y_1 + y_2) \big) \Big)(y,y) {\rm d}y\;.
\end{equation*}
Now, applying Lemma~\ref{lem:D2zeta_divergence_rate} to the expression above and then combining \eqref{e:difference_b2_tilde}, we deduce there exist $\widetilde{\alpha}_2 \in \RR$ and $\nu>0$ such that
\begin{equation} \label{e:b2_approx_final}
\begin{aligned}
    &\left| b_{2,\eps}^{(\delta)} + \Big( \frac{n (2n+1) C_0^*}{2 \pi^2} \cdot |\log \delta| + \widetilde{\alpha}_2 \Big) a \big( \xi_\eps^{(\delta)} \big) a' \big( \xi_\eps^{(\delta)} \big) \right| \\
    \lesssim &\Big( \frac{\delta}{\sqrt{\eps}} + \eps^{\gamma-\kappa'-\frac{1}{2}} \delta^{-2\kappa'} + \eps^\nu \Big) |\log \delta| + \delta^\nu
\end{aligned}
\end{equation}
for all $t\in[0,\tau_\eps^{(\delta)}]$. 

\subsection{Cancellations of the drifts -- proof of Proposition~\ref{pr:drift_convergence_final}}

Combining \eqref{e:b1_approx_final} and \eqref{e:b2_approx_final}, we see there exists $\nu>0$ such that
\begin{equation*}
    \begin{split}
    \big| b_{1,\eps}^{(\delta)}(t) + b_{2,\eps}^{(\delta)}(t) - &\alpha_2^{(\delta)} a \big( \xi_\eps^{(\delta)}(t) \big) a' \big( \xi_\eps^{(\delta)}(t) \big) \big| \lesssim \Big( \frac{\delta}{\sqrt{\eps}} + \eps^{\gamma-\kappa'-\frac{1}{2}} \delta^{-2\kappa'} + \eps^\nu \Big)|\log\delta|\\
    &+\eps^{-2\gamma-\frac{1}{2}}\left( \delta^{\frac{\eta}{5}}(1+t^{-\frac{1}{4}}) + \eps^{4 \gamma} |\log \delta|^2 \right) + \delta^\nu\;,
    \end{split}
\end{equation*}
where  
\begin{equation*}
    \alpha_2^{(\delta)} = 2 n (2n+1) C_0^* \Big( C^{(\delta)} - \frac{|\log \delta|}{4 \pi^2} \Big) - \widetilde{\alpha}_2\;.
\end{equation*}
Applying Lemma~\ref{lem:Cdelta_divergence}, we see $\alpha_2^{(\delta)}$ converges to
\begin{equation}\label{e:alpha2}
    \alpha_2 = 2n(2n+1)C_0^*\left(\alpha-\frac{\log\mu}{8\pi^2}\right)-\widetilde{\alpha}_2
\end{equation}
as $\delta \rightarrow 0$. This completes the proof of Proposition~\ref{pr:drift_convergence_final}.

\section{Analysis of the deterministic flow}
\label{sec:deterministic}

The aim of this section is to prove the deterministic statements that have been used as ingredients in the proofs of Propositions~\ref{pr:martingale_convergence} and~\ref{pr:drift_convergence_final}. Most of them concern quantities of the form $\sqrt{\fD}^{\otimes k} \big( a_\eps^{\otimes k} \, D^k \zeta (v_\eps^{(\delta)}) \big)$ for $k=1,2$, possibly also with regularization by $\qQ_\delta^{\otimes k}$. 

We first give in Section~\ref{sec:zeta_prelim} preliminaries on (twice) Fr\'echet differentiability of $\zeta$ in $\lL^\infty$ neighborhood of $\mM$, together with some properties of the kernels of $D^k \zeta$ and $p_t^{(k)} = D^k F^t$ for $k=1,2$. Then in Section~\ref{sec:deterministic_statements}, we list the precise statements to be proved about the quantities $\sqrt{\fD}^{\otimes k} \big( a_\eps^{\otimes k} \, D^k \zeta (v_\eps^{(\delta)}) \big)$ and their variants. We need to have good understanding on its stability (when $v_\eps^{(\delta)}$ is close to $\mM$ in $\lL^\infty$) as well as the error terms from Taylor expansions of $a_\eps$. 

We choose to work with $\lL^\infty$-norm since it allows us to directly use existing $\lL^\infty$-based estimates from \cite{XZZ24}. Note that even though $v_\eps$ itself is not function valued, all our analysis are based on the intermediate object $v_\eps^{(\delta)}$. The validity of using $\lL^\infty$-norm to measure the distance between $v_\eps^{(\delta)}$ and $\mM$ is guaranteed by Lemma~\ref{le:v_eps_regularization}. These are the norms used in the proof of Propositions~\ref{pr:martingale_convergence} and~\ref{pr:drift_convergence_final}. 

Hence, in the rest of this section, we always assume $v \in \vV_\beta$ in the sense that
\begin{equation*}
    \dist (v, \mM) := \dist_{\lL^\infty} (v, \mM) < \beta\;,
\end{equation*}
and let $\zeta: \vV_\beta \rightarrow \RR$ denote the limiting functional given by Proposition~\ref{pr:LinftyEC}.

\subsection{Kernel and limiting functional of the deterministic flow}
\label{sec:zeta_prelim}

We collect known results on properties of $D^k \zeta$ and the kernel $p_t^{(k)}(\cdot  \,, \cdot\,, v)$ for $k=1,2$. The statements below are mostly from \cite{Fun95} and \cite{XZZ24}. 

For $v \in \lL^\infty$  and $y \in \RR$, let $p_t(y, \cdot\,; v)$ be the solution to
\begin{equation*} 
    \d_t \, p_t (y, \cdot\,; v) = \Delta \, p_t (y, \cdot\,; v) + f_n'\big( F^t(v) \big) \, p_t (y, \cdot\,; v)\;, \qquad p_0(y, \cdot\,; v) = \delta_y\;.
\end{equation*}
For $0\leq s\leq t$, we also write
\begin{equation*}
    p_{s,t}(y,z;v):=p_{t-s}\big(y,z;F^s(v)\big)\;.
\end{equation*}
For $v\in\lL^\infty$ and $(y_1,y_2)\in\RR^2$, let $p_t^{(2)}(y_1,y_2, \cdot\,; v)$ be the solution $u$ to the equation
\begin{equation*} 
    \d_t u = \Delta u + f_n' \big( F^t(v) \big) u + f_n''\big(F^t(v)\big) \, p_t(y_1,\cdot\,;v) \, p_t(y_2,\cdot\,;v)\;, \quad u[0] = 0\;.
\end{equation*}
Then by Duhamel's principle, we have
\begin{equation*}
    p_{t}^{(2)}(y_1,y_2,z;v) = \int_0^t\int_\RR p_{s,t}(z',z;v)f_n''\big(F^s(v)\big)(z') \prod_{j=1}^{2} p_{s}(y_j,z';v) \,\md z' \,\md s\;.
\end{equation*}
It is shown in \cite[Lemma~9.7]{Fun95} and \cite[Proposition~3.19]{XZZ24} that $p_t(v)$ and $p_t^{(2)}(v)$ are the kernels of $DF^t(v)$ and $D^2F^t(v)$ for $v \in \vV_\beta$ in the sense that
\begin{equation*}
    \begin{split}
    \bracket{D F^t(v), \varphi}(z) &= \int_{\RR} p_{t}(y, z; v) \varphi(y) \,\md \vec{y}\;,\\
    \bracket{D^2 F^t(v), (\varphi_1, \varphi_2)}(z) &= \iint\limits_{\RR^2} p_{t}^{(2)}(y_1, y_2, z; v) \varphi_1(y_1) \varphi_2 (y_2)\, \md y_1 \,\md y_2
    \end{split}
\end{equation*}
for $\varphi, \varphi_1, \varphi_2 \in \lL^\infty$.

\begin{lem}\label{lem:Dzeta}
    The limiting point functional $\zeta$ satisfies the following pointwise bounds and properties: 
    \begin{enumerate}
        \item $\zeta$ is twice Fr\'echet differentiable in $\vV_\beta$, and its Fr\'echet derivatives $D^k\zeta(v)$ has a kernel for $k=1,2$ and $v\in\vV_{\beta}$ (still denoted by $D^k\zeta(v)$). Moreover, the kernels have the expressions
        \begin{equation*} 
            D\zeta(y;v) = -\frac{1}{\|m'\|_{\lL^2}^2}\Big(m_{\zeta(v)}'(y)+\int_0^{+\infty}\int_{\RR} m_{\zeta(v)}'(z) \mathbf{r}_t(z;v)p_t(y,z;v) \,\md z\,\md t\Big)\;,
        \end{equation*}
        \begin{equs}
        D^2\zeta(y_1,y_2;v)= -\frac{1}{\|m'\|_{\lL^2}^2} &\int_{0}^{+\infty} \int_\RR\Big(  f_n'' \big(F^t (v)(z) \big) \prod_{j=1}^{2} p_t(y_j, z; v) \\
        &+  \r_t(z;v) \, p_t^{(2)}(y_1, y_2, z; v)\Big)m_{\zeta(v)}'(z)\,\md z\,\md t\;,
        \end{equs}
        where
        \begin{equation*}
        \mathbf{r}_t(v):= f_n'\big(F^t(v)\big) - f_n'(m_{\zeta(v)})
        \end{equation*}
        satisfies
        \begin{equation} \label{e:r_decay}
        \Vert \mathbf{r}_t(v) \Vert_{\lL^\infty}\leq C e^{-ct} \dist (v, \mM)
        \end{equation}
        for some $c, C>0$, uniformly over $v \in \vV_\beta$ and $t\geq 0$. 
        
        \item The derivatives of $\zeta$ satisfy the translation property
        \begin{equation} \label{e:Dzeta_translation}
            \begin{split}
            D \zeta \big(y; v(\cdot - z) \big) &= D \zeta \big(y-z; v\big)\;,\\
            D^2 \zeta \big(y_1, y_2 ; v(\cdot - z) \big) &= D \zeta \big(y_1-z, y_2 - z; v\big)\;.
            \end{split}
        \end{equation}
        Furthermore, $D^2 \zeta (m)$ is anti-symmetric in the sense that
        \begin{equation} \label{e:D2zeta_odd}
            D^2 \zeta (y_1, y_2; m) = - D^2 \zeta (-y_1, -y_2; m)\;.
        \end{equation}
        
        \item There exists $\lambda>0$ such that
        \begin{align}
            |D\zeta(y;v)| &\lesssim e^{-\lambda|y-\zeta(v)|}\label{e:Dzeta_bound}\;,\\
            |D^2\zeta(y_1,y_2;v)| &\lesssim e^{-\lambda(|y_1-\zeta(v)|+|y_2-\zeta(v)|)}\;,\label{e:D2zeta_bound}\\
            |D\zeta(y;v) - D\zeta(y;m_{\zeta(v)})| &\lesssim e^{-\lambda|y-\zeta(v)|} \dist(v,\mM)\;,\label{e:Dzeta_difference}\\
            |D^2\zeta(y_1,y_2;v)-D^2\zeta(y_1,y_2;m_{\zeta(v)})| &\lesssim e^{-\lambda(|y_1-\zeta(v)|+|y_2-\zeta(v)|)}\dist(v,\mM)\label{e:D2zeta_difference}
        \end{align}
        uniformly over $v \in \vV_{\beta}$ and $y, y_1, y_2 \in \RR$. 
    \end{enumerate}
\end{lem}
\begin{proof}
    This is the content of \cite[Lemma~3.11, Lemma~3.13, Proposition~4.2]{XZZ24}, and it is also studied in \cite{Fun95}.
\end{proof}

\begin{lem} \label{le:Dzeta_reg_int}
    For every $\alpha \in [0,2)$ and $v\in\vV_{\beta}$, we have $D\zeta(v)\in\wW^{\alpha,1}$ with
    \begin{equ}\label{e:Dzeta_reg}
        \sup_{v\in\vV_\beta} \|D\zeta(v)\|_{\wW^{\alpha,1}}<+\infty\;.
    \end{equ}
    In addition, for every $v\in \vV_\beta\cap\cC^{5/2}$, we have the identity 
    \begin{equation}\label{e:zeta_magical_cancellation}
        \bracket{D\zeta(v), \Delta v+f_n(v)} = 0\;.
    \end{equation}
    Moreover, we have the bounds
    \begin{equ}\label{e:Dzeta_test_v_odd_1}
        |\bracket{D\zeta(v),v^{2n-1}}|\lesssim \dist(v,\mM)\;,
    \end{equ}
    \begin{equ}\label{e:Dzeta_test_v_odd_2}
        |\bracket{D\zeta(v),\big(\cdot-\zeta(v)\big)v^{2n-1}}-\bracket{D\zeta(m_{\zeta(v)}),\big(\cdot-\zeta(v)\big)m_{\zeta(v)}^{2n-1}}|\lesssim \dist(v,\mM)\;.
    \end{equ}
\end{lem}
\begin{proof}
The bounds \eqref{e:Dzeta_test_v_odd_1} and \eqref{e:Dzeta_test_v_odd_2} follow from Lemma~\ref{lem:Dzeta} and the fact that $\bracket{m', m^{2n-1}} = 0$. The identity \eqref{e:zeta_magical_cancellation} was proved in \cite[Theorem 2.3]{XZZ24}, and was also studied in \cite[Theorem 7.4]{Fun95}. The $\lL^2$ version of \eqref{e:Dzeta_reg} was established in \cite[Theorem 7.4]{Fun95}, but the same argument applies in our $\lL^\infty$ setting. For completeness, we reproduce the proof here. By the translation property \eqref{e:Dzeta_translation} of $D\zeta$, we can without loss of generality assume $v\in\vV_{\beta,0}$. For $t>0$ and $v\in\vV_{\beta,0}$, define $T_{t;v}$ by 
    \begin{equation*}
        (T_{t;v} \phi)(y) := \int_\RR p_{t}(y,z;v)\phi(z)\,\md z\;.
    \end{equation*} 
We first prove that for every $\alpha\in[0,2)$, we have
    \begin{equation}\label{e:T_reg}
        \Vert T_{t;v}\Vert_{\lL^1\to\wW^{\alpha,1}}\lesssim 1+ t^{-\alpha/2}\;.
    \end{equation}
    uniformly over all $v\in\vV_{\beta}$  and $t>0$. We first consider the case $v=m$, in which $T_{t;v}=e^{-t\aA}$. Then the conclusion follows from Proposition~\ref{prop:LinftySG}.

    Now for general $v\in\vV_{\beta,0}$, by Duhamel's formula we have
    \begin{equation*}
        T_{t;v} = e^{-t\aA} + \int_0^t T_{s;v} \r_s(v) e^{-(t-s)\aA}\,\md s\;.
    \end{equation*}
So \eqref{e:T_reg} follows from the decay property of $\r_s(v)$ in \eqref{e:r_decay} and Gr\"onwall's inequality. Now for \eqref{e:Dzeta_reg}, by Lemma~\ref{lem:Dzeta}, we have
\begin{equation*} 
    D\zeta(y;v) = -\frac{1}{\|m'\|_{\lL^2}^2}\Big(m'(y)+\int_0^{+\infty} T_{t;v} \big(m'\mathbf{r}_t(v)\big)(y)\,\md t\Big)\;.
\end{equation*}
Then the conclusion follows from \eqref{e:T_reg}, Lemma~\ref{lem:statationary_exponential_decay} and the decay property of $\r_t(v)$ in \eqref{e:r_decay}. 
\end{proof}

For every $\lambda \in \RR$ and $p\in[1,+\infty]$, we define the weighted $\lL^p$-norm by
\begin{equation*}
    \|g\|_{\lL^p_{\lambda}} := \Vert e^{\lambda |\cdot|}g\Vert_{\lL^p}=\left(\int_{\RR} e^{p\lambda |y|} |g(y)|^p \,\md y\right)^{\frac{1}{p}}\;,
\end{equation*}
where $p=+\infty$ corresponds to the supremum norm. We have the following collection of statements regarding the kernel $p$. 

\begin{lem}\label{lem:pt}
    There exists $c>0$ and $\overline{\lambda}>0$ such that for every $v\in \vV_{\beta,0}$, the kernel $p_t(y,z;v)$ can be decomposed as
    \begin{equation}\label{e:decomposition1}
        p_t(y,z;v) = -D\zeta(y;v) \, m'(z) + \widetilde{p}_{t}(y,z;v)\;,
    \end{equation}
    where for every $\lambda\in(0,\overline{\lambda}]$, the ``remainder" $\widetilde{p}_{t}$ satisfies the bound
    \begin{equation}\label{e:decomposition1-2}
        \Vert \widetilde{p}_{t}(y,\cdot\,;v)\Vert_{\lL^\infty_{-\lambda}}\lesssim t^{-\frac{1}{2}}e^{-c t}e^{-\frac{1}{2}\lambda |y|}
    \end{equation}
    uniformly over all $t>0$ and $v\in\vV_{\beta,0}$. As a consequence, for every $\lambda\in(0,\overline{\lambda}]$, we have
    \begin{equation}\label{e:bound1-1}
        \Vert p_t(y,\cdot\,;v)\Vert_{\lL^\infty_{-\lambda}}\lesssim (1+t^{-\frac{1}{2}})e^{-\frac{1}{2}\lambda |y|}
    \end{equation}
    uniformly over all $t>0$ and $v\in\vV_{\beta,0}$. Moreover, for every $\lambda\in(0,\overline{\lambda}]$, we have
    \begin{equation}\label{e:difference1-2}
        \Vert \widetilde{p}_{t}(y,\cdot\,;v)-\widetilde{p}_{t}(y,\cdot\,;m)\Vert_{\lL^\infty_{-\lambda}}\lesssim e^{-ct}e^{-\frac{1}{2}\lambda |y|}\dist(v,\mM)
    \end{equation}
    uniformly over all $t>0$ and $v\in\vV_{\beta,0}$. Furthermore, $\widetilde{p}_{t}$ has the expression
    \begin{equation} \label{e:tilde_p_v_m}
    \begin{aligned}
        \widetilde{p}_{t}(y,z;v)=&\widetilde{p}_{t}(y,z;m)+ \int_0^t\int_{\RR} \widetilde{p}_{t-s}(z',z;m)\mathbf{r}_s(z';v)p_s(y,z';v) \,\md z' \,\md s \\&+\Big(\int_t^\infty\int_{\RR} D\zeta(z';m)\mathbf{r}_s(z';v)p_s(y,z';v) \,\md z' \,\md s\Big)m'(z)\;.
    \end{aligned}
    \end{equation}
    Finally, for every $\lambda\in(0,\overline{\lambda}]$, we have the bound
    \begin{equation}\label{e:bound1-2}
         \Vert p_{s,t}(y,\cdot\,;v)\Vert_{\lL^1_{-\lambda}}\lesssim e^{-\frac{1}{4}\lambda |y|}
    \end{equation}
    for all $0\leq s<t$.
\end{lem}
\begin{proof}
    This is the content of \cite[Lemma~3.11, Proposition~4.2]{XZZ24}, except for \eqref{e:difference1-2} and \eqref{e:bound1-2}. Actually, \eqref{e:difference1-2} follows directly from \eqref{e:tilde_p_v_m}, \eqref{e:decomposition1-2}, \eqref{e:bound1-1}, \eqref{e:r_decay} and \eqref{e:Dzeta_bound}. Moreover, \eqref{e:bound1-2} follows from the Gaussian bound (see for example \cite[Lemma 9.3]{Fun95}) for the case $t-s\leq 1$ and \eqref{e:bound1-1} as well as the embedding $\lL^\infty_{-\lambda/2}\hookrightarrow \lL^1_{-\lambda}$ for the case $t-s\geq 1$ respectively.
\end{proof}

\subsection{Overview of the statements}
\label{sec:deterministic_statements}

The following two lemmas are needed in the proof of Proposition~\ref{pr:martingale_convergence} to control the error terms. 

\begin{lem} \label{lem:Dzeta_continuity_aeps}
We have
\begin{equ}
    \left\| \sqrt{\fD} \Big( a_\eps \, \big(D\zeta(v) - D\zeta(m_{\zeta(v)}) \big) \Big) \right\|_{\lL^2}\lesssim \dist(v,\mM)
\end{equ}
uniformly over $\eps\in[0,1]$ and $v\in\vV_\beta$. 
\end{lem}

\begin{lem}\label{lem:Dzeta_taylor}
For every $\nu\in(0,\frac{1}{2})$, we have
    \begin{equ}
        \left\| \sqrt{\fD} \Big( D \zeta (m) \cdot \big( a(x_0 + \sqrt{\eps} \cdot) - a(x_0) \big)  \Big) \right\|_{\lL^2}\lesssim\eps^\nu
    \end{equ}
    for all $\eps\in[0,1]$ and $x_0\in\RR$.
\end{lem}

\begin{rmk} \label{rmk:Dzeta_mollification}
    Since $\qQ_\delta D \zeta$ satisfies the same bounds as those needed for $D \zeta$ in Lemmas~\ref{lem:Dzeta_continuity_aeps} and~\ref{lem:Dzeta_taylor}, the conclusions of these two lemmas still hold with $D \zeta$ replaced by $\qQ_\delta D \zeta$, and the proportionality constants in the bounds are also independent of $\delta$. 
\end{rmk}

The following two lemmas are used to control small error terms from the drift $b_{2,\eps}^{(\delta)}$. 

\begin{lem}\label{lem:D2zeta_continuity_aeps}
We have
    \begin{align*}
        &\Big|\int_\RR \sqrt{\fD}^{\otimes 2}\Big(\qQ_\delta^{\otimes2}\big(D^2\zeta(v)-D^2\zeta(m_{\zeta(v)})\big)\cdot a_\eps^{\otimes 2}\Big)(y,y)\,\md y\Big|\lesssim |\log\delta|\dist(v,\mM)
    \end{align*}
    uniformly over $\eps\in[0,1]$, $\delta\in(0,\frac{1}{2}]$ and $v\in\vV_\beta$.
\end{lem}

\begin{lem}\label{lem:D2zeta_taylor}
For every $\nu\in(0,\frac{1}{2})$, we have
    \begin{align*}
        \bigg|\int_\RR \sqrt{\fD}^{\otimes 2}\Big(\qQ_\delta^{\otimes2}&D^2\zeta(m)\times \big( a(x_0 + \sqrt{\eps}y_1)a(x_0 + \sqrt{\eps}y_2)\\
        &-(a(x_0)^2 + a(x_0)a'(x_0)\sqrt{\eps}(y_1+y_2))\big)\Big)(y,y)\,\md y\bigg|\lesssim\eps^{\frac{1}{2}+\nu}|\log\delta|
    \end{align*}
    for all $\eps\in[0,1]$, $\delta\in(0,\frac{1}{2}]$ and $x_0\in\RR$.
\end{lem}

The following lemma is used in identifying the exact asymptotic behaviour of $\widetilde{b}_{2,\eps}^{(\delta)}$, which is the main part of the drift $b_{2,\eps}^{(\delta)}$. It allows us to identify exact cancellations in $b_{1,\eps}^{(\delta)} + b_{2,\eps}^{(\delta)}$. 

\begin{lem}\label{lem:D2zeta_divergence_rate}
    There exists $\widetilde{\alpha}_2\in\RR$ and $\nu>0$ such that 
    \begin{equation*}
    \begin{split}
    &\phantom{111}\frac{1}{2}\int \sqrt{\fD}^{\otimes 2}\Big(\qQ_\delta^{\otimes2}\big(D^2\zeta(m)\big)\times (y_1+y_2)\Big)(y,y)\,\md y\\
    &= - \Big(\frac{n(2n+1)C^*_0}{2\pi^2}|\log \delta| + \widetilde{\alpha}_2 \Big) + \oO(\delta^\nu)\;.
    \end{split}
\end{equation*}
\end{lem}

\subsection{Fractional derivative on the kernel and $D \zeta$}
\label{sec:deterministic_kernel}

We start with some technical lemmas.

\begin{lem}\label{lem:sqrtD_decay}
    For every $\varphi\in\sS$, we have
    \begin{equ}
        |(\sqrt{\fD}\varphi)(x)| \lesssim \bracket{x}^{-\frac{3}{2}}
    \end{equ}
    for all $x \in \RR$. 
\end{lem}
\begin{proof}
    With the expression of $\sqrt{\fD} \varphi$ from Lemma~\ref{lem:leibniz_1}, the conclusion follows by splitting the integral into the domains $\{|z| \leq 1\}$ and $\{|z|>1\}$. 
\end{proof}

\begin{lem}\label{lem:techninal_1}
For every $\alpha > 0$, we have
    \begin{equ}
        \int_\RR \bracket{y-z}^{-\alpha}e^{-\lambda|z|}\,\md z\lesssim\bracket{y}^{-\alpha}
    \end{equ}
    for all $y\in\RR$. Moreover, if $\alpha\in(0,1)$, then $\bracket{y-z}^{-\alpha}$ can be replaced by $|y-z|^{-\alpha}$.
\end{lem}
\begin{proof}
Splitting the integral into $\{|z| \leq \frac{|y|}{2}\}$ and $\{|z| > \frac{|y|}{2}\}$, the bounds then follow easily. 
\end{proof}

\begin{lem}\label{lem:technical_2}
For every $\lambda >0$, we have
    \begin{equ}
        \int_{0}^1\int_\RR e^{-\lambda|z|}|\sqrt{\fD_y}q_{s+\delta^2}(y,z)|^2\,\md z \,\md s \lesssim |\log \delta|\bracket{y}^{-3}
    \end{equ}
    for all $\delta\in(0,\frac{1}{2}]$ and $y\in\RR$.
\end{lem}
\begin{proof}
By scaling property of the heat kernel and Lemma~\ref{lem:sqrtD_decay}, we have the pointwise bound
\begin{equation} \label{e:frac_derivative_heat}
    \big| \sqrt{\fD_y} q_{s+\delta^2}(y,z) \big| \lesssim 
    \big( s+\delta^2 \big)^{-\frac{3}{4}} \cdot \left\langle \frac{y-z}{\sqrt{s+\delta^2}} \right\rangle^{-\frac{3}{2}}\;.
\end{equation}
Hence, we have
\begin{equation*}
    \int_{0}^1\int_\RR e^{-\lambda|z|}|\sqrt{\fD_y}q_{s+\delta^2}(y,z)|^2\,\md z \,\md s \lesssim \int_{0}^{1} \int_\RR \frac{e^{-\lambda |z|}}{(s+\delta^2)^{\frac{3}{2}}} \cdot \left\langle \frac{y-z}{\sqrt{s+\delta^2}} \right\rangle^{-3} \,\md z \,\md s\;.
\end{equation*}
We distinguish the two cases whether $|y| \leq 1$ or $|y|>1$. If $|y| \leq 1$, we simply bound $e^{-\lambda |z|}$ by $1$. Then, successively integrating out $z$ and $s$ variable, we see that the integral can be bounded by a constant multiple of $|\log \delta|$ for $|y| \leq 1$. 

For $|y| \geq 1$, we split the domain of integration in $z$ variable into $\{|z| \leq \frac{|y|}{2}\}$ and $\{|z| > \frac{|y|}{2}\}$. In the former, we have
\begin{equation*}
    \left\langle \frac{y-z}{\sqrt{s+\delta^2}} \right\rangle^{-3} \lesssim \left\langle \frac{y}{\sqrt{s+\delta^2}} \right\rangle^{-3} \lesssim (s+\delta^2)^{\frac{3}{2}} \, \langle y \rangle^{-3} \;, \qquad |z| \leq \frac{|y|}{2}\;, \; |y|>1\;,
\end{equation*}
and hence the corresponding integral is bounded by $\langle y \rangle^{-3}$. 

In the domain $\{|z| > \frac{|y|}{2}\}$, we have $e^{-\lambda |z|} \leq e^{-\frac{\lambda |y|}{2}}$, and hence the corresponding integral can be bounded by
\begin{equation*}
    e^{-\frac{\lambda |y|}{2}} \int_{0}^{1} (s+\delta^2)^{-\frac{3}{2}} \int_\RR \left\langle \frac{y-z}{\sqrt{s+\delta^2}} \right\rangle^{-3} \,\md z \,\md s \lesssim |\log \delta| \, \langle y \rangle^{-3}\;.
\end{equation*}
The assertion of the lemma then follows by combining all of the above cases. 
\end{proof}

To analyze the behaviour of $p_t$ near $t=0$, we subtract the heat kernel $q_t$ from $p_t$. More precisely, we write
\begin{equation*}
    q_t (y,z) := q_t (y-z)
\end{equation*}
where $q_t$ is the standard heat kernel given in \eqref{e:heat_kernel}. We write $r_t:=p_t-q_t$ for the difference, which turns out to be well-behaved near $t=0$.

\begin{lem}\label{lem:regularity_rt}
For every $\lambda>0$, we have
\begin{equation*}
    |r_t(y,z;v)| \lesssim e^{-\lambda|y|+\lambda|z|}\;.
\end{equation*}
Furthermore, the half derivative in $y$ of $r_t$ satisfies
\begin{equation*}
    |\sqrt{\fD_y}r_t(y,z;v)| \lesssim \bracket{y-z}^{-\frac{3}{2}}\;.
\end{equation*}
Both bounds are uniform over $v \in \vV_\beta$, $t \in (0,1]$ and $y,z \in \RR$. 
\end{lem}
\begin{proof}
By definition of $r_t$, we have
\begin{equation} \label{e:r_Duhamel}
    r_t(y,\cdot\,;v) = \int_0^t e^{(t-s)\Delta}\Big( f_n'\big(F^s(v)\big) \, \big(q_s(y,\cdot)+r_s(y,\cdot\,;v)\big)\Big)\,\md s\;.
\end{equation}
Taking $\lL^\infty_{-\lambda}$ norm and noticing that $\|q_s(y,\cdot)\|_{\lL^\infty_{-\lambda}}\lesssim s^{-\frac{1}{2}}e^{-\lambda|y|}$ as well as $\|f_n'(F^s(v))\|_{\lL^\infty} \lesssim 1$, one gets
\begin{equ}
\|r_t(y,\cdot\,;v)\|_{\lL^\infty_{-\lambda}} \lesssim \int_0^t \big( s^{-\frac{1}{2}}e^{-\lambda|y|}+\|r_s(y,\cdot\,;v)\|_{\lL^\infty_{-\lambda}} \big)\,\md s\;.
\end{equ}
Gr\"onwall's inequality then implies
\begin{equation*}
    \|r_t \big(y, \cdot; v\big)\|_{\lL_{-\lambda}^\infty} \lesssim \sqrt{t} e^{ct}e^{-\lambda|y|} \lesssim e^{-\lambda|y|} \quad \text{for} \; t \leq 1\;,
\end{equation*}
which proves the first claim. For the second one, writing out the integration kernel of $e^{(t-s)\Delta}$ in \eqref{e:r_Duhamel}, applying the operator $\sqrt{\fD_y}$ to both sides, and using $\|f_n'(F^s(v))\|_{\lL^\infty} \lesssim 1$, we get 
\begin{equation*}
    \big| \sqrt{\fD_y} r_t(y,z;v) \big| \lesssim \int_{0}^{t} \int_{\RR} | q_{t-s}(z,z') | \cdot \Big( \big| \sqrt{\fD_y} q_s (y,z') \big| + \big| \sqrt{\fD_y} r_s(y,z';v) \big| \Big) \,\md z' \,\md s\;.
\end{equation*}
We now fix $y \in \RR$, and write
\begin{equation*}
    A(t) := \sup_{z \in \RR} \Big( \bracket{y-z}^{\frac{3}{2}} \, \left| \sqrt{\fD_y} r_t(y,z;v) \right| \Big)\;.
\end{equation*}
By \eqref{e:frac_derivative_heat}, we have
\begin{equation} \label{e:frac_derivative_heat_small_time}
    \big| \sqrt{\fD_y} q_s(y,z') \big| \lesssim s^{-\frac{3}{4}} \left\langle \frac{y-z'}{\sqrt{s}} \right\rangle^{-\frac{3}{2}} \lesssim s^{-\frac{3}{4}} \langle y-z' \rangle^{-\frac{3}{2}}\;, \quad s \leq 1\;.
\end{equation}
Plugging \eqref{e:frac_derivative_heat_small_time} back into the inequality for $|\sqrt{\fD_y} r_t(y,z;v)|$, we get
\begin{equation*}
    \big| \sqrt{\fD_y} r_t(y,z;v) \big| \lesssim \int_{0}^{t} \big( s^{-\frac{3}{4}} + A(s) \big) \bigg( \int_{\RR} |q_{t-s}(z,z')| \cdot \bracket{y-z'}^{-\frac{3}{2}} \,\md z' \bigg) \,\md s\;.
\end{equation*}
Performing a change of variable $\frac{z-z'}{\sqrt{t-s}} \mapsto z'$ and integrating out the $z'$ variable, we obtain
\begin{equation*}
    \big| \sqrt{\fD_y} r_t(y,z;v) \big| \lesssim \bracket{y-z}^{-\frac{3}{2}} \int_{0}^{t} \big( s^{-\frac{3}{4}} + A(s) \big) \,\md s\;.
\end{equation*}
The conclusion then follows from the definition of $A$ and Gr\"onwall's inequality. 
\end{proof}

\begin{lem}\label{lem:regularity_pt}
For every $\lambda>0$, we have
\begin{equ}
    |\sqrt{\fD_y}p_t(y,z;v)|\lesssim (1+t^{-\frac{3}{4}}) \, \bracket{y}^{-\frac{3}{2}} \, e^{\lambda|z|}
\end{equ}
uniformly over $v\in\vV_{\beta,0}$, $t>0$ and $y,z\in\RR$.
\end{lem}
\begin{proof}
For $t \leq 1$, by \eqref{e:frac_derivative_heat_small_time} and Lemma~\ref{lem:regularity_rt}, we have
\begin{equation*}
    \big| \sqrt{\fD_y} p_t(y,z;v) \big| \lesssim t^{-\frac{3}{4}} \langle y-z \rangle^{-\frac{3}{2}} \lesssim t^{-\frac{3}{4}} \langle y \rangle^{-\frac{3}{2}} e^{\lambda |z|}\;.
\end{equation*}
For $t\geq 1$, we use semigroup property to write
\begin{equ}
    p_t(y,z;v) = \int_\RR p_\frac{1}{2}(y,z';v) \, p_{\frac{1}{2},t}(z',z;v)\,\md z'\;.
\end{equ}
By Lemma~\ref{lem:regularity_rt}, we have
\begin{equ}
    |\sqrt{\fD_y}p_{\frac{1}{2}}(y,z';v)|\lesssim\bracket{y-z'}^{-\frac{3}{2}}\;.
\end{equ}
Then by \eqref{e:bound1-1} (replacing $p_{0,t}$ by $p_{\frac{1}{2},t}$ and the factor $t$ in the bound by $t-\frac{1}{2}$) and Lemma~\ref{lem:techninal_1}, we get
\begin{equ}
    |\sqrt{\fD_y}p_t(y,z;v)|\lesssim\int_\RR \bracket{y-z'}^{-\frac{3}{2}}e^{-\frac{\lambda}{2}|z'|+\lambda|z|}\,\md z'\lesssim \bracket{y}^{-\frac{3}{2}}e^{\lambda|z|}\;.
\end{equ}
This completes the proof.
\end{proof}

\begin{lem}\label{lem:regularity_phatt1}
There exists $c>0$ such that for every $\lambda>0$, we have
\begin{equ}\label{e:regularity_Dzeta_bound}
    |\sqrt{\fD} D\zeta(y;v)|\lesssim \bracket{y}^{-\frac{3}{2}}\;,
\end{equ}
\begin{equ}
    |\sqrt{\fD_y} \widetilde{p}_{t}(y,z;v)|\lesssim t^{-\frac{3}{4}}e^{-ct}\bracket{y}^{-\frac{3}{2}}e^{\lambda|z|}\;,
\end{equ}
\begin{equ}\label{e:regularity_Dzeta_difference}
    |\sqrt{\fD} \big(D\zeta(y;v)-D\zeta(y;m)\big)|\lesssim \bracket{y}^{-\frac{3}{2}}\dist(v,\mM)\;,
\end{equ}
\begin{equ}
    |\sqrt{\fD_y} \big(\widetilde{p}_{t}(y,z;v)-\widetilde{p}_{t}(y,z;m)\big)|\lesssim e^{-ct}\bracket{y}^{-\frac{3}{2}}e^{\lambda|z|}\dist(v,\mM)
\end{equ}
uniformly over $v\in\vV_{\beta,0}$, $t>0$ and $y,z\in\RR$.
\end{lem}
\begin{proof}
We first consider the two bounds for $D \zeta$. Since $D \zeta (m)$ is a constant multiple of $m'$ and hence a Schwartz function, by Lemma~\ref{lem:sqrtD_decay}, we have
\begin{equation*}
    | \sqrt{\fD} D \zeta (y;m)| \lesssim \bracket{y}^{-\frac{3}{2}}\;.
\end{equation*}
To get the bounds with $m$ replaced by $v \in \vV_{\beta,0}$, we note that $D \zeta (v)$ and $D \zeta (m)$ are related by
\begin{equation*}
    D\zeta(y;v)=D\zeta(y;m)+\int_0^{\infty}\int_{\RR} D\zeta(z';m)\mathbf{r}_s(z';v)p_s(y,z';v) \,\md z' \,\md s\;.
\end{equation*}
Applying $\sqrt{\fD_y}$ on both sides above, the desired bound for $\sqrt{\fD} \big( D \zeta(v) - D \zeta(m) \big)$ then follows directly from Assertion 3 in Lemma~\ref{lem:Dzeta}, the bound \eqref{e:r_decay} and Lemma~\ref{lem:regularity_pt}. 

For the two bounds with $\widetilde{p}_t$, we first note that by semigroup property, we have
\begin{equation*}
    \widetilde{p}_{t}(y,z;m) = \int_\RR  p_s(y,z';m) \, \widetilde{p}_{t-s}(z',z;m) \,\md z'\;,
\end{equation*}
which holds for every $s \in [0,t]$. Applying $\sqrt{\fD_y}$ to both sides above and using Lemma~\ref{lem:regularity_pt} and the bound \eqref{e:decomposition1-2}, we get
\begin{equation*}
    |\sqrt{\fD} \widetilde{p}_{t}(y,z;m)| \lesssim t^{-\frac{3}{4}} e^{-ct} \bracket{y}^{-\frac{3}{2}} e^{\lambda |z|}\;.
\end{equation*}
To get the bounds with $m$ replaced by $v$ as well as their difference, we note that they are related by the identity \eqref{e:tilde_p_v_m}. The desired bounds then follow similarly. 
\end{proof}

\begin{lem}\label{lem:regularity_pt2}
For every $\lambda>0$ sufficiently small, we have
\begin{equation*}
    \begin{split}
    \|\sqrt{\fD_{y_1}}p_t^{(2)}(y_1,y_2,z;v) \|_{\lL^1_{-\lambda}(z)} &\lesssim (1+t)\bracket{y_1}^{-\frac{3}{2}}e^{-\frac{\lambda}{32}|y_2|}\;,\\
    \left\|\Big( \sqrt{\fD}^{\otimes 2}\qQ_\delta^{\otimes2}p_t^{(2)}(\cdot,\cdot,z;v) \Big) (y,y) \right\|_{\lL^1_{-\lambda}(z)} &\lesssim (|\log \delta|+t)\bracket{y}^{-3}
    \end{split}
\end{equation*}
uniformly over $v\in\vV_{\beta,0}$, $\delta\in(0,\frac{1}{2}]$, $t>0$ and $y_1,y_2,y\in\RR$.
\end{lem}
\begin{proof}
Recall that $p^{(2)}_t$ has the expression
\begin{equation} \label{e:expression_p2}
    p_{t}^{(2)}(y_1,y_2,z;v) = \int_0^t\int_\RR p_{s,t}(z',z;v) f_n'' \big((F^s v)(z')\big) \prod_{j=1}^{2} p_{s}(y_j,z';v) \,\md z' \,\md s\;.
\end{equation}
Since $\|f_n''\big(F^s(v) \big)\|_{\lL^\infty} \lesssim 1$, we have
\begin{equation*}
    \begin{split}
    &\phantom{111}\big\| \sqrt{\fD_{y_1}}p_t^{(2)}(y_1,y_2,z;v) \big\|_{\lL^1_{-\lambda}(z)}\\
    &\lesssim \int_{0}^{t} \int_\RR |\sqrt{\fD_{y_1}} p_s (y_1, z';v)| \cdot |p_s(y_2, z';v)| \, \Big(\int_{\RR} e^{-\lambda |z|} |p_{s,t}(z',z;v)| \md z \Big) \,\md z' \,\md s\;.
    \end{split}
\end{equation*}
Using Lemma~\ref{lem:regularity_pt} to control $|\sqrt{\fD_{y_1}} p_s (y_1, z';v)|$, and the bound \eqref{e:bound1-2} to successively control the integrations over $z$ and $z'$, we obtain
\begin{equation*}
    \begin{split}
    &\phantom{111}\big\| \sqrt{\fD_{y_1}}p_t^{(2)}(y_1,y_2,z;v) \big\|_{\lL^1_{-\lambda}(z)}\\
    &\lesssim \bracket{y_1}^{-\frac{3}{2}} \int_{0}^{t} (1 + s^{-\frac{3}{4}}) \bigg( \int_\RR e^{\frac{\lambda}{8} \cdot |z'|} \cdot |p_s(y_2,z';v)| \cdot e^{-\frac{\lambda}{4} |z'|} \,\md z' \bigg) \,\md s\\
    &\lesssim (1+t) \, \bracket{y_1}^{-\frac{3}{2}} \, e^{-\frac{\lambda}{32} \cdot |y_2|}\;.
    \end{split}
\end{equation*}
For the second bound, using the expression \eqref{e:expression_p2} and that $\|f_n''\big( F^s(v) \big)\|_{\lL^\infty} \lesssim 1$, we get the pointwise bound
\begin{equation*}
    \big| \sqrt{\fD}^{\otimes 2}\qQ_\delta^{\otimes2}p_t^{(2)}(\cdot,\cdot,z;v) (y,y) \big| \lesssim \int_{0}^{t} \int_\RR \big| p_{s,t}(z',z;v) \big| \,  \Big( \sqrt{\fD_y} \qQ_{\delta,y} p_s(y,z';v) \Big)^2 \,\md z' \,\md s\;,
\end{equation*}
where $\qQ_{\delta,y}$ means the mollification is in the $y$ variable. Multiplying $e^{-\lambda |z|}$ on both sides above and applying \eqref{e:bound1-2} to integrate $z$ out, we get
\begin{equation*}
    \|\sqrt{\fD}^{\otimes 2}\qQ_\delta^{\otimes2}p_t^{(2)}(\cdot,\cdot,z;v) (y,y)\|_{\lL^1_{-\lambda}(z)} \lesssim \int_{0}^{t} \int_\RR e^{-\frac{\lambda}{4} \cdot |z'|} \Big( \sqrt{\fD_y} \qQ_{\delta,y} p_s(y,z';v) \Big)^2  \,\md z' \,\md s\;.
\end{equation*}
We split the integration in time into $s \in [0, t \wedge 1]$ and $s \in [t \wedge 1, t]$. For the first part, by Lemmas~\ref{lem:regularity_rt} and~\ref{lem:technical_2}, it is bounded by
\begin{equation*}
    \begin{split}
    &\phantom{111}\int_{0}^{t\wedge1} \int_\RR e^{-\frac{\lambda}{4} \cdot |z'|} \Big(|\sqrt{\fD_y}q_{s+\delta^2}(y,z')| + \bracket{y-z'}^{-\frac{3}{2}} \Big)^2 \, \md z' \,\md s\\
    &\lesssim \bracket{y}^{-3} + \int_{0}^1 \int_\RR e^{-\frac{\lambda}{4} \cdot |z'|}|\sqrt{\fD_y}q_{s+\delta^2}(y,z')|^2\,\md z' \,\md s \lesssim |\log \delta| \cdot \bracket{y}^{-3}\;.
    \end{split}
\end{equation*}
For the second part, by Lemma~\ref{lem:regularity_pt}, the quantity without the mollification is bounded by
\begin{align*}
    \int_{t\wedge1}^t\int_\RR e^{-\frac{\lambda}{4} \cdot |z'|} \Big(\bracket{y}^{-\frac{3}{2}} \, e^{\frac{\lambda}{10}\cdot|z'|}\Big)^2\,\md z' \,\md s \lesssim t \, \bracket{y}^{-3}\;.
\end{align*}
Mollification with $\qQ_\delta$ does not change the above bound. The claim then follows. 
\end{proof}

\subsection{Cross fractional derivatives on $D^2 \zeta$}
\label{sec:deterministic_zeta}
The main goal of this subsection is to prove the following three lemmas. 
\begin{lem}\label{lem:regularity_D2zeta_half}
There exists $\lambda >0$ such that
    \begin{align*}
        \Big|\sqrt{\fD_{y_1}} D^2\zeta(v)(y_1,y_2)\Big|&\lesssim \bracket{y_1-\zeta(v)}^{-\frac{3}{2}}e^{-\lambda|y_2-\zeta(v)|}\;,\\
        \Big|\sqrt{\fD_{y_1}}\big(D^2\zeta(v) - D^2\zeta(m_{\zeta(v)})\big)(y_1,y_2)\Big|&\lesssim \bracket{y_1-\zeta(v)}^{-\frac{3}{2}}e^{-\lambda|y_2-\zeta(v)|}\dist(v,\mM)
    \end{align*}
    uniformly over $v\in\vV_\beta$ and $y_1,y_2\in\RR$.
\end{lem}

\begin{lem}\label{lem:D2zeta_continuity}
We have
    \begin{align*}
    \int_\RR \Big|\sqrt{\fD}^{\otimes 2}\Big(\qQ_\delta^{\otimes2}\big(D^2\zeta(v)-D^2\zeta(m_{\zeta(v)})\big)\Big)\Big|(y,y) \,\md y&\lesssim|\log\delta|\dist(v,\mM)
    \end{align*}
    uniformly over $v\in\vV_\beta$ and $\delta\in(0,\frac{1}{2}]$.
\end{lem}

\begin{lem}\label{lem:D2zeta_decay}
We have
    \begin{equ}
        \Big|\sqrt{\fD}^{\otimes 2}\qQ_\delta^{\otimes2}\big(D^2\zeta(m)\big)\Big|(y,y)\lesssim |\log \delta| \cdot \bracket{y}^{-3}
    \end{equ}
    for all $\delta\in(0,\frac{1}{2}]$ and $y\in\RR$.
\end{lem}

By the translation invariance (Assertion 2 in Lemma~\ref{lem:Dzeta}), we can without loss of generality assume that $\zeta(v) = 0 $. Recall that
\begin{equs}
    D^2\zeta(\vec{y};v)= - \frac{1}{\|m'\|_{\lL^2}^2} \bigg[&\int_0^\infty \int_\RR m'(z) \bigg( f_n''\big((F^t v)(z)\big) \prod_{j=1}^{2} p_t(y_j,z;v) \\
    &+ \Big( f_n'\big( (F^t v)(z) \big) -f_n'\big(m(z) \big) \Big) \, p_t^{(2)}(\vec{y},z;v) \bigg)\,\md z\,\md t \bigg]\;.
\end{equs}
When $v=m$, the expression is reduced to the first line only. We split the time integration into $[0,1]$ and $[1,+\infty)$, and write
\begin{equation} \label{e:D2zeta_decomposition}
    D^2 \zeta(\vec{y}; m) = \frac{2n(2n+1)}{\|m'\|_{\lL^2}^2} \Big( h_1 (\vec{y}) + h_2 (\vec{y}) \Big)\;,
\end{equation}
where
\begin{equation} \label{e:D2zeta_h}
    \begin{split}
    h_1 (\vec{y}) &= \int_{0}^{1} \int_\RR m'(z) \, m^{2n-1}(z) \, \prod_{j=1}^{2} p_t (y_j, z; m) \, \md z \, \md t\;,\\
    h_2 (\vec{y}) &= \int_{1}^{+\infty} \int_\RR m'(z) \, m^{2n-1}(z) \, \prod_{j=1}^{2} p_t (y_j, z; m) \, \md z \, \md t\;.
    \end{split}
\end{equation}
We also decompose $D^2 \zeta (v) - D^2 \zeta(m)$ into several parts as
\begin{equation*}
    D^2 \zeta (\vec{y}; v) - D^2 \zeta (\vec{y}; m) = \frac{(2n+1)}{\|m'\|_{\lL^2}^2} \bigg( 2n \sum_{k=1}^{3} g_k(\vec{y}) + g_4(\vec{y}) \bigg)\;,
\end{equation*}
where the functions $g_k$'s are given by
\begin{equation*}
    \begin{split}
    g_1 (\vec{y}) &= \int_{0}^{+\infty} \int_\RR m'(z) \, m^{2n-1}(z) \, \Big( \prod_{j=1}^{2} p_t (y_j, z; v) - \prod_{j=1}^{2} p_t (y_j, z; m) \Big) \, \md z \, \md t\;,\\
    g_2 (\vec{y}) &= \int_{0}^{1} \int_\RR m'(z) \, \Big( \big(F^t(v)\big)^{2n-1}(z) - m^{2n-1}(z) \Big) \, \prod_{j=1}^{2} p_t(y_j, z; v) \, \md z \, \md t\;,\\
    g_3 (\vec{y}) &= \int_{1}^{+\infty} \int_\RR m'(z) \, \Big( \big(F^t(v)\big)^{2n-1}(z) - m^{2n-1}(z) \Big) \, \prod_{j=1}^{2} p_t(y_j, z; v) \, \md z \, \md t\;,\\
    g_4 (\vec{y}) &= \int_{0}^{+\infty} \int_\RR m'(z) \, \Big( \big(F^t(v)\big)^{2n}(z) - m^{2n}(z) \Big) \, p_t^{(2)}(\vec{y}, z;v) \, \md z \, \md t\;.
    \end{split}
\end{equation*}
We have omitted the dependence of the $g_k$'s on $v$ for notational simplicity. We now start to control each of the $g_k$'s and $h_k$'s.

\begin{lem}\label{lem:regularity_g1}
There exists $\lambda >0$ such that
\begin{equ}
    \big|\sqrt{\fD_{y_1}}g_1\big|(y_1,y_2) \lesssim \bracket{y_1}^{-\frac{3}{2}}e^{-\lambda|y_2|}\dist(v,\mM)\;.
\end{equ}
We also have the bound
    \begin{equ}
        \big|\sqrt{\fD}^{\otimes 2}g_1\big|(y,y) \lesssim \bracket{y}^{-3}\dist(v,\mM)\;.
    \end{equ}
Both bounds are uniformly over $v\in\vV_{\beta,0}$ and $y_1,y_2,y\in\RR$.
\end{lem}
\begin{proof}
By the decomposition \eqref{e:decomposition1}, we have
\begin{equation} \label{e:product_pt_expansion}
    \begin{split}
    \prod_{j=1}^{2} p_t(y_j, &z;v) = \, \big( m'(z) \big)^2 \, \prod_{j=1}^{2} D \zeta (y_j; v)\\
    &- m'(z) \, \sum_{i \neq j} D \zeta (y_i; v) \, \widetilde{p}_t(y_j, z; v) + \prod_{j=1}^{2} \widetilde{p}_t (y_j,z;v)\;.
    \end{split}
\end{equation}
Note that $g_1$ involves integration in time over the whole positive real line, but the first term on the right hand side above has no decay in time. However, since $m$ is odd, integrating $z$ variable out implies the contribution to $g_1(\vec{y})$ from the first term on the right hand side of \eqref{e:product_pt_expansion} (together with $m' \cdot m^{2n-1}$) is actually $0$. Hence, we have
\begin{equation*}
    \begin{split}
    g_1(\vec{y}) =  &\int_{0}^{+\infty} \int_\RR m'(z) \, m^{2n-1}(z) \, \bigg[ \Big( \prod_{j=1}^{2} \widetilde{p}_t(y_j,z;v) - \prod_{j=1}^{2} \widetilde{p}_t (y_j,z;m) \Big)\\
    &- m'(z) \, \sum_{i \neq j} \Big( D \zeta(y_i;v) \widetilde{p}_t (y_j,z;v) - D \zeta(y_i;m) \widetilde{p}_t (y_j,z;m) \Big) \bigg] \md z \,\md t\;.
    \end{split}
\end{equation*}
Applying $\sqrt{\fD_{y_1}}$ and $\sqrt{\fD}^{\otimes 2}$ to both sides above respectively (and taking $y_1 = y_2 = y$ in the second case), the desired bounds for $\sqrt{\fD_{y_1}} g_1$ and $(\sqrt{\fD}^{\otimes 2} g_1)(y,y)$ then follows from the relevant bounds in Lemmas~\ref{lem:pt} and~\ref{lem:regularity_phatt1}. 
\end{proof}

\begin{lem}\label{lem:regularity_g2}
There exists $\lambda >0$ such that
\begin{align*}
    \big|\sqrt{\fD_{y_1}}g_2\big|(y_1,y_2) &\lesssim \bracket{y_1}^{-\frac{3}{2}}e^{-\lambda|y_2|}\dist(v,\mM)\;,\\
    \big|\sqrt{\fD_{y_1}} h_1\big|(y_1,y_2) &\lesssim \bracket{y_1}^{-\frac{3}{2}}e^{-\lambda|y_2|}\;.  
\end{align*}
We also have the bounds
\begin{align*}
    \big|\sqrt{\fD}^{\otimes 2}\qQ_\delta^{\otimes 2} g_2\big|(y,y) &\lesssim |\log \delta|\bracket{y}^{-3}\dist(v,\mM)\;,\\
    \big|\sqrt{\fD}^{\otimes 2}\qQ_\delta^{\otimes 2} h_1\big|(y,y) &\lesssim |\log \delta|\bracket{y}^{-3}\;.
\end{align*}
All proportionality constants above are independent of $v\in\vV_{\beta,0}$, $\delta\in(0,\frac{1}{2}]$ and $y_1,y_2,y\in\RR$.
\end{lem}
\begin{proof}
Using exponential decay of $m'$ and that
\begin{equation}\label{e:Ftv_closeness}
    \big\|\big(F^t(v)\big)^{2n-1} - m^{2n-1}\big\|_{\lL^\infty} \lesssim \dist (v, \mM)\;,
\end{equation}
the two bounds for $\sqrt{\fD_{y_1}} g_2$ and $\sqrt{\fD_{y_1}} h_1$ then follow directly from \eqref{e:bound1-2} and Lemma~\ref{lem:regularity_pt}.

The proof for the other two bounds are similar to the argument in Lemma~\ref{lem:regularity_pt2}. We give sketch for the one with $g_2$. Using exponential decay of $m'$ and \eqref{e:Ftv_closeness} again, we get
\begin{equation*}
    |(\sqrt{\fD}^{\otimes 2} \qQ_\delta^{\otimes 2} g_2)(y,y)| \lesssim \dist (v, \mM) \int_{0}^{1} \int_{\RR} e^{-\lambda |z|} \big| \sqrt{\dD_y} \qQ_{\delta,y} p_t(y,z;v) \big|^2\, \md z \,\md t\;.
\end{equation*}
Since
\begin{equation*}
    \sqrt{\fD_y} \qQ_{\delta,y} p_t(y,z;v) = \sqrt{\fD_y} q_{t+\delta^2} (y,z) + \qQ_{\delta,y} \sqrt{\fD_y} r_t(y,z;v)\;, 
\end{equation*}
it follows from Lemmas~\ref{lem:regularity_rt} and~\ref{lem:technical_2} that
\begin{equation*}
    \begin{split}
    |(\sqrt{\fD}^{\otimes 2} \qQ_\delta^{\otimes 2} g_2)(y,y)| &\lesssim \dist (v, \mM) \int_{0}^{1} \int_\RR e^{-\lambda |z|} \, \Big( \big| \sqrt{\fD_y} q_{s+\delta^2}(y,z) \big|^2 + \bracket{y-z}^{-3} \Big) \,\md z\, \md t\\
    &\lesssim |\log \delta| \, \bracket{y}^{-3} \dist (v, \mM)\;.
    \end{split}
\end{equation*}
The proof for the bound for $\sqrt{\fD}^{\otimes 2}\qQ_\delta^{\otimes 2} h_1$ is essentially the same except that one does not have the $\dist (v, \mM)$ factor. This completes the proof of the lemma. 
\end{proof}

\begin{lem}\label{lem:regularity_g3}
There exists $\lambda >0$ such that
\begin{equs}
    \big|\sqrt{\fD_{y_1}}g_3\big|(y_1,y_2) &\lesssim \bracket{y_1}^{-\frac{3}{2}}e^{-\lambda|y_2|}\dist(v,\mM)\;,\\
    \big|\sqrt{\fD_{y_1}} h_2\big|(y_1,y_2) &\lesssim\bracket{y_1}^{-\frac{3}{2}}e^{-\lambda|y_2|}\;.\label{e:regularity_h2_half}
\end{equs}
We also have the bounds
\begin{equs}
    \big|\sqrt{\fD}^{\otimes 2}g_3\big|(y,y) &\lesssim \bracket{y}^{-3}\dist(v,\mM)\;,\\
    \big|\sqrt{\fD}^{\otimes 2} h_2\big|(y,y) &\lesssim\bracket{y}^{-3}\;.\label{e:regularity_h2}
\end{equs}
All proportionality constants above are independent of $v\in\vV_{\beta,0}$ and $y_1,y_2,y\in\RR$.
\end{lem}
\begin{proof}
The difference with the previous lemma is that the time integration is in $[1,+\infty)$ instead of $[0,1]$, and hence one needs to get decay in time from the integrands. 

For $g_3$, the decay in time is provided by
\begin{equation*}
    \big\| \big( F^t(v) \big)^{2n-1} - m^{2n-1} \big\|_{\lL^\infty} \lesssim \big\| F^t(v) - m \big\|_{\lL^\infty} \lesssim e^{-ct} \dist (v,\mM)\;,
\end{equation*}
where we used Proposition~\ref{pr:LinftyEC}. Hence, the two bounds for $\sqrt{\fD_{y_1}} g_3$ and $\sqrt{\fD}^{\otimes 2} g_3$ follow directly from the relevant bounds for $p_t(y,z;v)$ and $\sqrt{D_y} p_t(y,z;v)$ in Lemmas~\ref{lem:pt} and~\ref{lem:regularity_pt} and the exponential decay of $m'$. 

We now turn to the two bounds concerning $h_2$. Similar as $g_1$, the part of the integrand which does not decay in time integrates to $0$ in $z$ variable. More precisely, with the expression \eqref{e:product_pt_expansion} and that $\bracket{m^{2n-1}, \, (m')^3} = 0$, we have
\begin{equation*}
    \begin{split}
    h_2(\vec{y}) = \int_{1}^{+\infty} \int_\RR m'(z) \, &m^{2n-1}(z) \, \bigg( \prod_{j=1}^{2} \widetilde{p}_t(y_j, z;m)\\
    &- m'(z) \, \sum_{i \neq j} D \zeta (y_i; m) \, \widetilde{p}_t(y_j, z;m) \bigg) \, \md z \, \md t\;.
    \end{split}
\end{equation*}
The desired bounds then follow from Lemma~\ref{lem:regularity_phatt1}. 
\end{proof}

As a result, one has the following lemma.
\begin{lem}\label{lem:regularity_decay_h2y}
We have
\begin{equ}\label{e:regularity_decay_h2y}
    \Big|\sqrt{\fD}^{\otimes 2}\big(h_2\times(y_1+y_2)\big)\Big|(y,y)\lesssim  \bracket{y}^{-2}\;.
\end{equ}
As a result, the integration of the left hand side above over $y \in \RR$ is well defined, and there exists $\nu>0$ such that
\begin{equ}
    \int_\RR \sqrt{\fD}^{\otimes 2}\big(\qQ_\delta^{\otimes 2}h_2\times(y_1+y_2)\big)(y,y)\,\md y  = \int_\RR \sqrt{\fD}^{\otimes 2}\big(h_2\times(y_1+y_2)\big)(y,y)\,\md y  +\oO(\delta^\nu)\;.
\end{equ}
\end{lem}
\begin{proof}
    By \eqref{e:regularity_h2_half}, \eqref{e:regularity_h2} and that $h_2$ is symmetric in its two variables, it follows that $h_2$ satisfies the assumption of Lemma~\ref{lem:leibniz_4}. Hence the bound \eqref{e:regularity_decay_h2y} holds. The approximation error estimate $\oO(\delta^\nu)$ is a direct consequence of \eqref{e:regularity_decay_h2y}. 
\end{proof}

\begin{lem}\label{lem:regularity_g4}
There exists $\lambda >0 $ such that
\begin{equation*}
    \big|\sqrt{\fD_{y_1}} g_4\big|(y_1,y_2) \lesssim \bracket{y_1}^{-\frac{3}{2}}e^{-\lambda|y_2|}\dist(v,\mM)\;.
\end{equation*}
We also have the bound
\begin{equ}
    \big|\sqrt{\fD}^{\otimes 2}\qQ_\delta^{\otimes 2} g_4\big|(y,y) \lesssim |\log \delta| \, \bracket{y}^{-3} \dist(v,\mM)\;.
\end{equ}
Both bounds are uniformly over $v\in\vV_{\beta,0}$, $\delta\in(0,\frac{1}{2}]$ and $y_1,y_2,y\in\RR$.
\end{lem}
\begin{proof}
    The claim follows directly from Lemma~\ref{lem:regularity_pt2}, the exponential decay of $m'$ and the exponential convergence of the deterministic flow in Proposition~\ref{pr:LinftyEC}.
\end{proof}

Combining the above (Lemmas~\ref{lem:regularity_g1}, ~\ref{lem:regularity_g2}, ~\ref{lem:regularity_g3} and~\ref{lem:regularity_g4}) together, one can get Lemmas~\ref{lem:regularity_D2zeta_half}, ~\ref{lem:D2zeta_continuity} and~\ref{lem:D2zeta_decay}. In preparation of the proof of Lemma~\ref{lem:D2zeta_divergence_rate}, we also need the following lemma.

\begin{lem}\label{lem:regularity_h1tildey}
Define $\tilde{h}_1$ by
\begin{equation*}
    \tilde{h}_1(y_1,y_2) := \int_0^1\int_\RR m'(z) \, m^{2n-1}(z) \Big( \prod_{j=1}^{2}p_t(y_j,z;m) - \prod_{j=1}^{2}q_t(y_j,z) \Big) \,\md z\,\md t\;.
\end{equation*}
We have the bound
\begin{equ} \label{e:regularity_decay_h1tildey}
    \Big|\sqrt{\fD}^{\otimes 2}\big(\tilde{h}_1\times(y_1+y_2)\big)\Big|(y,y)\lesssim  \bracket{y}^{-2}\;.
\end{equ}
As a result, there exists $\nu>0$ such that
\begin{equ}
    \int_\RR \sqrt{\fD}^{\otimes 2}\Big(\qQ_\delta^{\otimes2}\tilde{h}_1\times (y_1+y_2)\Big)(y,y)\,\md y = \int_\RR\sqrt{\fD}^{\otimes 2}\Big(\tilde{h}_1\times (y_1+y_2)\Big)(y,y)\,\md y +\oO(\delta^\nu)\;.
\end{equ}
\end{lem}
\begin{proof}
We have
\begin{equation*}
    \prod_{j=1}^{2} p_t (y_j,z;m) - \prod_{j=1}^{2} q_t(y_j,z) = r_t(y_1,z;m) \, p_t (y_2,z;m) + q_t (y_1,z) \, r_t(y_2,z;m)\;.
\end{equation*}
Hence, by Lemmas~\ref{lem:regularity_rt} and~\ref{lem:regularity_pt}, we have
\begin{equation*}
    \big|\sqrt{\fD_{y_1}} \tilde{h}_1\big|(y_1,y_2) \lesssim\bracket{y_1}^{-\frac{3}{2}}e^{-\lambda|y_2|}\;, \quad \big|\sqrt{\fD}^{\otimes 2} \tilde{h}_1\big|(y,y) \lesssim\bracket{y}^{-3}\;.
\end{equation*}
The claim then follows from Lemma~\ref{lem:leibniz_4} and the symmetry of $\tilde{h}_1$ in its two variables. 
\end{proof}

\subsection{Proofs of the statements in Section~\ref{sec:deterministic_statements}}
We first give a final ingredient to prove the statements in Section~\ref{sec:deterministic_statements}. 
\begin{lem}\label{lem:technical_3}
For every $\lambda >0$, we have the bound
    \begin{equ}
        \int_\RR \frac{e^{-\lambda|y-z-\zeta|}\big|a_\eps(y)-a_\eps(y-z)\big|}{|z|^{\frac{3}{2}}}\,\md z \lesssim \bracket{y-\zeta}^{-\frac{3}{2}}
    \end{equ}
    for all $\eps\in[0,1]$ and $y,\zeta\in\RR$.
\end{lem}
\begin{proof}
We split the integral into the domains $\{|z| \leq 1\}$ and $\{|z| > 1\}$, where we bound $|a_\eps(y) - a_\eps(y-z)|$ by $C |z|$ and constant $C$ respectively. The desired bound then follows. 
\end{proof}

We now have all the ingredients to prove the statements in Section~\ref{sec:deterministic_statements}. 

\begin{proof} [Proof of Lemma~\ref{lem:Dzeta_continuity_aeps}]
By Lemma~\ref{lem:leibniz_2}, we have
\begin{equation*}
    \begin{split}
    \Big( \sqrt{\fD} \Big( \big( D\zeta(v) - &D \zeta (m_{\zeta(v)}) \big) a_\eps \Big) \Big)(y) = \Big( \sqrt{\fD} \big( D \zeta(v) - D \zeta (m_{\zeta(v)}) \big) \Big)(y) \cdot a_\eps(y)\\
    &+ C_{\ff} \int_\RR \frac{\big( D \zeta(z;v) - D \zeta (z; m_{\zeta(v)}) \big) \big( a_\eps(y) - a_\eps(z) \big)}{|y-z|^{\frac{3}{2}}} \,\md z\;.
    \end{split}
\end{equation*}
Then the conclusion follows from \eqref{e:Dzeta_difference}, \eqref{e:regularity_Dzeta_difference} and Lemma~\ref{lem:technical_3}. 
\end{proof}

\begin{proof} [Proof of Lemma~\ref{lem:Dzeta_taylor}]
By Lemma~\ref{lem:leibniz_2}, we have
\begin{align*}
    \bigg( \sqrt{\fD} \Big(D\zeta(m) \, &\big( a(x_0 + \sqrt{\eps}\cdot)-a(x_0) \big) \Big) \bigg)(y) = \big(\sqrt{\fD} D\zeta\big)(y;m) \cdot \big( a(x_0 + \sqrt{\eps}y) - a(x_0) \big)\\
    &+ C_\ff\int_\RR \frac{D\zeta(y-z;m) \big(a(x_0+\sqrt{\eps}y)-a(x_0+\sqrt{\eps}(y-z))\big)}{|z|^\frac{3}{2}}\,\md z\;.
\end{align*}
Then the conclusion follows from \eqref{e:Dzeta_bound}, \eqref{e:regularity_Dzeta_bound}, Lemma~\ref{lem:techninal_1} and that $a \in \cC_c^\infty$. 
\end{proof}

\begin{proof} [Proof of Lemma~\ref{lem:D2zeta_continuity_aeps}]
Taking $\varphi = \qQ_\delta^{\otimes2}\big(D^2\zeta(v)-D^2\zeta(m_{\zeta(v)})\big)$ and $\psi = a_\eps^{\otimes 2}$ in Lemma~\ref{lem:leibniz_3}, we have 
\begin{equation*}
    \begin{split}
    \big(&\sqrt{\fD}^{\otimes 2} \varphi \psi \big)(y,y) = \big(\sqrt{\fD}^{\otimes 2}\varphi\big)(y,y) \cdot a_\eps^2(y)\\
    &+ 2C_\ff \, a_\eps(y) \int_\RR \frac{ (\sqrt{\fD_{y_1}} \varphi)(y,y-z)  \cdot \big( a_\eps(y) - a_\eps(y-z) \big)}{|z|^{\frac{3}{2}}} \md z\\
    &+C_\ff^2 \iint\limits_{\RR^2} \frac{\varphi(y-z_1,y-z_2)\big(a_\eps(y)-a_\eps(y-z_1)\big)\big(a_\eps(y)-a_\eps(y-z_2)\big)}{|z_1z_2|^\frac{3}{2}}\,\md z_1 \,\md z_2\;,
    \end{split}
\end{equation*}
where we have used the fact that $\varphi$ is symmetric in its two variables, and hence $(\sqrt{\fD_{y_1}} \varphi)(y,y-z) = (\sqrt{\fD_{y_2}}\varphi)(y-z,y)$. Then the conclusion follows from Lemmas~\ref{lem:D2zeta_continuity}, ~\ref{lem:regularity_D2zeta_half}, ~\ref{lem:technical_3} and \eqref{e:D2zeta_bound}.
\end{proof}

\begin{proof} [Proof of Lemma~\ref{lem:D2zeta_taylor}]
Taking $\varphi = \qQ_\delta^{\otimes2}D^2\zeta(m)$ and
\begin{equation*}
    \psi(\vec{y}) = \prod_{j=1}^{2} a(x_0 + \sqrt{\eps} y_j) - a^2(x_0) - \sqrt{\eps} \, a(x_0) \, a'(x_0) \, (y_1 + y_2)
\end{equation*}
in Lemma~\ref{lem:leibniz_3} and noting that $\varphi$ and $\psi$ are symmetric in their two variables, we then have
\begin{align*}
    \big( \sqrt{\fD}^{\otimes 2}&(\varphi\psi) \big)(y,y)      = \big(\sqrt{\fD}^{\otimes 2} \qQ_\delta^{\otimes2}D^2\zeta(m) \big)(y,y) \cdot \psi(y,y)\\
    &+2C_\ff \int_\RR \frac{ \big(\sqrt{\fD_{y_1}} \qQ_\delta^{\otimes2}D^2\zeta(m)\big)(y,y-z) \cdot \big(\psi(y,y)-\psi(y,y-z)\big)}{|z|^{\frac{3}{2}}}\,\md z\\
    &+C_\ff^2 \iint\limits_{\RR^2} \frac{ \qQ_\delta^{\otimes2}D^2\zeta(m)(y-z_1,y-z_2)A}{|z_1z_2|^\frac{3}{2}}\,\md z_1 \,\md z_2\;,
\end{align*}
where $A$ satisfies
    \begin{equation*}
        \begin{split}
        |A| &= |\psi(y,y)-\psi(y-z_1,y)-\psi(y,y-z_2)+\psi(y-z_1,y-z_2)|\\
        &\lesssim |\sqrt{\eps}z_1|^{\frac{1}{2}+\nu}|\sqrt{\eps}z_2|^{\frac{1}{2}+\nu}\;.
        \end{split}
    \end{equation*}
    Then the conclusion follows from Lemmas~\ref{lem:D2zeta_decay}, ~\ref{lem:regularity_D2zeta_half}, ~\ref{lem:techninal_1}, \eqref{e:D2zeta_bound} and the fact that $|\psi(y,y)|\lesssim |\sqrt{\eps}y|^{1+2\nu}$ as well as $|\psi(y,y)-\psi(y,y-z)|\lesssim|\sqrt{\eps}z|(|\sqrt{\eps}y|^{2\nu}+|\sqrt{\eps}z|^{2\nu})$.
\end{proof}

\begin{proof}[Proof of Lemma~\ref{lem:D2zeta_divergence_rate}]
By Assertion 1 in Lemma~\ref{lem:Dzeta}, we have
\begin{equ}
    D^2\zeta(\vec{y};m) = \frac{2n(2n+1)}{\|m'\|^2_{\lL^2}}\int_0^\infty \int_\RR m'(z) \, m^{2n-1}(z) \, \prod_{j=1}^{2} p_t(y_j,z;m) \,\md z \,\md t\;.
\end{equ}
We divide the integral into $\int_0^1$ and $\int_1^\infty$, which corresponds to the decomposition of $D^2 \zeta(m)$ into $h_1 + h_2$ (up to constant multiple) as in \eqref{e:D2zeta_decomposition} and \eqref{e:D2zeta_h}.

Note that $h_1$ contains the singular part of $D^2 \zeta (m)$ as it involves the singularity of $p_t$ near $t \approx 0$. This singularity is the same as the heat kernel $q_t$. Hence, we write
\begin{equation*}
    G(\vec{y}) := \int_0^1 \int_\RR m'(z) \, m^{2n-1}(z) \prod
    _{j=1}^{2} q_t(y_j-z)\,\md z\,\md t\;,
\end{equation*}
which is the singular part of $h_1$. Denote 
\begin{equ}
    P_\delta:=\frac{2n(2n+1)}{\|m'\|^2_{\lL^2}}\int_\RR \sqrt{\fD}^{\otimes 2}\Big(\qQ_\delta^{\otimes2}G\times \frac{1}{2}(y_1+y_2)\Big)(y,y)\,\md y\;.
\end{equ}
Then by Lemmas~\ref{lem:regularity_decay_h2y} and~\ref{lem:regularity_h1tildey}, we have
\begin{equ}
    \frac{1}{2}\int_\RR \sqrt{\fD}^{\otimes 2}\Big(\qQ_\delta^{\otimes2}\big(D^2\zeta(m)\big)\times (y_1+y_2)\Big)(y,y)\,\md y = P_\delta + \text{const} + \oO(\delta^\nu)\;.
\end{equ}
It remains to control $P_\delta$. By Lemma~\ref{lem:leibniz_3} and the symmetry of $G$ in its two variables, we have
\begin{equation}\label{e:expression_leibniz}
\begin{aligned}
    \sqrt{\fD}^{\otimes 2} \Big(\qQ_\delta^{\otimes2} G \times \frac{1}{2}(y_1+y_2)\Big) &(y,y) = \, \big(\sqrt{\fD}^{\otimes 2}\qQ_\delta^{\otimes2}G\big)(y,y) \cdot y\\
    &+C_\ff \int_\RR \frac{\big(\sqrt{\fD_{y_1}} \qQ_\delta^{\otimes2} G\big)(y,y-u)  \cdot u}{|u|^{\frac{3}{2}}}\,\md u\;.
\end{aligned}
\end{equation}
Noting that $\qQ_\delta q_t =q_{t+\delta^2}$, we have
\begin{equation} \label{e:G_mollify_expression}
    (\qQ_\delta^{\otimes2}G)(\vec{y}) = \int_0^1 \int_\RR m'(z) \, m^{2n-1}(z) \prod_{j=1}^{2} q_{t+\delta^2}(y_j-z) \,\md z\,\md t\;.
\end{equation}
Hence, the contribution to $P_\delta$ from the second term on the right hand side of \eqref{e:expression_leibniz} (after integrating $y \in \RR$) is
\begin{align*}
    &\int_0^1 \iiint\limits_{\RR^3} \sqrt{\fD_{y}}q_{t+\delta^2}(y-z)q_{t+\delta^2}(y-u-z)m^{2n-1}(z)m'(z)\frac{u}{|u|^{\frac{3}{2}}}\,\md u\,\md z\,\md y\,\md t\\
    =&\int_0^1 \int_\RR \Big(\int_\RR \sqrt{\fD_{y}} q_{t+\delta^2}(y) \, q_{t+\delta^2}(y-u) \md y \Big) \Big(\int_\RR m'(z) \, m^{2n-1}(z) \md z \Big) \frac{u}{|u|^{\frac{3}{2}}}\,\md u \,\md t=0\;,
\end{align*}
where we perform a change of variable $y\mapsto y+z$ in the second inequality and use the fact that $m$ is odd so $m' \cdot m^{2n-1}$ integrates to $0$ in the last step. Hence, the first term on the right hand side of \eqref{e:expression_leibniz} contains all contribution to $P_\delta$. 

Now, plugging the expression \eqref{e:G_mollify_expression} into the first term on the left hand side of \eqref{e:expression_leibniz} and integrating over $y \in \RR$, we get
\begin{equation*}
    P_\delta = \frac{2n(2n+1)}{\|m'\|^2_{\lL^2}} \int_0^1 \iint\limits_{\RR^2}  y \cdot \big(\sqrt{\fD_{y}}q_{t+\delta^2}\big)^2(y-z) \, m^{2n-1}(z) \, m'(z) \, \md z \,\md y\,\md t\;.
\end{equation*}
Performing a change of variable $y \mapsto y+z$ and using again the fact that $m' m^{2n-1}$ integrates to $0$, we obtain
\begin{equation*}
    P_\delta = \frac{2n(2n+1)}{\|m'\|^2_{\lL^2}} \Big(\int_{0}^{1} \int_\RR \big(\sqrt{\fD_{y}}q_{t+\delta^2}(y)\big)^2\,\md y\,\md t \Big) \Big(\int_\RR m^{2n-1}(z)m'(z)z\,\md z\Big)\;.
\end{equation*}
With Plancherel's identity, we can do the explicit computation
\begin{align*}
    \int_0^1 \int_\RR  \big(\sqrt{\fD_{y}}q_{t+\delta^2}\big)^2(y)\,\md y\,\md t = \frac{1}{8\pi^2} \log \Big(\frac{1+\delta^2}{\delta^2}\Big)\;.
\end{align*}
The conclusion then follows from the definition of $C_0^*$ in \eqref{e:C0}. 
\end{proof}

\appendix

\section{Proof of Lemma~\ref{lem:smallnessOfLinearSolution}}\label{app:linear_smallness}

This section is devoted to the proof of Lemma~\ref{lem:smallnessOfLinearSolution}, which states that the $\wW^{-\kappa',p}$ norm of the linearized solution $X_\eps$ is small. For notational convenience in the proof, we give details to establish the same bound in the H\"older-type space $\cC^{-\kappa'}$. The corresponding $\wW^{-\kappa',p}$ estimate can be obtained with only minor modifications, which we briefly sketch at the end of the proof. 

Recall that $X_\eps$ is the stationary-in-time solution to \eqref{e:linear_solution}. For simplicity of notations, we sometimes omit the time $t$ in the notation when it only concerns the law of $X_\eps[t]$ for the fixed time $t$. We have the following expression on the correlation function. 

\begin{lem} \label{lem:linear_Exy_expression}
We have the expression
    \begin{equation}\label{e:linear_Exy_expression}
    \begin{aligned}
        &\phantom{111}\EE \big(X_\eps(t,x) X_\eps(t, y) \big)\\
        &= \eps^{2 \gamma} \int_{\RR} |\eta| e^{-2 \pi i \eta (x-y)} \bigg( \iint\limits_{\RR^2} \frac{\widehat{a}(\theta_1) \overline{\widehat{a}(\theta_2)} \, e^{2 \pi i \sqrt{\eps} (\theta_1 x - \theta_2 y)} \,\md \theta_1 \,\md  \theta_2 }{4 \pi^2 \big( (\eta - \sqrt{\eps} \theta_1)^2 + (\eta - \sqrt{\eps} \theta_2)^2 \big) + 2 \mu}  \bigg) {\rm d}\eta\;.    
    \end{aligned}
    \end{equation}
\end{lem}
\begin{proof}
Write $X_\eps$ in terms of its Fourier transform, and we have
\begin{equation*}
    \EE \big(X_\eps(t,x) X_\eps(t, y) \big) = \iint\limits_{\RR^2} e^{2 \pi i (\theta_1 x + \theta_2 y)} \, \EE \big( \widehat{X_\eps}(t,\theta_1) \widehat{X_\eps}(t,\theta_2) \big) \, \md \theta_1 \,\md \theta_2\;.
\end{equation*}
Note that for $k=1,2$, we have
\begin{equation*}
    \widehat{X_\eps}(t,\theta_k) = \eps^\gamma \int_{-\infty}^{t} e^{-(4\pi^2 \theta_k^2 + \mu)(t- r_k)} \Big( \int_{\RR} \widehat{a_\eps}(\theta_k - \eta_k) \, |\eta_k|^{\frac{1}{2}} \, \widehat{\dot{W}}(r_k, \eta_k) \,\md \eta_k \Big) \, \md r_k\;,
\end{equation*}
and that the Fourier transform of the spacetime white noise has correlation
\begin{equation*}
    \EE \big( \widehat{\dot{W}}(r_1, \eta_1) \widehat{\dot{W}}(r_2, \eta_2) \big) = \delta (r_1 - r_2) \, \delta (\eta_1 + \eta_2)\;.
\end{equation*}
The conclusion then follows from plugging these identities back into the original expression and then a change of variable. 
\end{proof}

\begin{lem}\label{lem:linear_Exy_bound}
We have
    \begin{equ}
        \big| \EE \big( X_\eps(t,x) X_\eps(t,y) \big) \big| \lesssim \eps^{2\gamma} \big( \log |x-y|^{-1} \vee 1 \big)
    \end{equ}
for all $\eps \in [0,1]$ and $x,y \in \RR$. 
\end{lem}
\begin{proof}
We write
\begin{equation} \label{e:Exy_expansion}
\begin{aligned}
    &\frac{1}{4\pi^2(\eta - \sqrt{\eps}\theta_1)^2+4\pi^2(\eta - \sqrt{\eps}\theta_2)^2+2\mu} \\
    = &\Big(\frac{1}{4\pi^2(\eta - \sqrt{\eps}\theta_1)^2+4\pi^2(\eta - \sqrt{\eps}\theta_2)^2+2\mu} -\frac{1}{8\pi^2\eta^2+2\mu}\Big)+\frac{1}{8\pi^2\eta^2+2\mu}
\end{aligned}
\end{equation}
for the denominator term in the expression \eqref{e:linear_Exy_expression}. Note that
\begin{equ} \label{e:Exy_expansion_part1_1}
    \frac{1}{4\pi^2(\eta - \sqrt{\eps}\theta_1)^2+4\pi^2(\eta - \sqrt{\eps}\theta_2)^2+2\mu} \lesssim \frac{\bracket{\theta_1}^2}{\bracket{\eta}^2}\;,
\end{equ}
so together with the subtraction of $\frac{1}{8 \pi^2 \eta^2 + 2\mu}$, the first term on the right hand side of \eqref{e:Exy_expansion} is bounded by 
\begin{equ}\label{e:Exy_expansion_part1_2}
\frac{\big( |\theta_1 + \theta_2| \cdot |\eta| + \theta_1^2 + \theta_2^2 \big) \bracket{\theta_1}^2}{\bracket{\eta}^4}\;.
\end{equ}
Hence, the contribution of this part in $\EE \big(X_\eps(t,x)X_\eps(t,y) \big)$ is bounded by $\eps^{2\gamma}$. 

The second term on the right hand side of \eqref{e:Exy_expansion} is independent of $\theta_1$ and $\theta_2$, and hence by Fourier inversion, we see its contribution to $\EE \big(X_\eps(t,x)X_\eps(t,y) \big)$ is precisely
\begin{equ}
    \eps^{2\gamma} a_\eps(x)a_\eps(y) \int_\RR \frac{|\eta| e^{-2\pi i \eta(x-y)}}{8\pi^2\eta^2+2\mu}\,\md \eta\;,
\end{equ}
which is bounded by $\eps^{2\gamma} \big(\log |x-y|^{-1}\vee 1\big)$. This completes the proof. 
\end{proof}

For every non-negative $\varphi\in\cC_c^\infty(-1,1)$, we write $\varphi_x^\lambda:=\lambda^{-1}\varphi(\frac{\cdot-x}{\lambda})$. We have the following lemma. 

\begin{lem} \label{lem:linear_EXphi_bound}
We have the bound
\begin{equ}
    \EE \big| \bracket{X_\eps,\varphi_x^\lambda} \big|^2\lesssim\eps^{2\gamma}|\log\lambda|\;.
\end{equ}
Furthermore, for every $\nu \in [0,2]$, we have
\begin{equ}
    \EE \big| \bracket{X_\eps,\varphi_x^\lambda} - \bracket{X_\eps,\varphi_{x'}^\lambda}\big|^2 \lesssim \frac{\eps^{2\gamma} |\log\lambda|}{\lambda^\nu} \cdot |x-x'|^\nu\;.
\end{equ}
Both bounds are uniformly over all $\eps \in [0,1]$, $\lambda \in (0,\frac{1}{2}]$ and $x, x' \in \RR$. 
\end{lem}
\begin{proof}
The first bound follows directly from Fubini's Theorem to change the order of integration and then Lemma~\ref{lem:linear_Exy_bound}. For the second one, we have
\begin{equation*}
    \EE \big| \bracket{X_\eps, \varphi_x^\lambda - \varphi_{x'}^\lambda} \big|^2 \lesssim \eps^{2\gamma} \iint\limits_{\RR^2} \big( 1 \vee \, \big| \log |y-z| \, \big| \big) \cdot \big( \varphi_x^\lambda (y) - \varphi_{x'}^\lambda (y) \big) \big( \varphi_x^\lambda (z) - \varphi_{x'}^\lambda (z) \big) \,\md y \,\md z\;.
\end{equation*}
By H\"older continuity of $\varphi$, we have
\begin{equation*}
    \big| \varphi_x^\lambda (y) - \varphi_{x'}^\lambda (y) \big| \lesssim \frac{1}{\lambda} \cdot \Big( \frac{|x-x'|}{\lambda} \Big)^{\frac{\nu}{2}}\;,
\end{equation*}
and the same estimate holds for the terms with the $z$ variable. The conclusion then follows. 
\end{proof}

\begin{lem} \label{lem:linear_X_test_compact}
For every $\nu\in(0,2)$ and $p$ sufficiently large (depending on $\nu$), we have the bound
    \begin{equation*}
    \EE\sup_{\substack{|x|\leq \eps^{-\frac{1}{2}}}}|\bracket{X_\eps,\varphi_x^\lambda}|^p\lesssim\eps^{p\gamma-\frac{1}{2}}\lambda^{-p\nu}\;.
\end{equation*}
The proportionality constant depends on $\nu$ and $p$, but is independent of $\eps \in [0,1]$. 
\end{lem}
\begin{proof}
    By Lemma~\ref{lem:linear_EXphi_bound}, hypercontractivity of Gaussian and Kolmogorov's continuity theorem, for every $p$ sufficiently large and $k\in\ZZ$, we have
    \begin{equ}
    \EE \sup_{\substack{x\in[k,k+1]}}| \bracket{X_\eps,\varphi_x^\lambda}|^p \lesssim \eps^{p\gamma} \cdot \lambda^{-p\nu}\;.
    \end{equ}
    Summing over integers $k\in[-\eps^{-\frac{1}{2}},\eps^{-\frac{1}{2}}]$ gives the desired inequality.
\end{proof}

\begin{lem} \label{lem:linear_Ex_decay}
We have the bound
    \begin{equ}
        \EE X_\eps^2(x) \lesssim \eps^{2\gamma} |x|^{-1}
    \end{equ}
    for all $\eps \in [0,1]$ and $x \in \RR$ with $|x| > \eps^{-\frac{1}{2}}$. 
\end{lem}
\begin{proof}
By \eqref{e:linear_Exy_expression} and the fact that $a_\eps(x) = 0$ for $|x| > \eps^{-\frac{1}{2}}$, we have the expression
\begin{align*}
    x\EE X_\eps^2(x) = &\eps^{2\gamma} \int_\RR |\eta| \bigg( \iint\limits_{\RR^2} \widehat{a}(\theta_1)\overline{\widehat{a}(\theta_2)} \cdot x \, e^{2\pi i \sqrt{\eps} x(\theta_1-\theta_2)} \\
    & \Big(\frac{1}{4\pi^2 \big( (\eta - \sqrt{\eps}\theta_1)^2 + (\eta - \sqrt{\eps}\theta_2)^2 \big) + 2\mu}-\frac{1}{8\pi^2\eta^2+2\mu}\Big)\,\md \theta_1\,\md \theta_2\bigg)\,\md \eta\\
\end{align*}
Noting that $\sqrt{\eps}xe^{2\pi i \sqrt{\eps}x(\theta_1-\theta_2)} = \frac{1}{2\pi i}\frac{\d}{\d \theta_1}e^{2\pi i \sqrt{\eps}x(\theta_1-\theta_2)}$, and that 
\begin{align*}
    &\bigg|\frac{\d}{\d \theta_1} \Big(\frac{1}{4\pi^2 \big( (\eta - \sqrt{\eps}\theta_1)^2 + (\eta - \sqrt{\eps}\theta_2)^2 \big) + 2\mu}-\frac{1}{8\pi^2\eta^2+2\mu} \Big) \bigg|\\
    \lesssim & \frac{\big(|\eta|+|\theta_1|\big) \, \bracket{\theta_1}^2}{\bracket{\eta}^4} + \frac{\big(|\theta_1+\theta_2| \cdot |\eta| + \theta_1^2+\theta_2^2\big) \, (|\eta|+|\theta_1|) \, \bracket{\theta_1}^4}{\bracket{\eta}^6}\;,
\end{align*}
where we have used \eqref{e:Exy_expansion_part1_1}, the conclusion follows from integration by parts (with respect to $\theta_1$).
\end{proof}

\begin{lem}\label{lem:linear_Exy_continuity_space}
Fix $\nu \in (0,1)$. We have the bound
    \begin{equ}
        \EE \big(X_\eps(x)-X_\eps(y)\big)^2\lesssim \eps^{2\gamma} |x-y|^{\nu}
    \end{equ}
for all $\eps \in [0,1]$ and $x,y \in \RR$ with $|x|, |y| > \eps^{-\frac{1}{2}}$. 
\end{lem}
\begin{proof}
By \eqref{e:linear_Exy_expression}, we have the expression
\begin{align*}
    &\EE \big(X_\eps(x)-X_\eps(y)\big)^2 = \eps^{2\gamma} \int |\eta| \bigg( \iint\limits_{\RR^2} \widehat{a}(\theta_1) \overline{\widehat{a}(\theta_2)} \cdot \\
    &\phantom{111}\frac{\big(e^{-2\pi i (\eta-\sqrt{\eps}\theta_1)x} -e^{-2\pi i (\eta-\sqrt{\eps}\theta_1)y}\big)\big(e^{2\pi i (\eta-\sqrt{\eps}\theta_2)x} -e^{2\pi i (\eta-\sqrt{\eps}\theta_2)y}\big)}{4\pi^2(\eta - \sqrt{\eps}\theta_1)^2+4\pi^2(\eta - \sqrt{\eps}\theta_2)^2+2\mu}\,\md \theta_1\,\md \theta_2 \bigg)\,\md \eta\;.
\end{align*}
Note that $a_\eps(x) = a_\eps(y) = 0$ for $|x|,|y| > \eps^{-\frac{1}{2}}$. We then can replace the denominator in the above integral by
\begin{equation*}
    \frac{1}{4\pi^2(\eta - \sqrt{\eps}\theta_1)^2+4\pi^2(\eta - \sqrt{\eps}\theta_2)^2+2\mu} - \frac{1}{8 \pi^2 \eta^2 + 2 \mu}\;.
\end{equation*}
Then the conclusion follows from \eqref{e:Exy_expansion_part1_2} and the fact that
\begin{equ}
    |e^{-2\pi i (\eta-\sqrt{\eps}\theta_1)x} -e^{-2\pi i (\eta-\sqrt{\eps}\theta_1)y}|\lesssim |\eta-\sqrt{\eps}\theta_1|^\nu|x-y|^\nu\lesssim \big(|\eta|^\nu + |\theta_1|^\nu\big)|x-y|^\nu\;.
\end{equ}
This completes the proof. 
\end{proof}
\begin{lem} \label{lem:linear_outside}
For every $p>8$, we have the bound
    \begin{equation*}
    \EE\sup_{\substack{|x|\geq \eps^{-\frac{1}{2}}}}|X_\eps(x)|^p\lesssim \eps^{p\gamma-\frac{1}{2}}\;.
\end{equation*}
The proportionality constant is independent of $\eps \in [0,1]$. 
\end{lem}
\begin{proof}
For notational simplicity, we assume the supremum is taken over positive $x > \eps^{-\frac{1}{2}}$. Interpolating Lemmas~\ref{lem:linear_Ex_decay} and~\ref{lem:linear_Exy_continuity_space}, and using hypercontractivity as well as Kolmogorov's continuity criterion, we have
\begin{equ}
    \EE \sup_{x \in [M, M+1]} |X_\eps(x)|^p\lesssim \eps^{p\gamma}2^{-\frac{p\ell}{4}}
\end{equ}
for all $\eps \in [0,1]$, $\ell \in \NN$ and $M \in [2^\ell / \sqrt{\eps}, \; 2^{\ell+1} / \sqrt{\eps}]$. 

Summing over integers $M \in [2^\ell / \sqrt{\eps}, \; 2^{\ell+1} / \sqrt{\eps}]$ then gives
\begin{equation}\label{e:X_eps_decay}
    \EE \sup_{x \in [2^\ell / \sqrt{\eps}, \; 2^{\ell+1} / \sqrt{\eps}]} |X_\eps(x)|^p \lesssim \eps^{p \gamma - \frac{1}{2}} \cdot 2^{\ell-\frac{p\ell}{4}}\;.
\end{equation}
Further summing over the above bound over $\ell \in \NN$, the conclusion then follows for sufficiently large $p$. 
\end{proof}

\begin{lem} \label{lem:linear_fixed_time}
    For every $\kappa' \in (0,1)$ and $p$ sufficiently large (depending on $\kappa'$), we have the bound
    \begin{equ}
        \EE \|X_\eps[t]\|_{\cC^{-\kappa'}}^p \lesssim \eps^{p\gamma-\frac{1}{2}}\;.
    \end{equ}
    The proportionality constant is independent of $\eps \in [0,1]$ and $t \geq 0$. 
\end{lem}
\begin{proof}
The statement follows by combining Lemma~\ref{lem:linear_X_test_compact} and~\ref{lem:linear_outside}. 
\end{proof}

\begin{lem}
    For every $\kappa' \in (0,1)$, $N>0$ and $p$ sufficiently large (depending on $\kappa'$), we have
    \begin{equ}
        \EE \sup_{t\in[0,\eps^{-N}]}\|X_\eps[t]\|_{\cC^{-\kappa'}}^p\lesssim \eps^{p\gamma-\frac{1}{2}-N}
    \end{equ}
    for all $\eps \in [0,1]$. 
\end{lem}
\begin{proof}
For every $t >s$, by Duhamel's principle, we write
    \begin{equ}
        X_\eps[t]-X_\eps[s] = \big(e^{(t-s)(\Delta-\mu)}-\operatorname{id}\big)X_\eps[s] +\eps^\gamma\int_s^t e^{(t-r)(\Delta-\mu)}\big(a_\eps\sqrt{\fD}\,\md W(r)\big)\;.
    \end{equ}
For the first term, we use the fact that 
\begin{equ}
    \|\big(e^{(t-s)(\Delta-\mu)}-\operatorname{id}\big)X_\eps[s]\|_{\cC^{-\kappa'}}\lesssim (t-s)^{\frac{\kappa'}{4}}\|X_\eps[s]\|_{\cC^{-\frac{\kappa'}{2}}}\;.
\end{equ}
For the second term, with the same argument as Lemma~\ref{lem:linear_fixed_time} and noting additionally that the time integration is in $[s, t]$ instead of $(-\infty, t]$, one sees that 
\begin{equ}
    \EE\Big\|\int_s^t e^{(t-r)(\Delta-\mu)}\big(a_\eps\sqrt{\fD}\,\md W(r)\big)\Big\|_{\cC^{-\kappa'}}^p\lesssim (t-s)^{\frac{p\kappa'}{2}}\eps^{-\frac{1}{2}}\;.
\end{equ}
This completes the proof. 
\end{proof}
\begin{lem}
For every $k\in\NN$, we have the bound
\begin{equ}
    \EE \big| \bracket{X_\eps^{\diamond k},\varphi_x^\lambda} \big|^2 \lesssim \eps^{2k\gamma} \cdot |\log \lambda|^k
\end{equ}
for all $\eps\in[0,1]$ and $\lambda \in(0,\frac{1}{2}]$.
\end{lem}
\begin{proof}
    By the Wick's formula, we have the identity
    \begin{equation*}
        \EE \big( X_\eps^{\diamond k}(t,x) X_\eps^{\diamond k} (t,y) \big) = k! \Big( \EE \big( X_\eps(t,x) X_\eps(t,y) \big) \Big)^k\;.
    \end{equation*}
    The desired bound then follows directly from Lemma~\ref{lem:linear_Exy_bound}. 
\end{proof}

\begin{lem} \label{lem:linear_final_approx}
    For every $k \in \NN$, $\kappa'\in(0,1)$, $N>0$ and $p$ sufficiently large (depending on $\kappa'$), we have
    \begin{equ}
        \EE \sup_{t\in[0,\eps^{-N}]}\|X_\eps^{\diamond k}[t]\|_{\cC^{-\kappa'}}^p\lesssim \eps^{kp\gamma-\frac{1}{2}-N}
    \end{equ}
    for all $\eps \in [0,1]$. Furthermore, we also have the convergence
    \begin{equation*}
        \EE \sup_{t \in [0,\eps^{-N}]} \| H_{k}\big( \qQ_\delta X_\eps[t]; \, \fC_\eps^{(\delta)} \big) - X_\eps^{\diamond k}[t] \|_{\cC^{-\nu}}^p \lesssim \eps^{kp\gamma-1-N}\delta^{\frac{\nu p}{2}}\;,
    \end{equation*}
    where the proportionality constant is independent of $\eps, \delta \in [0,1]$. 
\end{lem}

\begin{rmk}
    Lemma~\ref{lem:linear_final_approx} gives the bound with $p=+\infty$. To obtain the corresponding bound for sufficiently large (but finite) $p$, we first note that for every $p \geq 1$, there exists $\nu > 0$ such that
    \begin{equation*}
        \|X_\eps^{\diamond k}\|_{\wW^{-\kappa',p}} \lesssim \eps^{-\nu} \|X_\eps^{\diamond k}\|_{\cC^{-\kappa'}}
    \end{equation*}
    when restricted on the interval $[-\frac{1}{\sqrt{\eps}}, \frac{1}{\sqrt{\eps}}]$, and that $\nu \downarrow 0$ as $p \uparrow +\infty$. Outside the interval $[-\frac{1}{\sqrt{\eps}}, \frac{1}{\sqrt{\eps}}]$, the bound \eqref{e:X_eps_decay} implies
    \begin{equation*}
        \EE \int_{|x| > \frac{1}{\sqrt{\eps}}} |X_\eps^{\diamond k}(t,x)|^p\, \md x \lesssim \eps^{k(p\gamma - \frac{1}{2})}\;.
    \end{equation*}
    The reason it carries over for the Wick product is that the renormalization $\E |X_\eps(t,x)|^2$ is finite for $|x|>\frac{1}{\sqrt{\eps}}$. Together, they imply Lemma~\ref{lem:smallnessOfLinearSolution} for all sufficiently large $p$. 
\end{rmk}

\section{Proof of Lemma~\ref{le:renormalisation_approx}}
\label{app:renormalization}

\begin{proof} [Proof of Lemma~\ref{le:renormalisation_approx}]
Write $\phi_{s}(\eta):= e^{-4\pi^2 s\eta^2}$. Then $\fC_\eps^{(\delta)}$ has the expression
\begin{equation} \label{e:renorm_funct_expression}
    \begin{split}
    \fC_\eps^{(\delta)}(x) = \eps^{2\gamma} &\iint\limits_{\RR^2} \widehat{a}(\theta_1) \, \overline{\widehat{a}(\theta_2)} \, e^{2\pi i \sqrt{\eps} (\theta_1 - \theta_2) x}\\
    &\cdot \bigg[\int_{0}^{+\infty} e^{-2\mu s} \bigg( \int_{\RR} |\eta| \prod_{j=1}^{2} \phi_{s+\delta^2} (\eta + \sqrt{\eps} \theta_j) \, {\rm d} \eta \bigg) \, {\rm d}s \bigg] \,\md \theta_1 \,\md \theta_2\;.
    \end{split}
\end{equation}
Taylor expanding $\phi_{s+\delta^2}(\eta + \sqrt{\eps} \theta_j)$ at $\eta$, we can write
\begin{equation} \label{e:phi_product_expansion}
    \begin{split}
    &\prod_{j=1}^{2} \phi_{s+\delta^2}(\eta + \sqrt{\eps} \theta_j) = \phi_{s+\delta^2}^2(\eta) + \sqrt{\eps} \, \phi_{s+\delta^2}(\eta) \, \phi_{s+\delta^2}'(\eta) \cdot (\theta_1 + \theta_2)\\
    &+ \prod_{j=1}^{2} \Big( \phi_{s+\delta^2}(\eta + \sqrt{\eps} \theta_j) - \phi_{s+\delta^2}(\eta) \Big) +  \phi_{s+\delta^2}(\eta) \sum_{j=1}^{2} R_{s+\delta^2,\eps}(\eta, \theta_j)\;,
    \end{split}
\end{equation}
where
\begin{equ}
    R_{s+\delta^2, \,\eps}(\eta, \theta_j) = \eps \theta_j^2 \int_0^1 (1-\nu) \, \phi_{s+\delta^2}''(\eta + \sqrt{\eps}\nu\theta_j) \,\md \nu\;.
\end{equ}
We plug each term on the right hand side of \eqref{e:phi_product_expansion} back into the integral expression of $\fC_\eps^{(\delta)}$ and analyze their contributions. 

For the term $\phi_{s+\delta}^2(\eta)$, by Fourier inversion formula, its contribution to $\fC_\eps^{(\delta)}$ is precisely
\begin{equation*}
    \eps^{2\gamma} a_\eps^2(x) \int_{\RR} \frac{|\eta| e^{-8\pi^2 \delta^2 \eta^2}}{2(4\pi^2 \eta^2 + \mu)} \,\md \eta = \eps^{2\gamma} C^{(\delta)} a_\eps^2(x)\;.
\end{equation*}
Also, since $\phi_{s+\delta^2}$ is even, by integrating $\eta$ first, we see the contribution from the term $\phi_{s+\delta^2}(\eta) \phi_{s+\delta^2}'(\eta)$ is precisely $0$. 

As for the product of the two first-order Taylor remainder terms, we have
\begin{equation} \label{e:phi_product_remainder_1}
    \begin{split}
    \left| \phi_{s+\delta^2}(\eta + \sqrt{\eps} \theta_1) - \phi_{s+\delta^2}(\eta) \right| &\leq \sqrt{\eps} |\theta_1| \int_{0}^{1} |\phi_{s+\delta^2}'(\eta + \sqrt{\eps} \theta_1 \nu)| \, {\rm d}\nu\\
    &\lesssim \sqrt{\eps} |\theta_1| \sqrt{s+\delta^2} \int_{0}^{1} e^{-c (s+\delta^2)(\eta + \sqrt{\eps} \theta_1 \nu)^2} \, {\rm d} \nu\;,
    \end{split}
\end{equation}
where we have used the pointwise bound
\begin{equation*}
    |\phi_r' (\widetilde{\eta})| = c_0 r |\widetilde{\eta}| \cdot e^{-4 \pi^2 r \widetilde{\eta}^2} \lesssim \sqrt{r} e^{-c r \widetilde{\eta}^2}
\end{equation*}
for some $c>0$. For the term with $\theta_2$, we simply bound the integral in $\nu$ by $1$ so that
\begin{equation} \label{e:phi_product_remainder_2}
    \left| \phi_{s+\delta^2}(\eta + \sqrt{\eps} \theta_2) - \phi_{s+\delta^2}(\eta) \right| \lesssim \sqrt{\eps} |\theta_2| \sqrt{s+\delta^2}\;.
\end{equation}
Plugging the bounds \eqref{e:phi_product_remainder_1} and \eqref{e:phi_product_remainder_2} into \eqref{e:renorm_funct_expression}, we see the contribution of the product of the two first order Taylor remainder terms to the integral in $\eta$ (including $|\eta|$) is bounded by
\begin{equation*}
    \begin{split}
    &\phantom{111}\eps |\theta_1 \theta_2| (s+\delta^2) \int_{\RR} |\eta| \int_{0}^{1} e^{-c(s+\delta^2) (\eta + \sqrt{\eps} \theta_1 \nu)^2} \,\md \nu \, \md \eta\\
    &\lesssim \eps |\theta_1 \theta_2| (s+\delta^2) \int_\RR \big( |\eta| + \sqrt{\eps} |\theta_1| \big) e^{-c(s+\delta^2) \eta^2} \,\md \eta \lesssim \eps |\theta_1 \theta_2| \Big( 1 + \sqrt{s+\delta^2} \, |\theta_1| \Big)\;,
    \end{split}
\end{equation*}
where the first bound follows from a change of variable $\eta \mapsto \eta +\sqrt{\eps} \theta_1 \nu$ together with proper changes of order of integration. Plugging this bound back into \eqref{e:renorm_funct_expression} and using the decay of $\widehat{a}$ and $e^{-2\mu s}$, we see the contribution to $\fC_\eps^{(\delta)}$ from the third term on the right hand side of \eqref{e:phi_product_expansion} is bounded by $\eps^{2\gamma+1}$. The contribution of the term $\phi_{s+\delta^2}(\eta) R_{s+\delta^2,\, \eps}(\eta,\theta_j)$ can be obtained in the same way. We have thus completed the proof. 
\end{proof}

\section{Fractional Leibniz rule}
\label{app:fractional_Leibniz}

We define the constant $C_\ff$ by
\begin{equation*}
    \frac{1}{C_\ff} := \int_{\RR} \frac{1 - \cos (2 \pi z)}{|z|^{\frac{3}{2}}} {\rm d}z = \int_\RR \frac{1 - e^{2 \pi i z}}{|z|^{\frac{3}{2}}} {\rm d}z\;.
\end{equation*}

\begin{lem}\label{lem:leibniz_1}
For every $\varphi\in\sS$, we have the representation
    \begin{equ}
        (\sqrt{\fD}\varphi)(x) = C_\ff \int_\RR \frac{\varphi(x)-\varphi(x-z)}{|z|^\frac{3}{2}}\,\md z\;.
    \end{equ}
\end{lem}
\begin{proof}
Note that by definition of $C_\ff$, we have
\begin{equation*}
    |\eta|^{\frac{1}{2}} = C_\ff \int_\RR \frac{1 - e^{2 \pi i z \eta}}{|z|^{\frac{3}{2}}} \,\md z\;,
\end{equation*}
and hence
\begin{equation*}
    \widehat{|\cdot|^{\frac{1}{2}}}(y) = C_\ff \int_{\RR} \frac{1}{|z|^{\frac{3}{2}}} \big( \delta(y) - \delta(y+z) \big) \,\md z
\end{equation*}
as a distribution in the variable $y$. We then deduce 
\begin{equation*}
    \big( \sqrt{\fD} \varphi \big)(x) = (\widehat{|\cdot|^{\frac{1}{2}}} * \varphi)(x) = C_\ff \int_\RR \frac{\varphi(x) - \varphi(x+z)}{|z|^{\frac{3}{2}}}\,\md z\;,
\end{equation*}
where we used that $|\cdot|$ is an even function so that its Fourier transform and inversion coincide. Flipping $z \mapsto -z$ does not change the value of the integral, thus concluding the proof. 
\end{proof}

\begin{lem}\label{lem:leibniz_2}
    For $\varphi,\psi\in\sS$, we have the identity
    \begin{align*}
        \big( \sqrt{\fD}(\varphi \psi) \big)(x)               =\big(\sqrt{\fD}\varphi\big)(x) \, \psi(x) + C_\ff \int_\RR \frac{\varphi(x-z) \, \big(\psi(x)-\psi(x-z)\big)}{|z|^\frac{3}{2}}\,\md z\;.
    \end{align*}
\end{lem}
\begin{proof}
    The claim follows directly from applying Lemma~\ref{lem:leibniz_1} to $\varphi \psi$ and $\varphi$ respectively and taking the difference. 
\end{proof}

\begin{lem}\label{lem:leibniz_3}
    For $\varphi,\psi\in\sS(\RR^2)$, we have the identity
    \begin{align*}
        \sqrt{\fD}^{\otimes 2}(\varphi\psi)(x_1,x_2)      =&\sqrt{\fD}^{\otimes 2}\varphi(x_1,x_2)\psi(x_1,x_2)\\
        &+C_\ff\int_\RR\frac{\sqrt{\fD_{x_1}}\varphi(x_1,x_2-z)\big(\psi(x_1,x_2)-\psi(x_1,x_2-z)\big)}{|z|^{\frac{3}{2}}}\,\md z\\
        &+C_\ff\int_\RR\frac{\sqrt{\fD_{x_2}}\varphi(x_1-z,x_2)\big(\psi(x_1,x_2)-\psi(x_1-z,x_2)\big)}{|z|^{\frac{3}{2}}}\,\md z\\
        &+C_\ff^2\iint\limits_{\RR^2} \frac{\varphi(x_1-z_1,x_2-z_2)A}{|z_1z_2|^\frac{3}{2}}\,\md z_1 \,\md z_2\;,
    \end{align*}
    where
    \begin{align*}
        A=\psi(x_1,x_2)-\psi(x_1-z_1,x_2)-\psi(x_1,x_2-z_2)+\psi(x_1-z_1,x_2-z_2)\;.
    \end{align*}
\end{lem}
\begin{proof}
    The claim follows directly from applying Lemma~\ref{lem:leibniz_2} twice.
\end{proof}

\begin{rmk}
    All the above identities can be extended to locally integrable functions with polynomial growth in standard ways, with the right hand sides interpreted as Schwartz distributions. 
\end{rmk}

\begin{lem}\label{lem:leibniz_4}
Suppose $\varphi: \RR^2 \rightarrow \RR$ is locally integrable, and there exists $\lambda>0$ such that
\begin{equation*}
    |\sqrt{\fD_{y_1}}\varphi|(y_1,y_2)\lesssim \bracket{y_1}^{-\frac{3}{2}}e^{-\lambda|y_2|}\;, \quad |\sqrt{\fD_{y_2}}\varphi|(y_1,y_2)\lesssim \bracket{y_2}^{-\frac{3}{2}}e^{-\lambda|y_1|}\;.
\end{equation*}
Suppose furthermore that
\begin{equation*}
    |\sqrt{\fD}^{\otimes 2}\varphi|(y,y)\lesssim \bracket{y}^{-3}\;.
\end{equation*}
Then we have
\begin{equ}
    \Big|\sqrt{\fD}^{\otimes 2}\Big((y_1+y_2)\varphi\Big)\Big|(y,y)\lesssim \bracket{y}^{-2}\;.
\end{equ}
The proportionality constant depends on $\varphi$ only through the proportionality constants in the assumptions of the lemma. 
\end{lem}
\begin{proof}
By Lemma~\ref{lem:leibniz_3}, we have
\begin{align*}
    \sqrt{\fD}^{\otimes 2}\Big((y_1+y_2) &\varphi\Big)(y,y)      =2\sqrt{\fD}^{\otimes 2}\varphi(y,y)y\\
    &+C_\ff\int_\RR\frac{\sqrt{\fD_{y_1}}\varphi(y,y-z)z}{|z|^{\frac{3}{2}}}\,\md z+C_\ff\int_\RR\frac{\sqrt{\fD_{y_2}}\varphi(y-z,y)z}{|z|^{\frac{3}{2}}}\,\md z\;.
\end{align*}
Then the conclusion follows from the assumptions and Lemma~\ref{lem:techninal_1}.
\end{proof}

\endappendix

\bibliographystyle{Martin}
\bibliography{Refs}

\end{document}